\documentclass{amsbook}
\usepackage{amsmath}
\usepackage{latexsym}
\usepackage{amscd}
\usepackage{eufrak}
\usepackage{mathrsfs}
\usepackage[english]{babel}

\DeclareMathAlphabet{\mathpzc}{OT1}{pzc}{m}{it}
\newtheorem{theorem}[equation]{Theorem}
\newtheorem{theorem-definition}[equation]{Theorem-Definition}
\newtheorem{lemma-definition}[equation]{Lemma-Definition}
\newtheorem{definition-prop}[equation]{Proposition-Definition}
\newtheorem{corollary}[equation]{Corollary}
\newtheorem{prop}[equation]{Proposition}
\newtheorem{lemma}[equation]{Lemma}
\newtheorem{cor}[equation]{Corollary}
\newtheorem{definition}[equation]{Definition}

\newtheorem{conjecture}[equation]{Conjecture}

\theoremstyle{definition}
\newtheorem{example}[equation]{Example}

\newtheorem{openquestion}[subsection]{Question}
\newtheorem*{question*}{Question}
\newtheorem{remark}[equation]{Remark}

\newcommand{\LL}{\ensuremath{\mathbb{L}}}

\newcommand{\N}{\ensuremath{\mathbb{N}}}
\newcommand{\Z}{\ensuremath{\mathbb{Z}}}
\newcommand{\Q}{\ensuremath{\mathbb{Q}}}

\newcommand{\R}{\ensuremath{\mathbb{R}}}
\newcommand{\C}{\ensuremath{\mathbb{C}}}

\newcommand{\A}{\ensuremath{\mathbb{A}}}

\newcommand{\NN}{\ensuremath{\mathbb{N}}}

\newcommand{\X}{\ensuremath{\mathscr{X}}}

\newcommand{\cX}{\ensuremath{\mathscr{X}}}
\newcommand{\mX}{\ensuremath{\mathfrak{X}}}
\newcommand{\mY}{\ensuremath{\mathfrak{Y}}}

\newcommand{\mH}{\ensuremath{\mathfrak{H}}}

\newcommand{\cA}{\ensuremath{\mathscr{A}}}
\newcommand{\cC}{\ensuremath{\mathscr{C}}}
\newcommand{\cF}{\ensuremath{\mathscr{F}}}
\newcommand{\cG}{\ensuremath{\mathscr{G}}}
\newcommand{\cH}{\ensuremath{\mathscr{H}}}

\newcommand{\cT}{\ensuremath{\mathscr{T}}}

\renewcommand{\R}{\ensuremath{\mathbb{R}}}
\renewcommand{\C}{\ensuremath{\mathbb{C}}}

\renewcommand{\A}{\ensuremath{\mathbb{A}}}
\renewcommand{\X}{\ensuremath{\mathfrak{X}}}
\renewcommand{\mY}{\ensuremath{\mathfrak{Y}}}

\newcommand{\mU}{\ensuremath{\mathfrak{U}}}
\newcommand{\mV}{\ensuremath{\mathfrak{V}}}

\renewcommand{\cA}{\ensuremath{\mathscr{A}}}

\renewcommand{\cH}{\ensuremath{\mathscr{H}}}

\renewcommand{\cF}{\ensuremath{\mathscr{F}}}

\newcommand{\Spec}{\ensuremath{\mathrm{Spec}\,}}
\newcommand{\Spf}{\ensuremath{\mathrm{Spf}\,}}
\newcommand{\Sp}{\ensuremath{\mathrm{Sp}\,}}

\newcommand{\Ima}{\mathrm{Im}}
\newcommand{\Lie}{\mathrm{Lie}}
\newcommand{\Pic}{\mathrm{Pic}}
\newcommand{\ord}{\mathrm{ord}}
\newcommand{\red}{\mathrm{red}}

\newcommand{\nor}{\mathrm{nor}}
\newcommand{\prin}{\mathrm{prin}}
\newcommand{\lcm}{\mathrm{lcm}}
\newcommand{\Jac}{\mathrm{Jac}}
\newcommand{\im}{\mathrm{im}}
\newcommand{\tame}{\mathrm{tame}}
\newcommand{\Var}{\mathrm{Var}}
\newcommand{\Gal}{\mathrm{Gal}}
\newcommand{\Hom}{\mathrm{Hom}}
\newcommand{\coker}{\mathrm{coker}}
\newcommand{\an}{\mathrm{an}}
\newcommand{\sm}{\mathrm{sm}}

\newcommand{\Gr}{\mathrm{Gr}}
\newcommand{\Art}{\mathrm{Art}}
\newcommand{\Sw}{\mathrm{Sw}}
\newcommand{\pot}{\mathrm{pot}}
\newcommand{\length}{\mathrm{length}}

\newcommand{\spl}{\mathrm{spl}}
\newcommand{\free}{\mathrm{free}}
\newcommand{\tors}{\mathrm{tors}}
\newcommand{\ab}{\mathrm{ab}}
\newcommand{\tor}{\mathrm{tor}}
\newcommand{\qc}{\mathrm{qc}}

\newcommand{\rank}{\mathrm{rank}}
\newcommand{\tr}{\mathrm{tr}}
\newcommand{\Trace}{\mathrm{Trace}}

\def\trank#1{t(#1)}
\def\ttame#1{t_{\tame}(#1)}
\def\tpot#1{t_{\pot}(#1)}

\def\Comp#1{\Phi(#1)}
\def \degree{\colon\!}
\def\dgr#1#2{\mbox{$[#1\degree #2]$}}
\def\CompS#1{S^{\Phi}_{#1}}

\numberwithin{equation}{subsection}

\newcommand{\sss}{\vspace{5pt} \subsubsection*{ }\refstepcounter{equation}{\hspace{-24pt}{\bfseries(\theequation)}\ }}

\begin{document}
 \frontmatter
\title{N\'eron models and base change}
\author{Lars Halvard Halle}
\address{Matematisk Institutt\\
Universitetet i Oslo\\Postboks 1053
\\
Blindern\\0316 Oslo\\ Norway} \email{larshhal@math.uio.no}

\author[Johannes Nicaise]{Johannes Nicaise}
\address{KULeuven\\
Department of Mathematics\\ Celestijnenlaan 200B\\3001 Heverlee \\
Belgium} \email{johannes.nicaise@wis.kuleuven.be}



\begin{abstract}
We study various aspects of the behaviour of N\'eron models of
semi-abelian varieties under finite extensions of the base field,
with a special emphasis on wildly ramified Jacobians. In Part 1,
we analyze the behaviour of the component groups of the N\'eron
models, and we prove rationality results for a certain generating
series encoding their orders. In Part 2, we discuss Chai's base
change conductor and Edixhoven's filtration, and their relation to
the Artin conductor. All of these results are applied in Part 3 to
the study of motivic zeta functions of semi-abelian varieties.
Part 4 contains some intriguing open problems and directions for
further research. The main tools in this work are non-archimedean
uniformization and a detailed analysis of the behaviour of regular
models of curves under base change.
\end{abstract}
 \maketitle

\tableofcontents

\chapter{Introduction}

\section{Motivation and background}
\subsection{N\'eron models of abelian varieties}
Let $R$ be a complete discrete valuation ring with
quotient field $K$ and algebraically closed residue field $k$. We
fix a separable closure $K^s$ of $K$.
  Let $A$
be an abelian $K$-variety.  There exists a canonical ``best''
extension of $A$ to a smooth commutative group scheme
$\mathscr{A}$ over $R$, known as the N\'eron model of $A$. Its
distinguishing feature is the so-called N\'eron mapping property,
which says that if $Z$ is a smooth $R$-scheme, then any morphism
of $K$-schemes $ Z \times_R K \to A $ extends uniquely to an
$R$-morphism $ Z \to \mathscr{A}$. In particular, $\mathscr A$ is
the minimal smooth $R$-model of $A$.

The existence of these models was proved by A.~N\'eron \cite{AN},
and they have since become an invaluable tool for studying
arithmetic properties of abelian varieties. There are also
numerous geometrical applications, such as the study of
compactifications of moduli spaces of principally polarized
abelian varieties. For a detailed scheme theoretical account of
the construction and basic properties of N\'eron models, we refer
to the excellent textbook \cite{neron}.

Taking the special fiber of the N\'eron model allows one to define
 a canonical reduction
$$\mathscr{A}_k := \mathscr{A} \times_R k $$
of $A$. This is a smooth commutative group scheme over $k$, but
the geometric structure of $\mathscr{A}_k$ can be substantially
more complicated than that of $A$. In general one cannot expect
that the reduction of $A$ is again an abelian variety; the
reduction $\mathscr{A}_k$ may not be proper, or even connected.

Let $ \mathscr{A}_k^o $ be the  identity component of
$\mathscr{A}_k$. The group
$$\Comp A = \mathscr{A}_k/\mathscr{A}_k^o$$
of connected components of  $\mathscr{A}_k$ is a finite abelian
group, known as the \emph{N\'eron component group} of $A$.
 The
identity component $ \mathscr{A}_k^o $ is canonically an extension
of an abelian $k$-variety $B$ by a commutative smooth linear
algebraic $k$-group $G$; this is the so-called Chevalley
decomposition of $\mathscr{A}_k^o$. One can show that $G$ is the
product of a torus $T$ and a unipotent group $U$, whose dimensions
are called the toric and unipotent rank of $A$, respectively. In
the special case where $ U = \{ 0 \} $, one says that $A$ has
\emph{semi-abelian reduction}. In the literature, one often uses
the term {\em semistable reduction} instead, but we prefer to
avoid it because it might lead to
 confusion with semistable models of $K$-varieties.

The most important  structural result concerning N\'eron models is
Grothendieck's Semistable Reduction theorem \cite[IX.3.6]{sga7.1}.
It asserts the existence of a unique minimal finite extension $L$
of $K$ in $K^s$ such that $A\times_K L$ has semi-abelian
reduction. We say that $A$ is tamely ramified if $L$ is a tame
extension of $K$, and wildly ramified else.

The N\'eron model is functorial in $A$, but otherwise it does not
 have good functorial properties.
 A first problem is that it behaves poorly in exact sequences. Another important complication is that its formation
 does not commute with base change.  Let $K'$ be a finite
separable extension of $K$ and denote by $R'$ the integral closure
of $R$ in $K'$. We set $A'=A\times_K K'$, and we denote by
$\mathscr{A}'$ the N\'eron model of $A'$. Since $
\mathscr{A}\times_R R' $ is smooth, the N\'eron mapping property
of $\mathscr{A}'$ implies that there exists a unique morphism of
$R'$-group schemes
$$h:\mathscr{A}\times_R R'\to \mathscr{A}'$$
that extends the canonical isomorphism on generic fibers. If $A$
has semi-abelian reduction,  then $h$ is an open immersion;  in
particular, it is an isomorphism on identity components. This
property underlies the importance of Grothendieck's Semistable
Reduction Theorem. In the general case, it is quite hard to
describe the properties of the base change morphism $h$.

\subsection{Motivic zeta functions} The problem of describing  how
N\'eron models behave under
 base extensions lies at the heart of our work on
\emph{motivic zeta functions} of abelian varieties. In order to
explain this notion, we first need to introduce some notation. Let
$\N'$ be the set of positive integers not divisible by the
characteristic of $k$. For each $d \in \N'$ we denote by $K(d)$
the unique degree $d$ extension of $K$ in $K^s$. We put $ A(d) :=
A \times_K K(d) $ and we denote by $\mathscr{A}(d)$ the N\'eron
model of $A(d)$. In \cite{HaNi}, we defined the motivic zeta
function $ Z_A(T) $ as
$$ Z_A(T) = \sum_{d \in \N'} [ \mathscr{A}(d)_k ] \LL^{\mathrm{ord}_A(d)} T^d \in K_0(\Var_k)[[T]]. $$
Here $ K_0(\Var_k)$ denotes the Grothendieck ring of
$k$-varieties,  $\LL$ denotes the class $[\A^1_k]$ of the affine
line in $K_0(\Var_k)$, and $\mathrm{ord}_A$ is a function from
$\N'$ to $\N$ whose definition will be recalled in Chapter
\ref{chap-motzeta}.

One can roughly say that  $ Z_A(T) $ is the generating series for
the reductions $\mathscr{A}(d)_k$ (up to a certain scaling by
$\LL$), and thus encodes in a very precise way how the N\'eron
model of $A$ changes under tamely ramified extensions of $K$.
Moreover, one can view this object as an analog of Denef and
Loeser's motivic zeta function for complex hypersurface
singularities. This link will not be pursued further in this
monograph; we refer to \cite{HaNi-survey} for a detailed survey,
and an overview of the literature.

Since each of the connected components of $\mathscr{A}(d)_k$ is
isomorphic to the identity component $\mathscr{A}(d)^o_k$, we have
the relation
$$ [ \mathscr{A}(d)_k ] = |\Comp{A(d)}| \cdot [ \mathscr{A}(d)^o_k ] $$
 in $K_0(\Var_k)$ for every $ d \in \N' $. Because of this fact, many properties of $ Z_A(T) $,
 such as rationality and the nature of its poles, are closely linked to analogous properties of
 the \emph{N\'eron component series}
 $$ S^{\Phi}_A(T) = \sum_{d \in \N'} |\Comp{A(d)}| T^d \in \Z[[T]] $$
that we introduced in \cite{HaNi-comp} (there it was denoted
$S_{\phi}(A;T)$). This series measures how the number of N\'eron
components varies under tame extensions of $K$. We were able to
prove in \cite[6.5]{HaNi-comp} that  it is a rational function
when $A$ is tamely ramified or $A$ has potential multiplicative
reduction.
 This was a key ingredient of our proof that the motivic zeta function $ Z_A(T) $ is rational if $A$ is tamely ramified \cite{HaNi}.
  Moreover, setting $T=\LL^{-s}$ and viewing $s$ as a formal
  variable, we showed that $Z_A(\LL^{-s})$ has a unique pole.
  Interestingly, this pole coincides with an important arithmetic
  invariant of the abelian variety $A$: the base change conductor
  $c(A)$, which was introduced for tori by
Chai and Yu \cite{chai-yu} and for semi-abelian varieties by Chai
\cite{chai}.  It is a nonnegative rational number that measures
the defect of semi-abelian reduction of $A$.

\subsection{Aim} One of the main purposes of this monograph is to extend the
above mentioned results beyond the case of tamely ramified abelian
varieties. For one thing, it is natural to ask what can be said
without the tameness assumption. In general, it is not even clear
if $ Z_A(T) $ is rational if $A$ is wildly ramified. We establish
this fact for Jacobians in Chapter \ref{chap-motzeta}. It remains
a considerable challenge to understand  the motivic zeta function
of a wildly ramified abelian variety that is not a Jacobian, and
likewise, to understand the N\'eron component series of a wildly
ramified abelian variety that is not a Jacobian and does not have
potential multiplicative reduction.

Going in a different direction,  one can ask about the situation
for more general group schemes than abelian varieties. In
\cite{HaNi} we developed the general theory of motivic zeta
functions to also include, in particular, the class of
semi-abelian varieties.
 The existence of N\'eron models of semi-abelian $K$-varieties was
 proven in \cite{neron}. An important difference with the case of
 abelian varieties is that the N\'eron model of a semi-abelian
 variety will, in general, only be locally of finite type.
 In order to get a meaningful definition of
the motivic zeta function, one has to consider the maximal
quasi-compact open subgroup scheme of the N\'eron model. On the
level of component groups, this means that we consider the torsion
subgroup of the component group.

 While the methods we developed to study the behaviour of component groups of
 abelian varieties under base change can easily be extended to
  semi-abelian varieties, the torsion subgroup
 is a much more subtle invariant. An important complication is the
 lack of a geometric characterization of this object. A
 natural candidate would be the following: the N\'eron component
 group of the maximal split subtorus of a semi-abelian $K$-variety
 $G$ is a lattice of maximal rank inside the component group
 $\Comp{G}$, and one might hope to capture the torsion part by showing that this injection is
 split. We will show that this is usually not the case. For
 algebraic tori one can encompass this problem by passing to the dual
 torus, but no similar technique seems to exist for general semi-abelian
 varieties. Our approach consists in defining a suitable notion
 of non-archimedean uniformization for semi-abelian varieties.
 However, the uniformization space will no longer be an algebraic
 object, so that the existence of a (formal) N\'eron model is no longer
 guaranteed; instead, we will make a careful study of the properties of the
 sheaf-theoretic N\'eron model defined by Bosch and Xarles
 \cite{B-X}.

\subsection{A guiding principle}\label{subsec-guide}
Before we move  on to present an overview of the contents of this
monograph, it may be instructional to point out some of the main
themes and strategies. One of the basic ideas in our work on
component series and motivic zeta functions of abelian varieties
is the expectation that the N\'eron model of an abelian
$K$-variety $A$ changes ``as little as possible'' under a tame
extension $K'/K$ that is ``sufficiently orthogonal'' to the
minimal extension $L/K$ where $A$ acquires semi-abelian reduction.
 This principle was a crucial ingredient in establishing
rationality and determining the poles of the component series $
S^{\Phi}_A(T) $ and the motivic zeta function $ Z_A(T)$ in
 \cite{HaNi-comp} and \cite{HaNi}.

What do we mean by ``as little as possible''? It is unreasonable
 to require that the base change morphism $$h:\mathscr{A}\times_R R'\to \mathscr{A}'$$ is an isomorphism: even when $A$ has semi-abelian reduction, the number
  of connected components of the N\'eron model might still change (as we'll see, the rate of growth
  is determined by the toric rank $t(A)$ of $A$). It can be shown by
  elementary examples that if $A$ does not have semi-abelian reduction, we cannot even require $h$ to be an open
immersion. The best we can ask for is that the following two
properties are satisfied.
\begin{enumerate}
\item The number of components grows as if $\mathscr{A}$ had
semi-abelian reduction, i.e., if $K'/K $ is a tame extension of degree $d$, then the equality
$$ |\Comp{A \times_K K'}| =  d^{t(A)} \cdot |\Comp{A}|  $$
holds; \item The $k$-varieties $\mathscr{A}^o_k$ and
$(\mathscr{A}')^o_k$ define the same class in $K_0(\Var_k)$. By
\cite[3.1]{Nicaise}, this is equivalent to the property that
$\mathscr{A}^o_k$ and $(\mathscr{A}')^o_k$ have the same unipotent
and reductive ranks, and isomorphic abelian quotients in the
Chevalley decomposition.
\end{enumerate}

The meaning of ``sufficiently orthogonal'' is less clear. The most
natural guess is that the extensions $K'$ and $L$ should be
linearly disjoint over $K$,
  which is equivalent to asking that $\dgr{K'}{K}$ and $\dgr{L}{K}$ are coprime because the extension $K'/K$ is tame.
  We've shown in \cite{HaNi-comp} and \cite{HaNi} that this condition is indeed sufficient when
  $A$ is tamely ramified or  $A$ has potential  multiplicative
  reduction. However, we will see that, for the Jacobian $\mathrm{Jac}(C)$ of a $K$-curve $C$,
  the
  condition has to be modified: one needs to replace the
  degree of $L$ over $K$ by another invariant $e(C)$ that we call the stabilization index of $C$. It is defined in
  terms of the geometry of the $R$-models of the curve $C$. For
  general wildly ramified abelian $K$-varieties, it is not even
  clear if a suitable notion of orthogonality exists; we will come
  back to this problem in Part \ref{sec-ques}.


\section{Content of this monograph}

\subsection{Overview of the chapters}
 We'll now give an overview of the chapters of this text. A
 brief summary can also be found at the beginning of each chapter.

\medskip

 Chapter \ref{chap-preliminaries} contains preliminary
 material on group schemes, models of curves and related topics.
 We recall key results from the literature and prove
 some basic new properties that will be needed in the remainder of
 the text.

\medskip
 {\bf Part \ref{part-comp}} is the longest part of this monograph; it is devoted to the
 study of N\'eron component groups of semi-abelian $K$-varieties.
 One of the main objectives is to prove the rationality of the
 N\'eron component series and to determine the order of its pole
 at $T=1$; this is the pole that influences the behaviour of the
 motivic zeta function.
 In Chapter \ref{chap-jacobians},  we investigate wildly ramified
Jacobian varieties. Even in this situation, many of the methods we
used for tamely ramified abelian varieties are no longer
sufficient, or applicable. To point out just one problem, let us
mention that one can find examples already for elliptic curves
where properties (1) and (2) in Section \ref{subsec-guide} above
do not hold for tame extensions of degree coprime to $[L:K]$. The
definition of the stabilization index $e(C)$ and the study of its
basic properties will occupy an important part of the chapter.

The approach we take is to make use of the close relationship
between N\'eron models of the Jacobian $\mathrm{Jac}(C)$ of a
$K$-curve $C$ and regular models of $C$.
 More precisely, if we fix a regular, proper and flat $R$-model $\mathscr{C}$ of a
 smooth and proper $K$-curve $C$ of index one, then the relative Picard scheme
$\Pic^0_{\mathscr{C}/R}$ is canonically isomorphic to the identity
component $\mathscr{A}^o$ of the N\'eron model $\mathscr{A}$ of
the Jacobian $A=\Jac(C)$; this is a fundamental theorem of
M.~Raynaud \cite{raynaud}. Because of this link, many invariants
associated to $\mathscr{A}$ can also be computed on $\mathscr{C}$.
In particular, this is true for the component group. A key
technical step in Chapter \ref{chap-jacobians} is therefore to
provide a detailed description of the behaviour of regular models
of $K$-curves under tame extensions of $K$.

 Assume that the
special fiber $\mathscr{C}_k$ is a strict normal crossings
divisor. Then we'll call $\mathscr{C}$ an $sncd$-model of $C$. For
any $d \in \N'$, we denote by $\mathscr{C}_d$ the normalization of
$\mathscr{C} \times_R R(d)$ and by $\mathscr{C}(d)$ the minimal
desingularization of $\mathscr{C}_d$. We will show
 that $\mathscr{C}(d)$ is an $sncd$-model of $C
\times_K K(d)$ and we explain how its structure can be determined
from the structure of $\mathscr{C}$. In order to obtain these
results, we show that $\mathscr{C}_d$ has at most locally toric
singularities, which can be explicitly resolved using the results
in Kiraly's PhD thesis \cite{Kir}. We provide an appendix with
some
 basic results on the resolution of locally toric singularities, because
 this part of \cite{Kir} has not been published. Alternatively,
 one could use the language of logarithmic geometry \cite{kato}.


 Next, we define the stabilization index $e(C)$.
If $C$ is tamely ramified, it coincides with the degree of the
minimal extension over which $C$ acquires semi-stable reduction,
but this is not true in general. The main property of $e(C)$ is
that one can make a very precise comparison of the special fibers
of $\mathscr{C}$ and $\mathscr{C}(d)$ for any $d \in \N'$ prime to
$e(C)$; see Proposition \ref{prop-data}. The key point is that the
combinatorial structure of $\mathscr{C}(d)_k$ only depends on the
combinatorial structure of $\mathscr{C}_k$, and not on the
characteristic of $k$. This results allows us to extend certain
facts from \cite{HaNi-comp} from tamely ramified abelian varieties
to wildly ramified Jacobians, which is crucial for proving the
rationality of the component series. Our strategy is to use
Winters' theorem on the existence of curves with fixed reduction
type to reduce to the case where $K$ has equal characteristic
zero; then every abelian $K$-variety is tamely ramified. It is
surprisingly hard to give a direct combinatorial proof of our
results on the component series, even for tamely ramified
Jacobians.

\smallskip
 In Chapter \ref{chap-uniform}, we switch our attention to
 N\'eron component groups of semi-abelian varieties. As we've
 explained before, the main problem here is to understand the
 behaviour of the torsion part of the component group under finite
 extensions of the base field $K$.

  Our approach is based on two main tools. The first one is {\em
  non-archimedean uniformization} in the sense of \cite{B-X}. This
  theory shows that, in the rigid analytic category, one can write every abelian $K$-variety as a
  quotient of a semi-abelian $K$-variety $E$ by an \'etale lattice
  such that the abelian part of $E$ has
  potential good reduction; in this way, one can try to reduce the
  study of N\'eron component groups to the case of tori and
  abelian varieties with potential good reduction.
   We extend this construction to
  semi-abelian varieties. The main complication is that, in this
  case, one needs to replace $E$ by an analytic extension of a
  semi-abelian $K$-variety by a $K$-torus. This extension will
  usually not be algebraic, so that one cannot use the theory
  of N\'eron models for algebraic groups. Neither can one apply
  the theory of formal N\'eron models in \cite{bosch-neron},
  because the set of $K$-points on the extension can be
  unbounded.

 This brings us to the second tool in our approach: the
 sheaf-theoretic N\'eron model of Bosch and Xarles \cite{B-X}.
 They interpret the formation of the N\'eron model as a
 pushforward operation on abelian sheaves from the smooth rigid site of $\mathrm{Sp}
 K$ to the smooth formal site on $\Spf R$, and show how one can
 recover the component group from this sheaf-theoretic
 interpretation. An advantage of this approach, which we already
 exploited in \cite{HaNi-comp}, is that one can quantify the lack
 of exactness of the N\'eron functor by means of the derived functors of the
 pushforward. A second advantage is that one can associate a
 sheaf-theoretic N\'eron model and a component group to any smooth abelian sheaf on
 $\mathrm{Sp} K$, and in particular to any commutative rigid
 $K$-group, even when a geometric (formal) N\'eron model does not
 exist. This is particularly useful when considering
 non-archimedean uniformizations of semi-abelian varieties as
  above.

 With these tools at hand, we can control the behaviour of the
 torsion part of the component group of a semi-abelian variety
 under finite extensions of the base field $K$. Our main result is
 Theorem \ref{thm-compsersab}, which says, in particular,
 that the component series of a semi-abelian $K$-variety $G$ is
 rational if the abelian part $G_{\ab}$ of $G$ is tamely ramified or
 a Jacobian, and also if $G_{\ab}$ has potential multiplicative reduction.

\medskip
 In {\bf Part \ref{part-conductor}}, we study Chai's base change conductor
 of a semi-abelian variety, and some related invariants that were
 introduced by Edixhoven. These invariants play a key role in the
 determination of the poles of motivic zeta functions of
 semi-abelian varieties.
 In
Chapter \ref{chap-conductor} we define a new invariant for wildly
ramified abelian $K$-varieties, the \emph{tame base change
conductor} $c_{\mathrm{tame}}(A)$. This value is defined as the
sum of the jumps (counting multiplicities) in Edixhoven's
filtration on the special fiber of the N\'eron model of $A$.
Equivalently, one can also define $c_{\mathrm{tame}}(A)$ as a
limit of base change conductors with respect to all finite tame
extensions of $K$ in $K^s$. If $A$ is the Jacobian of a curve $C$,
 then we show that these jumps, and hence also $c_{\mathrm{tame}}(A)$,
only depend on the combinatorial data of the special fiber of the
minimal $sncd$-model of $C$.

By construction, it is clear  that $c_{\mathrm{tame}}(A) = c(A)$
if $A$ is tamely ramified. In the wild case, this equality may no
longer hold. For instance, we'll show that for an elliptic
$K$-curve $E$,
  the equality $ c(E) = c_{\mathrm{tame}}(E) $ holds if and
only if $E$ is tamely ramified, and we interpret the error term in
the wildly ramified case.

\smallskip
It is natural to ask how $c(A)$  is related to other arithmetic
invariants that one can associate to an abelian $K$-variety $A$.
 Recall that, for
algebraic tori, the main result of \cite{chai-yu} states that the
base change conductor of the torus equals one half of the Artin
conductor of its cocharacter module. The situation is more
delicate for abelian varieties: the base change conductor is not
invariant under isogeny, so that it certainly cannot be computed
from the Tate module of $A$ with $\Q_\ell$- coefficients. In
Chapter \ref{chap-artin}, we study the case where $A$ is the
Jacobian of a curve $C$.
 A promising
candidate to consider is the so called \emph{Artin conductor}
$\Art(C)$ of the curve $C$. This important invariant was
introduced by T.~Saito in \cite{saito-cond} and measures the
difference between the Euler characteristics of the geometric
generic and special fibers of the minimal regular model of $C$,
with a correction term coming from the Swan conductor of the
 $\ell$-adic cohomology of $C$. It is not reasonable to expect
 that $\Art(C)$ and $c(A)$ contain equivalent information, since
 $\Art(C)$ vanishes if and only if $C$ has good reduction while
 $c(A)$ vanishes if and only if $C$ has semistable reduction.
 However, one can hope to express $c(A)$ in terms of $\Art(C)$ and
 a suitable measure for the defect of potential good reduction of $C$.

 We will establish such an expression for curves of genus $1$ or
$2$. For elliptic curves, we get a clear picture: we will show
that $-12\cdot c(E)$ is equal to $\Art(E)$ if $E$ has potential
good reduction, and to  $\Art(E)-v_K(j(E))$ else. In the latter
case, one can view $-v_K(j(E))$ as a measure for the defect of
potential good reduction by noticing that it is equal to
$|\Comp{E\times_K L}|/[L:K]$, for any finite extension $L$ of $K$
such that $E\times_K L$ has semi-abelian reduction.

Our method for genus $2$ curves is somewhat indirect; in this
case, the curve is hyperelliptic, which allows one to define a
 minimal discriminant $\Delta_{\mathrm{min}} \in K$ and its
valuation $v_K(\Delta_{\mathrm{min}})$.  We will compare $
v_K(\Delta_{\mathrm{min}}) $ with the base change conductor
$c(A)$. The relationship to $\Art(C)$ then becomes clear due to
work of Liu \cite{liu-genre2} and Saito \cite{saito-cond}, where
$\Art(C)$ and $ v(\Delta_{\mathrm{min}})$ are compared. One should
point out that already for genus $2$, the invariants $c(A)$ and
$\Art(C)$ seem to ``diverge''; the base change conductor $c(A)$
rather seems to behave like $v(\Delta_{\mathrm{min}})$. However,
it is not clear how to generalize the definition of
$v(\Delta_{\mathrm{min}})$ to curves of higher genus.


\medskip
We finally  turn our focus to motivic zeta functions in {\bf Part
\ref{part-motzeta}.}
 We recall the definition of the motivic zeta function $Z_G(T)$ of a semi-abelian $K$-variety $G$ in Chapter \ref{chap-motzeta}.
  Our principal result, Theorem
\ref{theorem-zetamain}, extends the main theorem in \cite{HaNi} to
Jacobians and to tamely ramified semi-abelian varieties. More
precisely, let $G$ be a semi-abelian $K$-variety with abelian part
$G_{\mathrm{ab}}$, and assume either that $G$ is tamely ramified
or that $G$ is a Jacobian. Then we prove that $ Z_G(T) $ is a
rational function and that $ Z_G(\LL^{-s}) $ has a unique pole at
$ s = c_{\mathrm{tame}}(G) $ of order $ \ttame{G_{\ab}} + 1$.
 Here $\ttame{G_{\ab}}$ denotes the tame potential toric rank of
$G_{\ab}$, that is, the maximum of the toric ranks of
$G_{\ab}\times_K K'$ as $K'$ runs over the finite tame extensions
of $K$. We will also discuss similar results for Prym varieties.

\smallskip
In Chapter \ref{chap-cohomological}, we establish a cohomological
interpretation of the motivic zeta function $Z_G(T)$ by means of a
trace formula; this is similar to the Grothendieck-Lefschetz trace
formula for varieties over finite fields and the resulting
cohomological interpretation of the Hasse-Weil zeta function. Let
$p$ be the characteristic exponent of $k$ and let $\ell$ be a
prime different from $p$. Denote by $K^t$ the tame closure of $K$
in $K^s$ and choose a topological generator $\sigma$ of the tame
inertia group $\Gal(K^t/K)$. We denote by $P_G(T)$ the
characteristic polynomial of $\sigma$ on the tame $\ell$-adic
cohomology of $G$. One of our main results says that,
 if $G$ has maximal unipotent rank,
the prime-to-$p$ part of the order of the N\'eron component group
of $G$ is equal to $P_G(1)$. We deduce from this result that, for
every tamely ramified semi-abelian $K$-variety $G$, the
$\ell$-adic Euler characteristic of the special fiber of the
quasi-compact part of the N\'eron model of $G$ is equal to the
trace of $\sigma$ on the $\ell$-adic cohomology of $G$. This
yields a trace formula for the specialization of the motivic zeta
function $Z_G(T)$ with respect to the $\ell$-adic Euler
characteristic
$$\chi:K_0(\Var_k)\to \Z:[X]\mapsto \chi(X).$$ We also give an
alternative proof of this result for a Jacobian
$A=\mathrm{Jac}(C)$ using a computation on the tame nearby cycles
on an $sncd$-model $\mathscr C$ for the curve $C$. On the way, we
recover a formula of Lorenzini for the characteristic polynomial
$P_A(T)$ in terms of the geometry of $\mathscr C_k$.

\if false
 An important  part of our work in \cite{HaNi} was to
establish a link between the base change conductor of a tamely
ramified abelian variety $A$ and the eigenvalues for the monodromy
action on the tame cohomology groups of $A$. This link still holds
for tamely ramified abelian varieties, but breaks down in the
wildly ramified case. Indeed, the tame cohomology is zero already
for elliptic curves, and it is in no way obvious if there is a
chomological interpretation of the poles of the zeta function.

Even after all  our efforts in this monograph, it is still
somewhat unclear what properties one might expect the motivic zeta
function to have in the wildly ramified case, except for
Jacobians. The obvious next case to investigate after Jacobians
are Prym varieties, which we consider in ???. The general
philosophy of Prym varieties is that they are quite close to
Jacobians, but from the viewpoint of motivic zeta functions this
is only partly true. On the positive side, we are able to show
rationality of the jumps in Edixhoven's filtration for Prym
varieties (ref). What turns out to be a more serious challenge is
to control the behaviour of the reduction data of Prym varieties
after ramified base extensions. We are only able to do this under
the assumption of potential multiplicative reduction. \fi

\medskip
To conclude, we formulate
  some interesting open problems and directions for further research in
 {\bf Part \ref{sec-ques}}.


\subsection{Notation}
\sss  We denote by $K$ a complete discretely valued field with
ring of integers $R$  and residue field $k$. We denote by
$\frak{m}$ the maximal ideal of $R$, and by
$$v_K:K\to \Z\cup \{\infty\}$$ the normalized discrete valuation
on $K$.
 We assume that $k$ is
separably closed and we denote by $p$ the characteristic exponent
of $k$. The letter $\ell$ will always denote a prime different
from $p$. For most of the results in this memoir, the conditions
that $R$ is complete and $k$ separably closed are not serious
restrictions: since the formation of N\'eron models commutes with
extensions of ramification index one \cite[10.1.3]{neron}, one can
simply pass to the completion of a strict henselization of $R$.
For some of the results we present, we need to assume that $k$ is
perfect (and thus algebraically closed); this will be clearly
indicated at the beginning of the chapter or section.

\sss  We fix a uniformizing parameter $\pi$ in $R$ and a separable
closure $K^s$ of $K$. We denote by $K^t$ the tame closure of $K$
in $K^s$, and by $P=\Gal(K^s/K^t)$ the wild inertia subgroup of
the inertia group $I=\Gal(K^s/K)$ of $K$.

\sss We denote by $\mathbb{N}'$ the  set of  positive integers
prime to $p$. For every $d \in \mathbb{N}'$, there exists a unique
degree $d$ extension of $K$ inside $K^t$, which we denote by
$K(d)$. It is obtained by joining a $d$-th root of $\pi$ to $K$.
The integral closure $R(d)$ of $R$ in $K(d)$ is again a complete
discrete valuation ring.
 For every $K$-scheme $Y$, we put $ Y(d) = Y \times_K K(d) $.

\sss
  We'll
consider the special fiber functor
$$(\cdot)_k:(\mathrm{Schemes}/R)\to (\mathrm{Schemes}/k):\mathscr{X}\mapsto \mathscr{X}_k = \mathscr{X} \times_R k$$
as well as the generic fiber functor
$$(\cdot)_K:(\mathrm{Schemes}/R)\to (\mathrm{Schemes}/K):\mathscr{X}\mapsto \mathscr{X}_K = \mathscr{X} \times_R K.$$
 We will use the same notations for the special fiber functor from
 the category of formal $R$-schemes locally topologically of
 finite type to the category of $k$-schemes locally of finite type, resp.~Raynaud's generic fiber functor from the category of formal
$R$-schemes locally topologically of
 finite type to the category of quasi-separated rigid
 $K$-varieties.

\sss  We denote by $(\cdot)^{\an}$ the rigid analytic GAGA functor
from the category of $K$-schemes of finite type to the category of
quasi-separated rigid $K$-varieties. From now on, all rigid
$K$-varieties will tacitly be assumed to be quasi-separated.

\sss For every separated $k$-scheme of finite type $X$, we denote
by $\chi(X)$ its $\ell$-adic Euler characteristic
$$\chi(X)=\sum_{i\geq 0} (-1)^i\mathrm{dim}\,H^i_c(X,\Q_\ell).$$
 The value $\chi(X)$
does not depend on $\ell$.

\sss For every field $F$,  an algebraic $F$-group is a group
scheme of finite type over $F$. A semi-abelian $F$-variety is a
smooth commutative algebraic $F$-group that is an extension of an
abelian $F$-variety by an algebraic $F$-torus.

\sss For every finitely generated abelian group $\Phi$, we denote
by $\Phi_{\tors}$ the subgroup of torsion elements and by
$\Phi_{\free}$ the free abelian  group $\Phi/\Phi_{\tors}$. These
operations define functors from the category of finitely generated
abelian groups to the category of finite abelian groups and the
category of free abelian  groups of finite rank, respectively.

\section{Acknowledgements}
The second author was supported by the Fund for Scientific
Research - Flanders (G.0415.10).

\mainmatter

\chapter{Preliminaries}\label{chap-preliminaries}
\section{Galois theory of $K$} In this section, we assume that $k$
is algebraically closed.
\subsection{The Artin conductor}
\sss \label{sss-kummer} Let $K'$ be a tame finite extension of $K$
of degree $d$.
 Denote by $R'$ the
integral closure of $R$ in $K'$, and by $\frak{m}'$ the maximal
ideal of $R'$. Then the quotient $\frak{m}'/(\frak{m}')^2$ is a
rank one vector space over $k$, and the left action of
$\Gamma=\Gal(K'/K)$ on $K'$ induces an injective group morphism
$$\Gamma\to \mathrm{Aut}_k(\frak{m}'/(\frak{m}')^2)=k^*$$ whose image
is the group $\mu_d(k)$ of $d$-th roots of unity in $k$. Thus we
can identify $\Gal(K'/K)$ and $\mu_d(k)$ in a canonical way.


\sss \label{sss-ramgr} If $L$ is a finite Galois extension of $K$
with Galois group $\Gamma=\Gal(L/K)$, then the lower numbering
ramification filtration $(\Gamma_{i})_{i\geq -1}$ on $\Gamma$ is
defined in \cite[IV\S1]{serre-corpslocaux}. We quickly recall its
definition. Denote by $R_L$ the valuation ring of $L$, by
$\frak{m}_L$ its maximal ideal and by $\pi_L$ a uniformizer of
$L$.
 For every $i\geq -1$, the subgroup $\Gamma_i$ consists of all elements $\sigma$ in $\Gamma$ such that
 $$v_L(\pi_L-\sigma(\pi_L))\geq i+1.$$
   Note that $\sigma$ belongs to $\Gamma_i$
  if and only if it acts trivially on $R_L/\frak{m}^{i'+1}_L$,
  with $i'$ the smallest integer
 greater than or equal to $i$. Since $k$ is algebraically closed,
 we have $\Gamma_i=\Gamma$ for all $i\leq 0$.

\sss If $K'$ is a finite extension of $K$ in $L$, and if we denote
by $H$ the Galois group $\Gal(L/K')$, then $H_i=\Gamma_i\cap H$
for all $i$ by Proposition 2 in \cite[IV\S1]{serre-corpslocaux}.
If $K'$ is Galois over $K$ and if we set $\Gamma'=\Gal(K'/K)$,
then the image of $\Gamma_i$ under the projection morphism
$\Gamma\to \Gamma'$ equals $\Gamma_{\varphi_{L/K'}(i)}$ for every
$i\geq -1$, where $\varphi_{L/K'}$ is the Herbrand function of the
extension $L/K'$
 \cite[IV\S3]{serre-corpslocaux}. This result is known as
 Herbrand's theorem.

\sss Let $V$ be a finite dimensional vector space over a field
$F$, and consider an action of the Galois group $\Gal(K^s/K)$ on
$V$.
 We assume that there exists a
finite extension $L$ of $K$ in $K^s$ such that the Galois action
on $V$ factors through the finite quotient $\Gamma=\Gal(L/K)$ of
$\Gal(K^s/K)$.   The (exponent of the) Artin conductor $\Art(V)$
of $V$ is defined by
$$\Art(V)=\sum_{i\in \N}\frac{1}{[\Gamma:\Gamma_i]}\dim(V/V^{\Gamma_i})\quad \in \Q.$$ It can be
written as the sum of the so-called tame part
$$\dim(V)-\dim(V^{\Gamma})$$
and the wild part
$$\mathrm{Sw}(V)= \sum_{i\in
\Z_{>0}}\frac{1}{[\Gamma:\Gamma_i]}\dim(V/V^{\Gamma_i})$$ which is
also called the Swan conductor of $V$. The Artin and Swan
conductor do not depend on the choice of $L$. They measure the
ramification, resp.~wild ramification, of the Galois
representation $V$, and vanish if and only if $\Gal(K^s/K)$,
resp.~$P$, act trivially on $V$.

\sss If $V$ is a finite dimensional $\Q_\ell$-vector space endowed
with a continuous and quasi-unipotent action of $\Gal(K^s/K)$,
then we define the Artin and Swan conductor of $V$ as the
corresponding conductor of the semi-simplification
$V^{\mathrm{ss}}$ of the Galois representation $V$. Recall that
the Galois action on the $\ell$-adic cohomology spaces with
compact supports of an algebraic $K$-variety is always
quasi-unipotent, by the Monodromy Theorem
\cite[Exp.~I,\,1.3]{sga7.1}.

\subsection{Isolating the wild part of the conductor}
 The following easy corollary of Herbrand's theorem will be quite
 useful for our purposes.

\begin{prop}\label{prop-ramfilt}
Let $L$ be a finite Galois extension of $K$ in $K^s$, and let $K'$
be a finite tame extension of $K$ of degree $d$ in $K^s$. We set
$L'=K'L\subset K^s$, $\Gamma=\Gal(L/K)$ and $\Gamma'=\Gal(L'/K')$,
and we denote by $e$ the greatest common divisor of $d$ and
$[L:K]$.
 Then for every $i>0$, the image of $\Gamma'_i$ under the natural
 morphism $\Gamma'\to \Gamma$ equals
$\Gamma_{ie/d}$.
\end{prop}
\begin{proof}
This follows immediately from Herbrand's theorem,  since $L'/L$ is
a tame extension of degree $d/e$ so that $\varphi_{L'/L}(i)=ie/d$
for all $i\geq 0$.
\end{proof}

\sss Proposition \ref{prop-ramfilt} has the following interesting
consequence: even though base change to a finite tame extension
$K'$ of $K$ of degree prime to $[L:K]$ does not change the Galois
group $\Gamma=\Gal(L/K)$, it does alter the ramification
filtration by pushing the higher ramification groups towards
infinity. This allows us to isolate the wild part of the Artin
conductor of a Galois representation of $K$ in the following way.

\begin{prop}\label{prop-swan}
Let $V$ be a finite dimensional vector space over $\Q$ or
$\Q_\ell$, endowed with a continuous action of $\Gal(K^s/K)$. We
denote by $\Art(V)$ and $\Sw(V)$ the Artin conductor, resp.~Swan
conductor, of the Galois representation $V$. For every $d\in \N'$,
we denote by $V(d)$ the restriction of $V$ to $\Gal(K^s/K(d))$.
 If we order the set $\N'$ by the divisibility relation, then
 $$\Sw(V)=\lim_{\stackrel{\longrightarrow}{d\in
 \N'}}\frac{\Art(V(d))}{d}.$$
\end{prop}
\begin{proof} Proposition \ref{prop-ramfilt}
 implies that $\Sw(V(d))=d\cdot \Sw(V)$ for all $d\in \N'$, while
 the tame part of $\Art(V(d))$ is bounded by the dimension of $V$.
\end{proof}

\pagebreak
\section{Subtori of algebraic groups}
\subsection{Maximal subtori}
\sss Let $F$ be a field, and let $G$ be a smooth commutative
algebraic group over $F$. It follows from \cite[3.2]{HaNi-comp}
that the
 abelian presheaf
 $$T\mapsto \Hom_F(T,G)$$
 on the category of $F$-tori
is representable by an $F$-torus $G_{\tor}$, and that the
tautological morphism $G_{\tor}\to G$ is a closed immersion. We
call $G_{\tor}$ the maximal subtorus of $G$.
 It follows from \cite[3.1]{HaNi-comp} and the subsequent remark that, for every field extension $F'$ of
$F$, the torus $G_{\tor}\times_F F'$ is the maximal subtorus of
$G\times_F F'$.
 The
dimension of $G_{\tor}$ is called the {\em reductive rank} of $G$
and denoted by $\rho(G)$. By \cite[3.4]{HaNi-comp}, the reductive
rank of $G/G_{\tor}$ is zero. The algebraic group $G$ is a
semi-abelian $F$-variety if and only if $G/G_{\tor}$ is an abelian
variety; in that case, we denote this quotient by $G_{\ab}$.

\sss Every $F$-torus $T$ has a unique maximal split subtorus
$T_{\spl}$ \cite[3.5]{HaNi-comp}. The cocharacter module of
$T_{\spl}$ is the submodule of Galois-invariant elements of the
cocharacter module of $T$.
 It follows that every smooth commutative algebraic $F$-group
$G$ has a unique maximal split subtorus, which we denote by
$G_{\spl}$.  If $T'$ is a split $F$-torus, then every morphism of
algebraic groups $T'\to G$ factors through $G_{\spl}$.
 We call
the dimension of $G_{\spl}$ the {\em split reductive rank} of $G$,
and denote it by $\rho_{\spl}(G)$. By the remark after
\cite[3.6]{HaNi-comp}, the split reductive rank of $G/G_{\spl}$ is
zero.

\sss \label{sss-anisotropic} An $F$-torus $T$ is called {\em
anisotropic} if $T_{\spl}$ is trivial, i.e., if its only
Galois-invariant character is the trivial character. By the
duality between tori and their character modules, all morphisms
from an anisotropic torus to a split torus or from a split torus
to an anisotropic torus are trivial.

\subsection{Basic properties of the reductive rank}

\begin{lemma}\label{lemm-redisogeny}
The reductive rank of a connected smooth commutative algebraic
$F$-group is invariant under isogeny.
\end{lemma}
\begin{proof}
Let $f:G\to H$ be an isogeny of connected smooth commutative
algebraic $F$-groups, and denote by $d$ the degree of $f$. Since
the kernel of $f$ is killed by $d$, there exists a morphism of
algebraic $F$-groups $g:H\to G$ such that $g\circ f=d_G$, where we
denote by $d_G$ the multiplication by $d$ on $G$. Then $$f\circ g
\circ f=f\circ d_G=d_H\circ f$$ so that $f\circ g=d_H$ because $f$
is faithfully flat.

The morphisms $f$ and $g$ induce morphisms $f_{\tor}:G_{\tor}\to
H_{\tor}$ and $g_{\tor}:H_{\tor}\to G_{\tor}$ such that
$f_{\tor}\circ g_{\tor}$ and $g_{\tor}\circ f_{\tor}$ are given by
multiplication by $d$. In particular, $f_{\tor}$ and $g_{\tor}$
are isogenies, so that $\dim\,G_{\tor}=\dim\,H_{\tor}$.
\end{proof}

\begin{lemma}\label{lemm-isog}
Let $\ell$ be a prime invertible in $F$, and let
$$f:G\to H$$ be a morphism of semi-abelian $F$-varieties such that
the induced morphism of $\ell$-adic Tate modules
$$T_\ell G\to T_\ell H$$ is an isomorphism. Then $f$ is an
isogeny.
\end{lemma}
\begin{proof}
Let $F^a$ be an algebraic closure of $F$, and denote by $F^s$ the
separable closure of $F$ in $F^a$. Let $J$ be a connected smooth
commutative algebraic $F$-group. By definition,
$$T_\ell J=\lim_{\stackrel{\longleftarrow}{n}} (_{\ell^n}J(F^s))$$
where $_{\ell^n}J$ is the kernel of multiplication by $\ell^n$ on
$J$. Since $_{\ell^n}J$ is an \'etale $F$-scheme
\cite[XV.1.3]{sga3.2}, the map
$$_{\ell^n}J(F^s)\to \, _{\ell^n}J(F^a) $$ is a bijection. Thus, as a $\Z_\ell$-module without Galois structure,
 $T_\ell J$ is invariant under base change to $F^a$, so
that we may assume that $F$ is algebraically closed.

The reduction of the identity component of $\ker(f)$ is a
semi-abelian $F$-variety, by \cite[5.2]{HaNi}. It has trivial
$\ell$-adic Tate module, so that it must be trivial. It follows
that $\ker(f)$ is finite. Likewise, the image of $f$
 is a semi-abelian subvariety of $H$ with the same $\ell$-adic Tate module as
 $H$. Since the Tate module of a semi-abelian $F$-variety $J$ is a
 free $\Z_\ell$-module of rank
 $$\dim\,J_{\tor}+2\dim\,J_{\ab},$$ it follows that $f$ is
 surjective. Thus $f$ is an isogeny.
\end{proof}

\begin{lemma}\label{lemm-add}
For every exact sequence
$$0\rightarrow G_1\rightarrow G_2\rightarrow G_3\rightarrow 0$$ of
 connected smooth commutative algebraic $F$-groups, we have
$$\rho_{\spl}(G_2)=\rho_{\spl}(G_1)+\rho_{\spl}(G_3).$$
\end{lemma}
\begin{proof}
 Base change to the perfect closure of $F$ does not affect the
 split reductive rank of an algebraic $F$-group. Thus we may
 assume that $F$ is perfect.
For $i=1,2,3$, we denote by $T_i$ the maximal split subtorus
$(G_i)_{\spl}$ of $G_i$.

\textit{Case 1: $G_1$ is a split torus.}
 The morphism $T_2\to G_3$ factors through $T_3$.
Replacing $G_3$ by $T_3$ and $G_2$ by the semi-abelian $F$-variety
$G_2\times_{G_3}T_3$, we may assume that $G_3$ is a split
$F$-torus. Then $G_2$ is an extension of two split $F$-tori, and
thus it is again a split $F$-torus (it is diagonalizable
\cite[IX.8.2]{sga3.2}, smooth \cite[VI$_B$.9.2]{sga3.1} and
connected). Therefore, $\rho_{\spl}(G_i)=\dim(G_i)$ for $i=1,2,3$
and the result is clear.

\textit{Case 2: General case.} Dividing $G_1$ and $G_2$ by $T_1$
and applying Case 1 to the exact sequences
 $0\to T_1\to G_1\to G_1/T_1\to 0$ and $0\to T_1\to G_2\to G_2/T_1\to
 0$, we can reduce to the case where $T_1$ is trivial. Arguing as
 in Case 1, we may assume that $G_3$ is a split torus, so that $T_3=G_3$.

 The kernel of the morphism
$T_2\to G_3$ is diagonalizable \cite[IX.8.1]{sga3.2}. It is also a
closed subgroup of $G_1$. Since $T_1$ is trivial, the kernel of
$T_2\to G_3$ must be finite. Thus it suffices to show that $T_2\to
G_3$ is surjective. Denote by $H$ the schematic image of $T_2\to
G_3$. This is a closed subgroup of the split torus $G_3$. The
quotient
 $G_3/H$ is again a split $F$-torus (it is
 connected, smooth \cite[VI$_B$.9.2(xii)]{sga3.1} and
diagonalizable \cite[IX.8.1]{sga3.2}).

 One deduces from the Chevalley decomposition of $G_2$ that the quotient $G_2/T_2$ is an extension of an abelian
$F$-variety by the product of a unipotent $F$-group and an
anisotropic $F$-torus. Thus the morphism of $F$-groups $G_2/T_2\to
G_3/H$ must be trivial, because all morphisms from an anisotropic
torus, a unipotent group or an abelian variety to a split torus
are trivial; this follows from \eqref{sss-anisotropic} and
\cite[XVII.2.4]{sga3.2}, and the fact that every regular function
on an abelian variety is constant.
 On the other hand, the morphism $G_2/T_2\to
G_3/H$ is surjective by surjectivity of $G_2\to G_3$, so that
$G_3/H$ must be trivial, and $H=G_3$. Since the image of $T_2\to
G_3$ is closed \cite[VI$_B$.1.2]{sga3.1}, it follows that $T_2\to
G_3$ is surjective.
\end{proof}
\begin{cor}\label{cor-add}
For every exact sequence
$$0\rightarrow G_1\rightarrow G_2\rightarrow G_3\rightarrow 0$$ of
 connected smooth  commutative algebraic $F$-groups, we have
$$\rho(G_2)=\rho(G_1)+\rho(G_3).$$
\end{cor}
\begin{proof}
This follows at once from Lemma \ref{lemm-add}, since
$\rho(G)=\rho_{\spl}(G\times_F F^s)$ for every commutative
algebraic $F$-group $G$.
\end{proof}
We will also need the following variant of Corollary
\ref{cor-add}.
\begin{lemma}\label{lemm-add2}
Let $\ell$ be a prime invertible in $F$. Let
$$G_1\rightarrow G_2\rightarrow G_3$$
 be a complex of connected smooth commutative algebraic $F$-groups
 such that the sequence of Tate modules
$$0\to T_\ell G_1\to T_\ell G_2\to T_\ell G_3\to 0$$
 is exact. Then
 $$\rho(G_2)=\rho(G_1)+\rho(G_3).$$
\end{lemma}
\begin{proof}
As in the proof of Lemma \ref{lemm-isog}, we may suppose that $F$
is algebraically closed. If we denote by $U_i$ the unipotent radical
of $G_i$, then the morphism
$$T_\ell G_i\to T_\ell (G_i/U_i)$$ is an isomorphism, since multiplication by $\ell$ defines an automorphism of
$U_i$. Therefore, dividing $G_i$ by its unipotent radical, we may
assume that $G_i$ is a semi-abelian $F$-variety for $i=1,2,3$.

 Then $G_1$ is $\ell$-divisible, so that the
 sequence
 $$0\to T_\ell G_1\to T_\ell G_2\to T_\ell (G_1/G_2)\to 0$$ is
 exact. This means that
 $$T_\ell (G_1/G_2)\to T_\ell G_3$$ is an isomorphism. But
 $G_1/G_2$ is a semi-abelian $F$-variety, so that the morphism
 $$G_1/G_2\to G_3$$ is an isogeny by Lemma \ref{lemm-isog}, and
 $\rho(G_2/G_1)=\rho(G_3)$ by Lemma \ref{lemm-redisogeny}. Now the
 result follows from Corollary \ref{cor-add}.
\end{proof}

\pagebreak
\section{N\'eron models}
\subsection{The N\'eron model and the component group}
\sss A N\'eron $lft$-model of a smooth commutative algebraic
$K$-group $G$ is a separated smooth $R$-scheme $\mathscr{G}$,
endowed with an isomorphism $\mathscr{G}\times_R K\to G$, such
that for every smooth $R$-scheme $Z$, the natural map
$$\Hom_R(Z,\mathscr{G})\to \Hom_K(Z\times_R K,G)$$ is a bijection.
This universal property implies that a N\'eron $lft$-model is
unique up to unique isomorphism if it exists, so that we can
safely speak of {\em the} N\'eron model of $G$. It also entails
that the group law on $G$ extends uniquely to a commutative group
law on $\mathscr{G}$ that makes $\mathscr{G}$ into a separated
smooth group scheme over $R$. Moreover, the formation of N\'eron
models is functorial: if $f:G\to H$ is a morphism of smooth
commutative algebraic $K$-groups such that $G$ and $H$ have
N\'eron $lft$-models $\mathscr{G}$ and $\mathscr{H}$,
respectively, then the universal property implies that $f$ extends
uniquely to a morphism of group schemes $\mathscr{G}\to
\mathscr{H}$ over $R$.

\sss Every semi-abelian $K$-variety $G$ admits a N\'eron
$lft$-model, by \cite[10.2.2]{neron}. The constant $k$-group
scheme $\Comp{G}=\mathscr{G}_k/\mathscr{G}_k^o$ of connected
components of $\mathscr{G}_k$ is called the N\'eron component
group of $G$, or component group for short. We will identify it
with the abelian group $\Comp{G}(k)$.
 The specialization morphism
$$G(K)=\mathscr{G}(R)\to \mathscr{G}_k(k)$$ induces an
isomorphism $G(K)/\mathscr{G}^o(R)\to \Comp{G}$.

\sss \label{sss-qcneron}  The group $\Comp{G}$ is finitely
generated \cite[5.4]{HaNi}, and its rank is equal to the dimension
of the maximal split subtorus of $G$ \cite[4.11]{B-X}. In
particular, the N\'eron $lft$-model of $G$ is quasi-compact if and
only if $G$ does not contain a subgroup isomorphic to
$\mathbb{G}_{m,K}$. This happens, for instance, if $G$ is an
abelian $K$-variety.

\sss By \cite[7.2.1]{neron}  and  \cite[6.2]{bosch-neron}, the
formal $\frak{m}$-adic
 completion $\widehat{\mathscr{G}}$ of $\mathscr{G}$ is a formal
 N\'eron model of the rigid $K$-group $G^{\an}$ in the sense of
 \cite[1.1]{bosch-neron}.
 The special fibers of $\mathscr{G}$ and $\widehat{\mathscr{G}}$
 are canonically isomorphic. This means that we can use formal and rigid geometry to study component groups
  of semi-abelian $K$-varieties. In particular, we will make extensive use of {\em rigid uniformization} \cite[1.1]{B-X}.

\sss \label{sss-neronqc} From the fact that the torsion part of
$\Comp{G}$ is finite, one can easily deduce that $\mathscr{G}$ has
a unique maximal quasi-compact open subgroup scheme
$\mathscr{G}^{\qc}$ over $R$ \cite[3.6]{HaNi}, which we call the
N\'eron model of $G$. It is characterized by a universal property
in \cite[3.5]{HaNi}. The component group
$\mathscr{G}^{\qc}_k/(\mathscr{G}^{\qc}_k)^o$ is the torsion part
of the component group $\Comp{G}$.

\sss The notation that we use in this article is slightly
different from the one in \cite{HaNi}: there we denoted the
N\'eron $lft$-model by $\mathscr{G}^{lft}$ and the N\'eron model
by $\mathscr{G}$. In the present article, the notation introduced
above will be more convenient.

\subsection{The toric rank}\label{ss-torrank}
\sss \label{sss-ranks} Besides the component group, we can define
the following fundamental invariants of a semi-abelian $K$-variety
$G$.  Let $\mathscr{G}$ be the N\'eron $lft$-model of $G$. The
toric rank $t(G)$ of $G$ is the reductive rank of
$\mathscr{G}^o_k$, i.e., the dimension of the maximal subtorus $T$
of $\mathscr{G}^o_k$. If $G$ is a torus, then $t(G)$ is equal to
the dimension of the maximal split subtorus of $G$
\cite[3.13]{HaNi-comp}.

\sss \label{sss-unipotabrank} If $k$ is perfect, then we define
the abelian rank $a(G)$ and the unipotent rank $u(G)$ of $G$ as
the dimension of the abelian quotient $B$, resp. the unipotent
part $U$, in the Chevalley decomposition
$$0\to T\times_k U\to \mathscr{G}^o_k\to B\to 0$$
 of $\mathscr{G}^o_k$
\cite{conrad-chevalley}. Note that the sum $t(G)+a(G)+u(G)$ equals
the dimension of $G$.

\begin{prop}\label{prop-add}
 For every exact sequence of
semi-abelian $K$-varieties
$$0\to G_1\to G_2\to G_3\to 0$$ with $G_1$ a torus, we
have
$$t(G_2)=t(G_1)+t(G_3).$$
\end{prop}
\begin{proof}
 We
denote by $\mathcal{G}_i$ the N\'eron $lft$-model of $G_i$, for
$i=1,2,3$.
 For every prime
$\ell$ invertible in $k$, the sequence of Tate modules
\begin{equation}\label{eq-tatemod0}
0\to T_\ell G_1\to T_\ell G_2\to T_\ell G_3\to 0\end{equation} is
exact because the group $G_1(K^s)$ is $\ell$-divisible. Moreover,
if we denote by $X(G_1)$ the character module of $G_1$, then there
exists an $I$-equivariant isomorphism of $\Z_\ell$-modules
$$T_\ell G_1\cong X(G_1)^{\vee}\otimes_{\Z}\Z_\ell(1)$$
and thus an isomorphism
$$H^1(I,T_\ell G_1)\cong H^1(I,X(G_1)^{\vee})\otimes_{\Z}\Z_\ell.$$
 Thus, choosing the prime $\ell$ such that it does not divide the degree of the
 splitting field of $G_1$ over $K$, we may assume that
$$H^1(I,T_\ell G_1)=0.$$ Then, taking $I$-invariants in the sequence
\eqref{eq-tatemod0}, we get an exact sequence
$$0\to T_\ell G_1(K)\to T_\ell G_2(K)\to T_\ell G_3(K)\to
0$$ that we can identify with the sequence
\begin{equation*}\label{eq-tatemod2}
0\to T_\ell (\mathcal{G}_1)^o_k\to T_\ell (\mathcal{G}_2)^o_k\to
T_\ell (\mathcal{G}_3)^o_k \to 0\end{equation*} by the arguments
in \cite[IX.2.2.5]{sga7.1}. Now the result follows from Lemma
\ref{lemm-add2}.
\end{proof}
\begin{corollary}\label{cor-torrank}
For every semi-abelian $K$-variety $G$, we have
$$t(G)=t(G_\tor)+t(G_\ab).$$
\end{corollary}

\subsection{N\'eron models and base change} \sss \label{sss-nerbc} Let $G$ be a
semi-abelian $K$-variety, with N\'eron $lft$-model $\mathscr{G}$.
 Let $K'$ be a finite extension of $K$, with valuation ring $R'$,
 and denote by
 $\mathscr{G}'$ the N\'eron $lft$-model of $G'=G\times_K K'$.
By the universal property of the N\'eron $lft$-model, there exists
a unique morphism of $R'$-group schemes
$$h:\mathscr{G}\times_R R'\to \mathscr{G}'$$ that extends the
natural isomorphism between the generic fibers. It is not an
isomorphism, in general. The morphism $h$ induces a morphism of
component groups
\begin{equation}\label{eq-alpha}
\alpha_G:\Comp{G} \to \Comp{G'}.\end{equation}
 One of the principal aims of this monograph is to study the
properties of $\alpha_G$. Here we give an elementary example,
which we will need in some of the proofs in Part \ref{part-comp}.

\subsection{Example: the N\'eron $lft$-model of a split algebraic
torus} \label{sss-torcomp} \sss For every integer $n>0$, the
N\'eron $lft$-model of $\mathbb{G}^n_{m,K}$ is constructed by
gluing copies of $\mathbb{G}^n_{m,R}$ along their generic fibers.
For $n=1$, this procedure is described in \cite[10.1.5]{neron}.
The general case follows from the fact that the formation of
N\'eron $lft$-models commutes with products.

\sss Let $T$ be a split $K$-torus of dimension $n$, and let
$\mathscr{T}$ be its N\'eron $lft$-model.
  It follows from
the construction of the N\'eron $lft$-model of
$\mathbb{G}^n_{m,K}$ that there exists a canonical isomorphism
\begin{equation}\label{eq-torcomp}\Comp{T}\cong X(T)^{\vee}
\otimes_{\Z}(K^*/R^*)=\Hom_{\Z}(X(T),K^*/R^*)\end{equation} where
$X(T)$ denotes the character group of $T$. In particular,
$\Comp{T}$ is a free $\Z$-module of rank $n=\mathrm{dim}(T)$.

\sss Under the isomorphism \eqref{eq-torcomp}, a group morphism
$f:X(T)\to K^*/R^*$ corresponds to the unique connected component
of $\mathscr{T}_k$ that contains the specializations of all the
$K$-points $x$ in $T$ such that $f(\chi)$ is the class of
$\chi(x)$ in $K^*/R^*$ for every character $\chi$ of $T$.

\sss Let $K'$ be a finite extension of $K$, with valuation ring
$R'$. Then under the isomorphism \eqref{eq-torcomp} applied to $T$
and $T'=T\times_K K'$, the morphism
$$\alpha_T:\Comp{T}\to \Comp{T'}$$
corresponds to the group morphism
$$X(T)^{\vee} \otimes_{\Z}(K^*/R^*)\to X(T')^{\vee} \otimes_{\Z}((K')^*/(R')^*)$$
induced by the inclusion of $K^*$ in $(K')^*$ and the isomorphism
$X(T)\to X(T')$. It follows that $\alpha$ is an isomorphism from
$\Comp{T}$ onto the sublattice $e\cdot \Comp{T'}$ of $\Comp{T'}$,
where $e$ denotes the ramification index of the extension $K'/K$.
Thus $\alpha$ is injective, and its cokernel is isomorphic to
$(\Z/e\Z)^n$.

\subsection{The N\'eron component series}\label{subsec-compseries}
\sss Let $G$ be an abelian $K$-variety.  In \cite{HaNi-comp}, we
introduced a generating series that encodes the orders  of the
component groups of $G$ after base change to finite tame
extensions of $K$. The N\'eron component series $\CompS{G}(T)$ of
$G$ is defined as
$$\CompS{G}(T)=\sum_{d\in \N'}|\Comp{G(d)}|\cdot T^d\quad \in \Z[[T]].$$
Recall that we denote by $G(d)$ the abelian variety obtained from
$G$ by base change to the unique degree $d$ extension $K(d)$ of
$K$ in $K^s$.

\sss We can extend this definition to semi-abelian $K$-varieties.
Since in this case, the component group $\Comp{G}$ might be
infinite, we will consider the order of the torsion part
$\Comp{G}_{\tors}$. As we've seen, this torsion part is precisely
the group of connected components of the special fiber of the
N\'eron model $\mathscr{G}^{\qc}$ of $G$. We define the N\'eron
component series $\CompS{G}(T)$ of $G$ by
$$\CompS{G}(T)=\sum_{d\in \N'}|\Comp{G(d)}_{\tors}|\cdot T^d\quad \in \Z[[T]].$$
\subsection{Semi-abelian reduction}

\sss It can be quite difficult to describe the behaviour of the
N\'eron model of a semi-abelian $K$-variety under finite
extensions of the base field $K$. The most important tool is
Grothendieck's Semi-Stable Reduction Theorem, which we will now
recall.

\sss Let $G$ be a semi-abelian $K$-variety with N\'eron
$lft$-model $\mathscr{G}$. We say that $G$ has multiplicative,
resp.~semi-abelian, reduction if $\mathscr{G}_k^o$ is a torus,
resp.~a semi-abelian $k$-variety. We say that $G$ has good
reduction if it has semi-abelian reduction and if, in addition, $G_{\ab}$ has good reduction, i.e., if the N\'eron
model of $G_{\ab}$ is an abelian $R$-scheme.

\sss \label{sss-sabtorab} If $G$ is a $K$-torus, then $G$ has
semi-abelian reduction if and only if it has multiplicative
reduction; by \eqref{sss-ramgr}, this happens if and only if $G$
is split. It follows from \cite[4.1]{HaNi-comp} that a
semi-abelian $K$-variety $G$ has semi-abelian reduction if and
only if $G_{\ab}$ has semi-abelian reduction and $G_{\tor}$ is
split.

\begin{prop}\label{lemm-semiab}
Let $G$ be a semi-abelian $K$-variety with semi-abelian reduction,
and let $H$ be a subtorus of $G$. Then $H$ and $G/H$ have
semi-abelian reduction.
\end{prop}
\begin{proof}
 The
torus $H$ is a subtorus of $G_{\tor}$, so that $H$ and
$G_{\tor}/H$ must be split \cite[IX.8.1]{sga3.2}. The quotient
$G/H$ is an extension of $G_{\ab}$ by the split torus
$G_{\tor}/H$. Since $G_{\ab}$ has semi-abelian reduction, we
obtain that $G/H$ has semi-abelian reduction.
\end{proof}

\begin{theorem}[Semi-Stable Reduction Theorem]\label{prop-semiab}
\item \begin{enumerate} \item The semi-abelian variety $G$ has
semi-abelian reduction if and only if the action of $\Gal(K^s/K)$
on $T_\ell G$ is unipotent.
 \item  There exists a unique minimal finite extension $L$ of $K$
in $K^s$ such that $G\times_K L$ has semi-abelian reduction. The
field $L$
 is a finite Galois extension of $K$. \item If
$G$ has semi-abelian reduction, then $G\times_K K'$ has
semi-abelian reduction for every finite separable extension $K'$
of $K$.
\end{enumerate}
\end{theorem}
\begin{proof}
Point (1) follows from \cite[4.1]{HaNi-comp}, and (3) follows from
 \eqref{sss-sabtorab} and \cite[3.3]{sga7.1}. Thus it is enough
 to prove (2).
  It follows from Grothendieck's Semi-Stable Reduction Theorem
for abelian varieties \cite[IX.3.6]{sga7.1} that there exists a
finite separable extension $K_0$ of $K$ such that $G_{\ab}\times_K
K_0$ has semi-abelian reduction. By (3), we may assume that the
torus $G_{\tor}\times_K K_0$ is split; then $G\times_K K_0$ has
semi-abelian reduction by \eqref{sss-sabtorab}. Denote by $I'$ the
subset of $\Gal(K^s/K)$ consisting of elements that act
unipotently on $T_\ell(G)$. This is a normal subgroup of
$\Gal(K^s/K)$. It is open, because it contains the open subgroup
 $\Gal(K^s/K_0)$ of $\Gal(K^s/K)$, by (1). The fixed field $L$ of $I'$
satisfies the properties in the statement.
\end{proof}

\sss The importance of the Semi-Stable Reduction Theorem lies in
the following fact. A semi-abelian $K$-variety $G$ has
semi-abelian reduction if and only if the base change morphism $f$
in \eqref{sss-nerbc} is an open immersion for every finite
extension $K'$ of $K$. This is an immediate consequence of the
Semi-Stable Reduction Theorem and \cite[3.1(e)]{sga7.1}. Thus if
$G$ has semi-abelian reduction, then the formation of the identity
component $\mathscr{G}^o$ of the N\'eron $lft$-model of $G$
commutes with base change to finite extensions of $K$. However, we
shall see later on that the component group $\Comp{G}$ will still
change, unless $G$ is an abelian $K$-variety with good reduction.

\sss The potential toric rank of $G$ is defined as the toric rank
of $G\times_K L$ and denoted by $t_{\pot}(G)$. The potential
abelian rank is defined analogously. We say that $G$ has potential
good reduction if
 $G\times_K L$ has good reduction. Likewise, we say that $G$ has
 potential multiplicative reduction if $G\times_K L$ has
 multiplicative reduction.

\subsection{Non-archimedean uniformization}
\sss Let $A$ be an abelian $K$-variety. We denote by $L$ the
minimal extension of $K$ in $K^s$ such that $A\times_K L$ has
semi-abelian reduction.

\sss \label{sss-uniformAV} The \textit{non-archimedean
uniformization} of $A$ consists of the following data
\cite[1.1]{B-X}:
\begin{itemize}
\item a semi-abelian $K$-variety $E$ that is the extension of an
abelian $K$-variety $B$ with potential good reduction by a
$K$-torus $T$:
$$0\rightarrow T\rightarrow E\rightarrow B\rightarrow 0.$$

\item an \'etale lattice $M$ in $E$, of rank $\mathrm{dim}\,T$,
and an \'etale covering of $K$-analytic groups
$$E^{\an}\rightarrow A^{\an}$$
with kernel $M$.
\end{itemize}
The lattice $M$ and the torus $T$ split over $L$, and $B\times_K
L$ has good reduction. In particular, we can view $M$ as a
$\Gal(L/K)$-module.

\sss The non-archimedean uniformization behaves well under base
change, in the following sense. Let $K'$ be a finite extension of
$K$. If we denote by $(\cdot)'$ the base change functor from $K$
to $K'$, then
$$ 0\to
(M')^{\an} \to (E')^{\an} \to (A')^{\an} \to 0
$$
is the non-archimedean uniformization of $A'=A\times_K K'$.

\section{Models of curves}
In this section, we assume that $k$ is algebraically closed.
\subsection{Sncd-models and combinatorial data}
\sss  Let $C$ denote a smooth, proper and geometrically connected
$K$-curve of genus $g$.
 An $R$-model of $C$ is a  proper and flat $R$-scheme $ \mathscr{C} $
endowed with an isomorphism of $K$-schemes $\mathscr{C}_K\to C$.
We say that a model $ \mathscr{C} $ of $C$ is an $sncd$-model if
$\mathscr{C}$ is regular and the special fiber $$ \mathscr{C}_k =
\sum_{i \in I} N_i E_i $$  is a divisor with strict normal
crossings on $ \mathscr{C} $. It is well known that $C$ always
admits an $sncd$-model. Moreover, if $g>0$, there exists a minimal
$sncd$-model of $C$, which is unique up to unique isomorphism (see
\cite[9.3]{Liubook}). We will often impose the condition that $C$
has index one, i.e., that $C$ admits a divisor of degree one. By
\cite[7.1.6]{raynaud}, this condition is equivalent to the
property that
 $$\gcd\{N_i\,|\,i\in I\}=1.$$


\sss  Let $ \mathscr{C}$ be a regular flat $R$-scheme whose
special fiber
 $ \mathscr{C}_k = \sum_{i \in I} N_i E_i $ is a divisor with strict normal crossings. For each
irreducible component $E_i$ of $ \mathscr{C}_k $, we put $
E_i^{\circ} =  E_i \setminus \cup_{j \neq i} E_j $. We associate a
graph $\Gamma(\mathscr{C}_k)$ to $ \mathscr{C}_k $ as follows. We
let the vertex set $\{\upsilon_i\}_{i \in I}$ correspond
bijectively to the set of irreducible components $\{E_i\}_{i \in
I}$ of $\mathscr{C}_k$. Whenever $ i \neq j$, the vertices $
\upsilon_{i} $ and $ \upsilon_{j} $ are connected by $ \vert E_i
\cdot E_j \vert $ distinct edges. If $\mathscr{C}$ is proper over
$R$, then $\Gamma(\mathscr{C}_k)$ is simply the dual graph
associated to the semi-stable curve $ (\mathscr{C}_k)_{\red} $. By
\emph{the combinatorial data} of $ \mathscr{C}_k $ we mean the
graph $\Gamma(\mathscr{C}_k)$ with each vertex $ \upsilon_{i} $
labelled by a couple $(N_i,g_i)$, where $N_i$ denotes the
multiplicity of $E_i$ in $\mathscr{C}_k$ and $g_i$ denotes the
genus of $E_i$ if $E_i$ is proper over $k$, and $g_i=-1$ else. We
use this definition for $g_i$ to treat in a uniform way the case
where $\mathscr{C}$ is proper over $R$ and the case where
$\mathscr{C}$ is obtained by resolving the singularities of an
excellent local $R$-scheme.

\subsection{A theorem of Winters}
 The following result by Winters will play a crucial role in this
 paper. It will allow us to transfer certain results from residue
 characteristic $0$ to positive residue characteristic.

\begin{theorem}[Winters]\label{theo-winters}
Let $C$ be a smooth, proper, geometrically connected curve over
$K$, and let $\mathscr{C}$ be an $sncd$-model for $C$. Assume that
$C$ admits a divisor of degree one. Then there exist a smooth,
proper, geometrically connected curve $D$ over $\C((t))$ and an
$sncd$-model $\mathscr{D}$ for $D$ over $\C[[t]]$ such that
$\mathscr{C}_k$ and $\mathscr{D}_k$ have the same combinatorial
data.
\end{theorem}
\begin{proof}
This follows from \cite[3.7 and 4.3]{winters}.
\end{proof}
\sss Note that the curve $D$ automatically has the same genus as
$C$, since the genus can be computed from the combinatorial data
of $\mathscr{C}_k$ (see for instance \cite[3.1.1]{ni-saito}). If
$\mathscr{C}$ is relatively minimal, then so is $\mathscr{D}$,
since the existence of $(-1)$-curves can also be read off from the
combinatorial data.

\subsection{N\'eron models of Jacobians}
\sss 
 Let $C$ be a smooth,
proper, geometrically connected curve over $K$ of index one. One
can use the geometry of $R$-models of $C$ to study the N\'eron
model of the Jacobian variety $A=\Jac(C)$ of $C$, because of the
following fundamental theorem of Raynaud \cite[9.5.4]{neron}: if
$\mathscr{C}$ is a regular $R$-model of $C$, then the relative
Picard scheme $\mathrm{Pic}^{0}_{\mathscr{C}/R}$ is canonically
isomorphic to the identity component of the N\'eron model
$\mathscr{A}$ of $A$.

\sss This result has several interesting consequences. Let $
\mathscr{C}/R$ be an $sncd$-model of $C$, with special fiber $
\mathscr{C}_k = \sum_{i \in I} N_i E_i $.  The abelian quotient in
the Chevalley decomposition of $\mathscr{A}^o_k$ is isomorphic to
$$\prod_{i\in I}\mathrm{Pic}^o_{E_i/k},$$ by \cite[9.2.5 and
9.2.8]{neron}, and the toric rank $\trank{A}$ of $A$  is equal to
the first Betti number of the graph
 $\Gamma(\mathscr{C}_k)$ \cite[9.2.5 and 9.2.8]{neron}.

\sss It is also possible to compute the component group $ \Comp A
$ from the combinatorial data of $ \mathscr{C}_k $, as follows.
 Consider the complex of abelian groups
$$ \mathbb{Z}^I \overset{\alpha}{\longrightarrow} \mathbb{Z}^I \overset{\beta}{\longrightarrow} \mathbb{Z}, $$
where $ \alpha = (E_i \cdot E_j)_{i,j \in I} $ is the intersection
matrix of $\mathscr{C}_k$ and $ \beta$ sends the $i$-th standard
basis vector of $\Z^I$ to $N_i \in \Z$, for every $i\in I$. Then,
by \cite[9.6.1]{neron}, there is a canonical isomorphism $ \Comp A
\cong \ker(\beta)/\im(\alpha) $.

\sss In particular, both the
 component group $\Comp{A}$ and the toric rank $\trank{A}$ only
 depend on the combinatorial data of $\mathscr{C}_k$, and not on
 the characteristic exponent $p$ of $k$.

\subsection{Semi-stable reduction}
\sss Let $C$ be a smooth, proper, geometrically connected curve
over $K$. An $sncd$-model $\mathscr{C}$ is called semi-stable if
its special fiber $\mathscr{C}_k$ is reduced. We say that $C$ has
semi-stable reduction if $C$ has a semi-stable $sncd$-model.
 T.~Saito has proven in \cite[3.8]{Saito} that, unless $C$ is
a genus one curve without rational point, the curve $C$ has
semi-stable reduction if and only if the action of $\Gal(K^s/K)$
on $$H^1(C\times_K K^s,\Q_\ell)$$ is unipotent; see also
\cite[3.4.2]{ni-saito}.

\sss Assume that $C$ has genus $g\neq 1$ or that $C$ is an
elliptic curve, and set $A=\Jac(C)$. Then there exist canonical
Galois-equivariant isomorphisms \begin{eqnarray*} H^1(A\times_K
K^s,\Q_\ell)&\cong& H^1(C\times_K K^s,\Q_\ell),
\\H^1(A\times_K
K^s,\Q_\ell)&\cong&\Hom_{\Z_\ell}(T_\ell A,\Q_\ell).
\end{eqnarray*}
 Thus Theorem \ref{prop-semiab}(1) implies that $A$ has
 semi-abelian reduction if and only if $C$ has semi-stable
 reduction. If $C$ has index one, this equivalence can also be
 deduced (with some additional work) from Raynaud's isomorphism $$\mathrm{Pic}^{0}_{\mathscr{C}/R}\cong \mathscr{A}^o.$$

\part{N\'eron component groups of semi-abelian
varieties}\label{part-comp}
\chapter[Models of curves]{Models of curves and the N\'eron component series of a
Jacobian}\label{chap-jacobians}

In this chapter, we assume that $k$ is algebraically closed. Let
$C$ be a smooth, proper, geometrically connected curve over $K$.
We will study the behaviour of $sncd$-models of $C$ under finite
tame extensions of the base field $K$. Our main technical result
is that these models can be compared in a very explicit way if the
degree of the base change is prime to the {\em stabilization
index} $e(C)$ of $C$, a new invariant that we introduce in
Definition \ref{def-stabind}. Using this result, we prove the
rationality of the N\'eron component series of a Jacobian variety
over $K$ (Theorem \ref{theorem-compser}).

\section{Sncd-models and tame base change}\label{sec-sncd}
\subsection{Base change and normalization}

\sss \label{sss-desingsetup} Let $C$ be a smooth, proper,
geometrically connected curve over $K$. Let $\mathscr{C}/R$ be an
$sncd$-model of $C$ and let $d$ be an element of $\mathbb{N}'$. We
denote by $\mathscr{C}_{d}$ the normalization of
$\mathscr{C}\times_R R(d)$ and by
$$
f : \mathscr{C}_{d} \to \mathscr{C} $$ the canonical morphism. We
denote by
$$
\rho : \mathscr{C}(d) \rightarrow \mathscr{C}_d $$ the minimal
desingularization of $\mathscr{C}_d$. For the applications we have
in mind, it is important to describe $ \mathscr{C}_{d} $ and $
\mathscr{C}(d) $ in a precise way. Such a description will be
given in Proposition \ref{prop-normdesing}. In particular, we will
show that $ \mathscr{C}(d) $ is an $sncd$-model. These results are
well known in more restrictive settings (cf.~\cite{Halle-stable}).
However, to our best knowledge, they have not appeared in the
literature in the generality that we need (although some of them
are claimed without proof in Section 3 of \cite{lorenzini}). In
particular, we need to deal with the situation where
$\mathscr{C}_k$ contains irreducible components with
multiplicities divisible by $p$ that intersect each other, and
this case is not covered in \cite{Halle-stable}.

\sss  The main tool we use is the description of the minimal
desingularization of a \emph{locally toric singularity} given in
Kiraly's PhD thesis \cite{Kir}. This provides a convenient
 combinatorial description of the minimal desingularization of a
 tame cyclic quotient singularity that also applies to the case where $R$ has  mixed characteristic
 (the equal characteristic case was worked out in \cite{CES}).
 Since this part of \cite{Kir} has not been published, we
have gathered the results that we need as an appendix,
in Section \ref{sec-loctor}. Alternatively, one could use tools
from logarithmic geometry (desingularization of log-regular
schemes) since Kiraly's locally toric singularities correspond
precisely to Kato's toric singularities 
\cite[3.1]{kato}.

\sss  Let $x$   be a closed point on $\mathscr{C}_k$, and let
$y_1, \ldots, y_r$ denote the finitely many points in the special
fiber of $\mathscr{C}_{d}$ mapping to $x$. Since normalization of
an excellent scheme commutes with localization and completion
\cite[\S2.2]{Halle-stable}, the homomorphism $
\widehat{\mathscr{O}}_{\mathscr{C},x} \to
\widehat{\mathscr{O}}_{\mathscr{C}_d,y_i} $ induced by $f$ can be
identified with the composed sequence of homomorphisms

$$
\widehat{\mathscr{O}}_{\mathscr{C},x} \to
\widehat{\mathscr{O}}_{\mathscr{C},x} \otimes_R R(d) \to
(\widehat{\mathscr{O}}_{\mathscr{C},x} \otimes_R R(d))^{\nor}
\cong \prod_{i=1}^r \widehat{\mathscr{O}}_{\mathscr{C}_d,y_i} \to
\widehat{\mathscr{O}}_{\mathscr{C}_d,y_i},
$$
where, for any reduced ring $A$, we write $A^{\nor}$ for the
integral closure of $A$ in its total ring of fractions.

 \sss  The situation where only one irreducible
component of $\mathscr{C}_k$ passes through $x$ is completely
described in \cite[\S2.4]{Halle-stable}, so we assume in the
following that there are two distinct irreducible components of
$\mathscr{C}_k$ passing through $x$. In that case, we can find an
isomorphism
$$ \widehat{\mathscr{O}}_{\mathscr{C},x} \cong R[[x_1,x_2]]/(\pi-u \cdot x_1^{m_1}x_2^{m_2}), $$
where $m_1$ and $m_2$ are the multiplicities of the components of
$\mathscr{C}_k$ intersecting at $x$, and where $ u \in R[[x,y]] $
is a unit. As we do not want to assume that either $m_1$ or $m_2$
is prime to $p$, we cannot get rid of the unit $u$ by a coordinate
change. This prevents us from simply transferring the results in
\cite{Halle-stable}.

\subsection{Local computations}\label{subsec-loccomp}

\sss  Let $d \in \mathbb{N}'$ and put $R' = R(d)$. We choose a
uniformizer $ \pi' $ in $R'$ such that $ \pi'^d = \pi$. We will
now explain how to normalize $A \otimes_R R'$, where
$$ A = R[[x_1,x_2]]/(\pi-u \cdot x_1^{m_1}x_2^{m_2}). $$

To do this, write $ c = \gcd(d,m_1,m_2) $. Since $c$ is prime to
$p$, we can choose a unit $ v \in R[[x_1,x_2]] $ such that $ v^c =
u $. Then
$$ (A \otimes_R R')^{\nor} \cong  \prod_{\xi \in \mu_c(k)}\left(R'[[x_1,x_2]]/(\pi'^{d'} - \xi v \cdot x_1^{m_1'} x_2^{m_2'})\right)^{\nor}, $$
where $d' = d/c$, $m_1'=m_1/c$ and $m_2'=m_2/c$.  Let us put $ e_i
= \gcd(d',m_i') $ so that $ d' = e_1 e_2 d'' $ and $ m_i' = e_i
m_i'' $ for $ i = 1, 2 $, with $d''$ and $m_i''$ in $\N$. We
moreover fix a unit $ w \in R[[x_1,x_2]] $ such that $ w^{e_1e_2}
= \xi v $. Then one can argue as in the proof of
\cite[2.4]{Halle-stable} to see that the $R'$-homomorphism
$$
\alpha: R'[[x_1,x_2]]/(\pi'^{d'} - \xi v \cdot x_1^{m_1'}
x_2^{m_2'}) \to R'[[y_1,y_2]]/(\pi'^{d''} - w \cdot y_1^{m_1''}
y_2^{m_2''})$$ defined by $ x_1 \mapsto y_1^{e_2} $ and $ x_2
\mapsto y_2^{e_1} $ is finite and injective, and that it
 induces an isomorphism of fraction fields. Thus the source and
 target of $\alpha$ have the same normalization.

We next fix a unit $ w' \in R[[x_1,x_2]] $ such that $w'^{d''} = w
$, and define an $R'$-homomorphism
$$ \beta : R'[[y_1,y_2]]/(\pi'^{d''} - w \cdot y_1^{m_1''} y_2^{m_2''}) \to R'[[z_1,z_2]]/(\pi' - w' \cdot z_1^{m_1''} z_2^{m_2''}) $$
by $ y_i \mapsto z_i^{d''} $ for $ i = 1, 2 $. We let $
\mu_{d''}(k) $ act on the latter ring by $ \zeta\ast z_1= \zeta
z_1 $ and $ \zeta\ast z_2 = \zeta^r z_2 $, where $ 0 < r < d'' $
is the unique integer such that $ r m_2'' + m_1'' \equiv 0 $
modulo $d''$. Note that $ w' \in R[[x_1,x_2]] $ is invariant under
this action.

\begin{lemma}\label{lemma-invariantring}
The image of $\beta$ is contained in $$(R'[[z_1,z_2]]/(\pi' - w'
\cdot z_1^{m_1''} z_2^{m_2''}))^{\mu_{d''}(k)},$$ and the
homomorphism
 $$R'[[y_1,y_2]]/(\pi'^{d''}
- w \cdot y_1^{m_1''} y_2^{m_2''}) \to (R'[[z_1,z_2]]/(\pi' - w'
\cdot z_1^{m_1''} z_2^{m_2''}))^{\mu_{d''}(k)}$$ is a
normalization morphism.
 Thus we obtain an isomorphism
$$ ( R'[[x_1,x_2]]/(\pi'^{d'} - \xi v \cdot x_1^{m_1'}
x_2^{m_2'}))^{\nor} \cong (R'[[z_1,z_2]]/(\pi' - w' \cdot
z_1^{m_1''} z_2^{m_2''}))^{\mu_{d''}(k)}. $$
\end{lemma}
\begin{proof}
This can be proved in a similar way as
\cite[Prop.~2.7]{Halle-stable}.
\end{proof}

\subsection{Minimal desingularization}\label{subsec-mindesing}
\sss  Let $ \mathscr{C}/R $ be an $sncd$-model of $C$ with special
fiber $ \mathscr{C}_k = \sum_{i \in I} N_i E_i $. For any $ d \in
\mathbb{N}' $, let $ f : \mathscr{C}_d \to \mathscr{C} $ be the
composition of base change to $R(d)$ and normalization  and let $
\rho : \mathscr{C}(d) \to \mathscr{C}_d $ be the minimal
desingularization as in Section \ref{sss-desingsetup}. Proposition
\ref{prop-normdesing} below lists some properties of these
morphisms, and of the schemes $ \mathscr{C}_d $ and $
\mathscr{C}(d) $, that are relevant for the applications further
on.



\begin{prop}\label{prop-normdesing}
 For every $ d \in \mathbb{N}' $,  the following properties hold:
\begin{enumerate}
\item For each irreducible component $E_i$ of $ \mathscr{C}_k $,
the scheme $$F_i=\mathscr{C}_d\times_{\mathscr{C}}E_i$$ is a
disjoint union of smooth irreducible curves $F_{ij}$. The
multiplicity $N_i'$ of $ (\mathscr{C}_d)_k $ along each component
$ F_{ij} $ is given by
$$ N_i' = N_i/\gcd(d,N_i), $$ and the morphism
$$\mathscr{C}_d\times_{\mathscr{C}}E_i^o\to E_i^o$$ is a Galois
cover of degree $\gcd(N_i,d)$.

\item If $E_i$ is a rational curve that intersects the other
components of $\mathscr{C}_k$ in precisely one (resp. two) points,
then each $F_{ij}$ is a rational curve that intersects the other
components of $(\mathscr{C}_d)_k$ in precisely one (resp. two)
points. In both cases, the number of connected components of
$F_{i}$ is equal to $n_i=\gcd(N_i,N_{a},d)$ where $a$ is any
element of $I\setminus \{i\}$ such that $E_{a}$ intersects $E_i$.
In particular, $n_i$ does not depend on the choice of $a$.

\item Each singular point of $ \mathscr{C}_d $ is an intersection
point of two distinct irreducible components of the special fiber.
Moreover, let $x$ be a point that belongs to the intersection of
two distinct irreducible components $F$ and $F'$ of $
(\mathscr{C}_d)_k $ which dominate irreducible components $E$ and
$E'$ of $\mathscr{C}_k$, respectively. Let $N$ and $N'$ be the
multiplicities of $E$  and $E'$ in $ \mathscr{C}_k $.  Then the
special fiber of the minimal desingularization $\mathscr{D}$ of
the local germ $\Spec \mathscr{O}_{\mathscr{C}_d,x}$  is a divisor
with strict normal crossings
 whose combinatorial data only depend on $N$, $N'$ and
$d$. Moreover, each exceptional component of $\mathscr{D}_k$ is a
rational curve that meets
 the other irreducible components of $\mathscr{D}_k$ in precisely
 two points.
\end{enumerate}
In particular, the $R(d)$-scheme $\mathscr{C}(d)$ is an
$sncd$-model of $C(d)$.
\end{prop}
\begin{proof}
(1) For any $i \in I$, an easy local computation shows that
 $$\mathscr{C}_d\times_{\mathscr{C}}E_i^o\to E_i^o$$ is a Galois cover of degree $\gcd(N_i,d)$ and that $N'_i=N_i/\gcd(N_i,d)$
(cf.~\cite[\S2.4]{Halle-stable}). Let $ x $ be a closed point of
$\mathscr{C}_d$ such that $f(x) $ belongs to the intersection of
$E_i$ with another component $E_j$. Then, using the explicit
computation of $ \widehat{\mathscr{O}}_{\mathscr{C}_d,x} $ in
Lemma \ref{lemma-invariantring}, the proof of
\cite[2.9]{Halle-stable} shows that $F_i$ is smooth at $x$ also.

(2) This is shown in the proofs of Propositions 3.1 and 3.2 in
\cite{Halle-stable}. To be precise, these proofs are written in
\cite{Halle-stable} under an additional assumption on the
multiplicities of the components of $\mathscr{C}_k$, but this
assumption can be removed by using the computations in Section
\ref{subsec-loccomp} instead of the ones in
\cite[\S2]{Halle-stable}.

(3) Since $\mathscr{C}_d$ is normal, the singular locus consists
of finitely many closed points in the special fiber. If $y \in
\mathscr{C}_d$ is a closed point that belongs to a unique
irreducible component of $(\mathscr{C}_d)_k$, then $
\mathscr{C}_d$ is regular at $y$ by \cite[Cor.~2.2]{Halle-stable}.
Let now $x$ be a point as in the statement of the lemma. Then, by
Lemma \ref{lemma-invariantring}, the germ $\Spec
\mathscr{O}_{\mathscr{C}_d,x}$ is a tame cyclic quotient
singularity, and its formal structure only depends on the triple
$(N,N',d)$ by the computations in Section \ref{subsec-loccomp}.
The structure of its minimal resolution is described in
Proposition \ref{prop-mindesing}.
  In particular, it follows from that
description that the special fiber of the minimal resolution is a
divisor with strict normal crossings.
\end{proof}

\section{The characteristic polynomial and the stabilization index}\label{sec-e}
In this section, we introduce two invariants of a smooth, proper
and geometrically connected $K$-curve $C$: the characteristic
polynomial $P_C(t)$ and the stabilization index $e(C)$. They have
a natural cohomological interpretation when $C$ is tamely
ramified, but their meaning is somewhat mysterious in the wildly
ramified case. A crucial feature of the stabilization index $e(C)$
is that the behaviour of a relatively minimal $sncd$-model of $C$
under tamely ramified base change can be controlled completely if
the degree of the base change is prime to $e(C)$.
 This property will be
essential for our results on N\'eron component groups of
Jacobians.

\subsection{The characteristic polynomial}
\sss \label{sssec-notC} Let $C$ be a smooth, proper, geometrically
connected $K$-curve of genus $g$. Let $\mathscr{C}/R$ be an
$sncd$-model of $C$ with special fiber $\mathscr{C}_k = \sum_{i
\in I} N_i E_i $.

\begin{definition}\label{def-charpol}
The characteristic polynomial of $C$ is the monic polynomial
$$P_C(t)=(t-1)^2\prod_{i\in I}(t^{N_i}-1)^{-\chi(E_i^o)}$$ in
$\Z[t]$ of degree $2g$.
\end{definition}
The fact that $P_C(t)$ is indeed a polynomial of degree $2g$ was
proven by the second author in \cite[3.1.6]{ni-saito}, and
previously by Lorenzini in \cite{lorenzini} under the assumption
that $\gcd\{N_i\,|\,i\in I\}=1$. Although $P_C(t)$ is defined in
terms of the $sncd$-model $\mathscr{C}$, it is easy to see that it
does not depend on the choice of such a model, since the
expression in Definition \ref{def-charpol} does not change if we
blow up $\mathscr{C}$ at a closed point of $\mathscr{C}_k$.
 We do not know how to define $P_C(t)$ intrinsically on $C$,
 without reference to an $sncd$-model. However, in \cite{ni-saito}, the second author
 proved the following result.

 \begin{prop}\label{prop-charpol}
Let $\sigma$ be a topological generator of the tame inertia group
$\mathrm{Gal}(K^t/K)$, and denote by $P'_{C}(t)$ the
characteristic polynomial of $\sigma$ on $$H^1(C\times_K
K^t,\Q_\ell).$$ For every $i\in I$, we denote by $N'_i$ the
prime-to-$p$ part of $N_i$. Then the following properties hold:
\begin{enumerate}
\item $P'_C(t)=(t-1)^2\prod_{i\in I}(t^{N'_i}-1)^{-\chi(E_i^o)}$,
\item $P'_C(t)$ divides $P_C(t)$, and they are equal if and only
if $C$ is tamely ramified.
\end{enumerate}
 \end{prop}
 \begin{proof}
This follows from formula (3.4) and Corollary 3.1.7 in
\cite{ni-saito}.
 \end{proof}
 We will
 now explain the behaviour of $P_C(t)$ under tamely ramified base
 change.

\begin{definition}
 For every monic
polynomial
$$Q(t)=\prod_{j=1}^r (t-\alpha_j)\in \C[t]$$ and every integer
$d>0$, we set
$$Q^{(d)}(t)=\prod_{j=1}^r (t-\alpha^d_j)\in \C[t].$$
\end{definition}
\begin{lemma}\label{lemma-combinat}
Consider integers $a,b,d>0$, and set $e=\gcd(a,d)$. If
$Q(t)=(t^a-1)^b$, then
$$Q^{(d)}(t)=(t^{a/e}-1)^{eb}.$$
\end{lemma}
\begin{proof}
The morphism
$$\mu_a(\C)\to \mu_{a/e}(\C):\zeta\mapsto \zeta^d$$ is a
surjection, and every fiber contains precisely $e$ elements.
\end{proof}

  If $C$
 is tamely ramified, then Proposition \ref{def-charpol} implies at
 once that $P_{C(d)}(t)=P^{(d)}_C(t)$ for every element $d$ in $\N'$.
 The following proposition states that this remains true in the wildly
 ramified case (see also \cite[2.6]{lorenzini}).

 \begin{prop}\label{prop-charpolbc}
Let $d$ be an element in $\N'$. For every $i\in I$, we set
$d_i=\gcd(d,N_i)$. Then we have
$$P_{C(d)}(t)=P^{(d)}_C(t)=(t-1)^2\prod_{i\in I}(t^{N_i/d_i}-1)^{-d_i\chi(E_i^o)}.$$
 \end{prop}
\begin{proof}
We will compute $P_{C(d)}(t)$ on the $sncd$-model $\mathscr{C}(d)$
for $C(d)$ from Section \ref{subsec-mindesing}. By Proposition
 \ref{prop-normdesing}(3), every exceptional component $F$ of the
minimal desingularization $\rho:\mathscr{C}(d)\to \mathscr{C}_d$
satisfies $\chi(F^o)=0$, so that these exceptional components do
not contribute to $P_{C(d)}(t)$. For every $i\in I$, we have that
$$\mathscr{C}(d)\times_{\mathscr{C}}E_i^o$$ is a disjoint
union of strata $F^o_{ij}$, with $F_{ij}$ the irreducible
components of $\mathscr{C}(d)_k$ dominating $E_i$. By Proposition
 \ref{prop-normdesing}(1), each of these components $F_{ij}$ has
multiplicity $N_i/d_i$, and
$$\mathscr{C}(d)\times_{\mathscr{C}}E_i^o\to E_i^o$$ is a Galois
cover of degree $d_i$. Since $d$ is prime to $p$, we know that
$$\sum_{j}\chi(F_{ij}^o)=\chi(\mathscr{C}(d)\times_{\mathscr{C}}E_i^o)=d_i\cdot \chi(
E_i^o)$$ where the first equality follows from the additivity of
the Euler characteristic and the second from Hurwitz' theorem.
 Thus
$$P_{C(d)}(t)=(t-1)^2\prod_{i\in
I}(t^{N_i/d_i}-1)^{-d_i\chi(E_i^o)}.$$ It remains to show that
this expression is equal to $P_C^{(d)}(t)$. Since the operator
$Q(t)\mapsto Q^{(d)}(t)$ is clearly multiplicative, this follows
immediately from Lemma \ref{lemma-combinat}.
\end{proof}

\subsection{The stabilization index}
\sss  We keep the notations introduced in \eqref{sssec-notC}.

\begin{definition}\label{def-prin}
We say that $E_i$ is a principal component of $\mathscr{C}_k$ if
the genus of $E_i$ is non-zero
 or $ E_i \setminus E_i^{\circ} $ contains at least three
points. We denote by $I_{\prin} \subseteq I$ the subset
corresponding to principal components of $\mathscr{C}_k$.
\end{definition}

\begin{definition}\label{def-stabind}
We set
$$ e(\mathscr{C}) = \lcm_{i \in I_{\prin}}\{N_i\}. $$
If $\mathscr{C}_{\min}$ is a relatively minimal $sncd$-model of
$C$, then we set
$$ e(C) = e(\mathscr{C}_{\min}). $$
 We call $ e(C) $ the stabilization index of $C$.
\end{definition}
\noindent If $g=0$, then $C=\mathbb{P}^1_K$ and the special fiber
of every relatively minimal $sncd$-model of $C$ is
 isomorphic to $\mathbb{P}^1_k$, so that $e(C)=1$. If $g>0$, then
 $C$ has a unique minimal $sncd$-model, so that it is clear that
 $e(C)$ is well-defined. It should be noted that the value $e(\mathscr{C})$ depends on the
model $\mathscr{C}$ and not only on $C$. Every irreducible
component $E$ of $\mathscr{C}_k$ can be turned into a principal
component by blowing up $3- |E\setminus E^o|$ distinct points on
$E^o$. On the other hand, $e(C)$ always divides $e(\mathscr{C})$,
since no additional principal components can be created in the
contraction of $\mathscr{C}$ to a relatively minimal $sncd$-model.

 In the tamely ramified case, the stabilization index can be
 interpreted as follows.
\begin{prop}\label{prop-min}  Assume that $g\neq 1$ or that
$C$ is an elliptic curve. Let $L$ be the minimal extension of $K$
in $K^s$ such that $C\times_K L$ has semi-stable reduction, and
let $\sigma$ be a topological generator of the tame inertia group
$\Gal(K^t/K)$.

 The curve $C$ is tamely ramified if and only
if $e(C)$ is prime to $p$.  In that case, $e(C)$ is equal to
$[L:K]$, and this value is the smallest element $d$ in $\N'$ such
that $\sigma^d$ acts unipotently on
$$H^1(C\times_K K^t,\Q_\ell).$$
\end{prop}
\begin{proof}
 First, assume that $C$ is tamely ramified.  Then
$e(C)=\dgr{L}{K}$ by \cite[7.5]{Halle-stable} or
\cite[3.4.4]{ni-saito}. Since $\dgr{L}{K}$ is prime to $p$,  we
see that $e(C)$ is prime to $p$. Suppose, conversely, that $e(C)$
is prime to $p$, and that $\mathscr{C}$ is relatively minimal.
Then we have $\chi(E_i^o)\geq 0$ for every component $E_i$ of
$\mathscr{C}_k$ such that $N_i$ is not prime to $p$, because such
a component can not be principal. Now \cite[3.1.5]{ni-saito}
implies that $\chi(E_i^o)= 0$ for every such component $E_i$, and
\cite[3.1.7]{ni-saito} tells us that $C$ is tamely ramified. The
remainder of the statement is a consequence of \cite[3.11]{Saito};
see also \cite[3.4.2]{ni-saito}.
\end{proof}
\noindent The equality $e(C)=\dgr{L}{K}$ can fail for wildly
ramified curves: even the prime-to-$p$ parts of $e(C)$ and
$\dgr{L}{K}$ can be different, as is illustrated by Examples
\ref{ex-diff} and \ref{ex-diff2}.
 It
would be quite interesting to find a cohomological interpretation
of $e(C)$ in the wildly ramified case.

\begin{example}\label{ex-diff}
Assume that $k$ has characteristic $2$ and that $R=W(k)$. Let $C$
be the elliptic $K$-curve with Weierstrass equation
$$ y^2 = x^3 + 2. $$
 It is easily computed, using Tate's algorithm, that this equation is minimal, and that $C$ has
reduction type $II$ over $R$ and acquires good reduction over the
wild Kummer extension $L=K(\sqrt{2})$ of $K$. Thus $e(C)=6$
whereas $\dgr{L}{K}=2$.
\end{example}

\begin{example}\label{ex-diff2}
Assume that $k$ has characteristic $2$ and that $R=k[[\pi]]$. Let $C$
be the elliptic $K$-curve with Weierstrass equation
$$ y^2 + \pi^2 y = x^3 + \pi^3 . $$
Using Tate's algorithm, we find that $C$ has reduction type
$I_0^*$ over $R$, so that $e(C)=2$.

On the other hand, let $\alpha$ be an element of $K^s$ satisfying
$$ \alpha^2 + \pi^2 \alpha = \pi^3, $$ and set $K'=K(\alpha)$.
This is a quadratic Artin-Schreier extension of $K$.
 Let $ L = K'(3) $ be the unique tame extension of $K'$ in $K^s$ of
degree $3$. Then it is easy to check that $C\times_K L$ has good
reduction and that $L$ is the minimal extension of $K$ with this
property, so that $[L:K] = 6$.
\end{example}

 \sss  Thanks to Propositions
\ref{prop-charpol} and \ref{prop-min}, one can compute $e(C)$ from
the characteristic polynomial $P_C(t)$ if $C$ is tamely ramified.
 We'll now show that this recipe is also valid in the wild case.
This will then allow us to control the behaviour of $e(C)$ under
tame base change, using Proposition \ref{prop-charpolbc}.

\begin{prop}\label{prop-charpol-e}
Assume that $g\neq 1$ or that $C$ is an elliptic curve. The
stabilization index $e(C)$ is the smallest integer $e>0$ such that
$\zeta^e=1$ for every complex root $\zeta$ of $P_C(t)$.
\end{prop}
\begin{proof}
Assume that $\mathscr{C}$ is a relatively minimal $sncd$-model of
$C$. Let $e$ be the smallest strictly positive integer such that
$\zeta^e=1$ for every complex root $\zeta$ of $P_C(t)$. It is
clear that $e|e(C)$, because every irreducible component $E$ of
$\mathscr{C}_k$ with $\chi(E^o)<0$ is principal. Thus it suffices
to show that $N$ divides $e$ if $N>1$ is the multiplicity of a
principal component in $\mathscr{C}_k$. This follows at once from
\cite[3.2.3]{ni-saito}, because \cite[3.2.2]{ni-saito} implies
that $\mathscr{C}$ is not $N$-tame.
\end{proof}


\begin{prop}\label{prop-e(C)}
For every integer $ d \in \mathbb{N}' $, we have that
$$ e(\mathscr{C}(d)) = e(\mathscr{C})/\gcd(e(\mathscr{C}),d). $$
\end{prop}
\begin{proof}
Let $E$ be an irreducible component of $\mathscr{C}_k$, and let
$E'$ be an irreducible component of $\mathscr{C}(d)_k$ that
dominates $E$. Then $g(E') \geq g(E)$ by
\cite[IV.2.5.4]{hartshorne} and it
  is clear that,
if $E$ meets the other components of $\mathscr{C}_k$ in $n$
distinct points, then $ E'$ meets the other components of
$\mathscr{C}(d)_k$ in at least $n$ points. Thus $E'$ is principal
if $E$ is principal. By Proposition \ref{prop-normdesing}(1), the
multiplicity $N'$ of $E'$ in $\mathscr{C}(d)_k$ equals
$N/\gcd(N,d)$, where $N$ is the multiplicity of $E$ in
$\mathscr{C}_k$.

On the other hand, it follows from Proposition
\ref{prop-normdesing}(2) that $E'$ is non-principal if $E$ is
non-principal.
 By Proposition \ref{prop-normdesing}(3) any
exceptional component of $ \rho : \mathscr{C}(d) \to \mathscr{C}_d
$ is non-principal. If we write $ \mathscr{C}(d)_k = \sum_{j \in
I(d)} N_j' E'_j $, it follows that
$$ e(\mathscr{C}(d)) = \lcm_{j \in I(d)_{\prin}} \{N'_j\} = \lcm_{i \in I_{\prin}} \{N_i/\gcd(N_i,d)\} = e(\mathscr{C})/\gcd(e(\mathscr{C}),d). $$
\end{proof}

\begin{cor}\label{cor-e(C)} Assume that $C$ is not a genus one curve without rational point whose Jacobian has
 additive reduction and potential multiplicative reduction. Then for every $ d \in \mathbb{N}' $, we have that
$$ e(C(d)) = e(C)/\gcd(e(C),d). $$
\end{cor}
\begin{proof}
If $g\neq 1$ or $C$ is an elliptic curve, this follows from
Propositions \ref{prop-charpolbc} and \ref{prop-charpol-e}. Thus
we may assume that $C$ is a genus one curve without rational
point. Its Jacobian is an elliptic curve $E$, and $C$ is an
$E$-torsor over $K$. If we denote by $m$ the order of the class of
$C$ in the Weil-Ch\^atelet group $H^1(K,E)$, then the reduction
type of $C$ is $m$ times the reduction type of $E$, by
\cite[6.6]{liu-lorenzini-raynaud}. Assume that $\mathscr{C}$ is
the minimal $sncd$-model of $C$. Looking at the Kodaira-N\'eron
reduction table, we see that all principal components of the
special fiber $\mathscr{C}_k$ have the same multiplicity $N$.
We've shown in the proof of Proposition \ref{prop-e(C)} that every
component of $\mathscr{C}(d)_k$ that dominates a principal
component of $\mathscr{C}_k$ is itself principal, and that these
are the only principal components of $\mathscr{C}(d)_k$. Their
multiplicities are all equal to
$$N/\gcd(N,d)=e(C)/\gcd(e(C),d).$$ No new principal components are
created by the contraction of $\mathscr{C}(d)$ to the minimal
$sncd$-model $\mathscr{C}(d)_{\min}$ of $C(d)$, and the special
fiber of $\mathscr{C}(d)_{\min}$ contains at least one principal
component, unless $E(d)$ has multiplicative reduction. In the
latter case, it follows from our assumptions that $E$ already had
multiplicative reduction, so that $e(C)=e(C(d))=1$.
\end{proof}

\sss \label{subsection-genus1withoutpoint} Let us briefly comment
on the case where $C$ is a genus one curve without rational point;
this case is irrelevant for the applications in the following
sections, since we will only be interested in the Jacobian of $C$,
which is an elliptic curve.

 Let $E$ be an elliptic curve over $K$. The genus one $K$-curves
 with Jacobian $E$ are classified by the Weil-Ch\^atelet group
 $H^1(K,E)$. It is known that the group $H^1(K,E)$ is non-trivial
when $E$ has semi-stable reduction. If $E$ has additive reduction,
then $H^1(K,E)$ is a $p$-group, and it is
 non-trivial when $p>1$  \cite[6.7]{liu-lorenzini-raynaud}.
 Now let $C$ be a genus one curve with Jacobian $E$. We denote by $m$ the order of the class of $C$ in $H^1(K,E)$.  As we already recalled in the proof of
Corollary \ref{cor-e(C)}, the reduction type of $C$ is $m$ times
the reduction type of $E$ \cite[6.6]{liu-lorenzini-raynaud}. This
implies that $m$ is equal to the index of $C$ (the greatest common
divisor of the degrees of the closed points of $C$), and also to
the minimum of the degrees of the closed points on $C$
\cite[9.1.9]{neron}.
\begin{enumerate}
\item  Proposition \ref{prop-min} fails for $C$ if $p>1$,
$\mathscr{C}_k$ has a principal component, $m$ is divisible by $p$
and $E$ is tamely ramified. This can happen, for instance, if
$p\geq 5$ and $E$ has additive reduction.

\item Even if $k$ has characteristic zero, Proposition
\ref{prop-min} can fail for $C$.  If $E$ has multiplicative
reduction, then $e(C)=1$ but $\dgr{L}{K}=m$. Likewise, Proposition
\ref{prop-charpol-e} might fail: if $E$ has good reduction, then
$P_C(t)=(t-1)^2$ while $e(C)=m$.

 \item
 The
   case excluded in the
 statement of Corollary \ref{cor-e(C)} never occurs if $p=1$,
 because then $H^1(K,E)=0$ if $E$ has additive reduction.
 However, it does occur if $p>1$. Assume that $E$ has additive reduction and acquires multiplicative reduction over $K(d)$, for some $d\in \N'$. Then
 $m$ is a power of $p$ and $e(C)=m\cdot e(E)$ while $e(C(d))=1$, so that the property in Corollary
 \ref{cor-e(C)} does not hold if $m>1$.
 \end{enumerate}


\subsection{Applications to $sncd$-models and base change}
\sss  We keep the notations from \eqref{sssec-notC}. An important
feature of the invariant $e(\mathscr C)$ is that we can give a
rather precise description of the special fiber of $\mathscr C(d)$
in terms of the special fiber of $\mathscr C$, as long as $d \in
\N'$ is prime to $e(\mathscr C)$.

\begin{lemma}\label{lemma-fiberdescription} Assume that $C$ is not
a genus one curve without rational point whose Jacobian has
multiplicative reduction. Let $d$ be an element of $\mathbb{N}'$
that is prime to $e(\mathscr{C})$. For every irreducible component
$E_i$ of $\mathscr{C}_k$, the $k$-scheme
$$F_i=\mathscr{C}_d\times_{\mathscr{C}}E_i$$ is smooth and
irreducible. We denote by $N'_i$ the multiplicity of
$(\mathscr{C}_d)_k$ along $F_i$. Then $N'_i=N_i/\gcd(N_i,d)$, and
$F_i\to E_i$ is a ramified tame  Galois cover of degree
$\gcd(N_i,d)$.

If $E_i$ is principal, or  $E_i$ is a rational curve such that
$E_i\cdot \sum_{j\neq i}E_j=1$, then $ F_i\rightarrow E_i$ is an
isomorphism, and $N'_{i} = N_i$.  If $E_i$ is a rational curve
such that $E_i\cdot \sum_{j\neq i}E_j=2$, then $F_i\cong
\mathbb{P}^1_k$ and $ F_i \rightarrow E_i $ is either an
isomorphism, or
 ramified over the two points of $E_i\setminus E_i^o$.  Moreover, if $i$ and $j$ are distinct elements
of $I$, then over any point of $E_i\cap E_j$ lies exactly one
point of $F_i\cap F_j$.
\end{lemma}
\begin{proof} We've already seen in Proposition \ref{prop-normdesing}(1)
that $F_i$ is smooth, that $N'_i=N_i/\gcd(N_i,d)$ and that
$\mathscr{C}_d\times_{\mathscr{C}}E^o_i\to E^o_i$ is a tame Galois
cover of degree $\gcd(N_i,d)$.

If $E_i$ is principal, then $N_i$ divides $e(\mathscr{C})$, so
that $N_i$ is prime to $d$, $N'_i=N_i$ and $F_i\to E_i$ is an
isomorphism. Now assume that $E_i$ is a non-principal component;
in particular, it is rational.
 If $E_i$ is the only component of $\mathscr{C}_k$, then it follows
from \cite[3.1.1]{ni-saito} that $g=0$ and $\mathscr{C}_k=E_i$, so
that the result is clear. Thus we may suppose that $E_i$ meets
another component $E_j$ of $\mathscr{C}_k$. We choose a point $x$
in $E_i\cap E_j$.

  We claim that
 $\gcd(N_i,N_j,d)=1$. Point (2) in Proposition
 \ref{prop-normdesing}
 then implies that $F_i$ is a connected smooth rational curve. If $F_i\to E_i$ is not an isomorphism, one deduces from
 Hurwitz' theorem that $F_i\to E_i$ must be ramified over
 $E_i\setminus E_i^o$ and that $|E_i\setminus E_i^o|=2$.

It remains to prove our claim that $\gcd(N_i,N_j,d)=1$. We are in
one of the following three cases:
 \begin{enumerate}
 \item $E_i$ is part of a chain of
rational components $\mathscr{C}_k$ that meets a principal
component $E$ of $\mathscr{C}_k$, \item  $\mathscr{C}_k$ is a
chain of rational curves, \item $\mathscr{C}_k$ is a loop of
rational curves.
\end{enumerate}
  In case (1), the multiplicity $N$ of $E$
in $\mathscr{C}_k$ is prime to $d$ by our assumption that $d$ is
prime to $e(\mathscr{C})$. In case (2), one deduces from
 \cite[3.1.1]{ni-saito} that
$g=0$ and that the ends of the chain have multiplicity one. In
case (3), we must have $g=1$ by \cite[3.1.1]{ni-saito} and
$\Jac(C)$ has multiplicative reduction by
\cite[6.6]{liu-lorenzini-raynaud}. Then by our assumptions, $C$
has a rational point so that at least one of the components of
$\mathscr{C}_k$ has multiplicity one.
 In all cases, the argument in \cite[6.3]{Halle-stable} now shows
 that $\gcd(N_i,N_j,d)=1$.
\end{proof}

\begin{prop}\label{prop-data}  For each element $d$ of $\mathbb{N}'$
prime to $e(\mathscr{C})$, the combinatorial data of
$\mathscr{C}(d)_k $ only depend on the combinatorial data of
$\mathscr{C}_k $. In particular, they do not depend on the
characteristic exponent $p$ of $k$.
\end{prop}
\begin{proof}
This follows immediately from Proposition \ref{prop-normdesing}
and Lemma \ref{lemma-fiberdescription}, unless $C$ is a genus one
curve without rational point whose Jacobian has multiplicative
reduction. In fact, even that case is covered by the proof of
Lemma \ref{lemma-fiberdescription}, unless $\mathscr{C}_k$ is a
loop of rational curves. In that situation, the result follows
 easily from Proposition \ref{prop-normdesing}.
\end{proof}

\section{The N\'eron component series of a Jacobian}\label{sec-compseries}
In this Section, we study the N\'eron component series of a
Jacobian variety. This series was defined in Section
\ref{subsec-compseries} of Chapter \ref{chap-preliminaries}.
 We keep the notations from \eqref{sssec-notC}. We assume that $g$
 is positive, and we denote by $A$ the Jacobian $\Jac(C)$ of the curve $C$.
\subsection{Rationality of the component series}

\begin{prop}\label{prop-compfu}
Let $K'/K$ be a finite tame extension of $K$ whose degree $d$ is
prime to $e(C)$. Then
$$ |\Comp{A\times_K K'}|=d^{\trank{A}}\cdot |\Comp{A}|.$$
\end{prop}
\begin{proof}
 Embedding $K'$ in $K^s$, we may assume that $K=K(d)$.
 Let $\mathscr{C}$ be the minimal $sncd$-model of
$C$. By Theorem \ref{theo-winters} we can find a smooth, proper
and geometrically connected $\mathbb{C}((t))$-curve $D$ with
minimal $sncd$-model $\mathscr{D}$ over $\C[[t]]$ such that the
special fiber $\mathscr{D}_{\C}$  has the same combinatorial data
as $\mathscr{C}_k$. This implies that $e(C)=e(D)$.
 By Lemma \ref{lemma-fiberdescription}, the
combinatorial data of $\mathscr{D}(d)_k$ and $\mathscr{C}(d)_k$
coincide. Thus if we set $B=\Jac(D)$, then $\trank{A}=\trank{B}$,
$\Comp{A}\cong \Comp{B}$ and $ \Comp{A(d)}\cong \Comp{B(d)}$, so
that we may assume that $K=\C((t))$. Then $C$ is tamely ramified.
By Proposition \ref{prop-min}, the integer $e(C)$ equals the
degree of the minimal extension of $K$ over which $C$ acquires
semi-stable reduction, so that the result follows from
\cite[5.7]{HaNi-comp}.
\end{proof}
\noindent The following example, which was included already in
\cite[5.9]{HaNi-comp}, shows that in the statement of Proposition
\ref{prop-compfu}, we cannot replace $e(C)$ by the degree of the
minimal extension of $K$ where $C$ acquires semi-stable reduction.

\begin{example}\label{ex-comp}
We consider again the elliptic curve $C$ in Example \ref{ex-diff}.
Then $\Comp{C}=0$ because $C$ has reduction type $II$. On the
other hand, using Tate's algorithm one checks that $C(3)$ has
reduction type $I_0^*$ so that $|\Comp{C(3)}|=4$.
\end{example}

\begin{lemma}\label{lemma-torrank}
For every $d \in \mathbb{N}'$ prime to $e(C)$, we have that
$\trank{A(d)} = \trank{A}$.
\end{lemma}
\begin{proof}
 Let $\mathscr{C}/R$ be the minimal $sncd$-model of $C$. We
consider the graphs $\Gamma = \Gamma(\mathscr{C}_k)$ and
$\Gamma(d) = \Gamma(\mathscr{C}(d)_k)$, but we forget the weights.
From Proposition \ref{prop-normdesing} and
 Lemma \ref{lemma-fiberdescription}, we deduce that $\Gamma(d)$ is
obtained from $\Gamma$ by subdividing every edge $ \varepsilon$
 into a chain
 of $n_{\varepsilon}$ edges, for some $n_{\varepsilon}\geq 1$. This clearly implies that $Gamma$ and $\Gamma(d)$ have the same first Betti number, which means
that $\trank{A(d)} = \trank{A}$.
\end{proof}

\sss  Before we can state the main result in this section, we need
to recall the definition of the \emph{tame potential toric rank}
of $A$ which was introduced in \cite[6.3]{HaNi-comp}. This value
is defined by
$$ \ttame{A} = \max \{\trank{A(d)} \mid d \in \mathbb{N}' \}. $$
If we denote by $L$ the minimal extension of $K$ in $K^s$ such
that $A\times_K L$ has semi-abelian reduction, and by $d$ the
prime-to-$p$ part of $[L:K]$, then $\ttame{A}=t(A(d))$ by
\cite[6.4]{HaNi-comp}. In particular, if $A$ is tamely ramified, $
\ttame{A}$ equals the potential toric rank $\tpot{A}$ of $A$.
 In general, we have $\ttame{A}\leq
\tpot{A}$ by \cite[3.9]{HaNi-comp}.

\begin{theorem}\label{theorem-compser} Let $C/K$ be a smooth, projective and geometrically connected
curve of genus $g > 0$. Assume that $C$ admits a zero divisor of
degree one, and set $A=\Jac(C)$.  The component series
$$S^{\Phi}_A(T) = \sum_{d \in \N'} |\Comp{A(d)}| T^d $$ is rational.
More precisely, it belongs to the subring
$$ \mathscr Z = \Z \left[T, \frac{1}{T^j - 1} \right]_{j \in \Z_{>0}} $$
of $ \Z[[T]]$. It has degree zero if $p=1$ and $A$ has potential
good reduction, and negative degree in all other cases. Moreover,
$S^{\Phi}_A(T)$ has a pole at $T=1$ of order $\ttame{A} + 1$.
\end{theorem}
\begin{proof}
To ease notation, we put $ e = e(C) $ in this proof. Since the
tame case is covered in \cite[6.5]{HaNi-comp}, we can and will
assume here that $A$ is wildly ramified. This means in particular
that $p>1$ and $p \vert e$, by Proposition \ref{prop-min}.

Let us denote by $\mathscr{S}$ the set of divisors of $ e $ that
 are prime to $p$. We define the auxiliary series
$$ S'_A(T) = \sum_{d \in \mathbb{N}',\, \gcd(d,e)=1} |\Comp{A(d)}| T^d. $$
Then we can write
$$ S^{\Phi}_A(T) = \sum_{a \in \mathscr{S}} S'_{A(a)}(T^a), $$
since $ e(C(a)) = e(C)/a $ by Corollary \ref{cor-e(C)}.

By Lemma \ref{lemma-torrank}, we have $ \ttame{A} = \max
\{\trank{A(a)} \mid a \in \mathscr{S}\} $. Using
\cite[6.1]{HaNi-comp}, we see that it is enough to prove the
following claims, for every $a$ in $\mathscr{S}$:
\begin{enumerate}
\item the series $ S'_{A(a)}(T^a) $ belongs to $ \mathscr Z $ and
has a pole at $ T=1$ of order $\trank{A(a)} + 1$. \item the degree
of $ S'_{A(a)}(T^a) $ is  negative.
\end{enumerate}

 To prove these claims, it suffices to consider the case $a=1$.
 We denote by $ \mathscr P_{e} $ the set of elements in $\{1,
\ldots, e \} $ that are prime to $e$. Since $p$ divides $e$, we
can write
\begin{eqnarray*} S'_A(T)& =& |\Comp{A}| \cdot \sum_{d \in
\mathbb{N}',\, \gcd(d,e)=1} d^{\trank{A}} T^d  \\ &=& |\Comp{A}|
\cdot \sum_{b \in \mathscr P_{e}} \left( \sum_{q \in \N}
(qe+b)^{t(A)} T^{qe+b} \right)\end{eqnarray*} where the first
equality follows from Proposition \ref{prop-compfu}. An easy
computation \cite[6.2]{HaNi-comp} shows that for each $b \in
\mathscr P_{e}$,
 the series $ \sum_{q \in \N} (qe+b)^{t(A)} T^{qe+b} $  belongs to $
\mathscr Z $, has a pole at $T=1$ of order $\trank{A} + 1$, and
has negative degree. This clearly implies claim (2), and claim (1)
now follows from \cite[6.1]{HaNi-comp}.
\end{proof}

\sss  We expect that Theorem \ref{theorem-compser} is valid for
all abelian $K$-varieties $A$. If $A$ is tamely ramified or $A$
has potential multiplicative reduction, we proved this in
\cite[6.5]{HaNi-comp}. The crucial open case is the one where $A$
is a wildly ramified abelian variety with potential good
reduction; in that case, it is not clear how to control the
behaviour of the $p$-part of $\Comp A$ under finite tame
extensions of $K$ (see Example \ref{ex-comp}). If $A$ is tamely
ramified and $A$ has potential good reduction, then the $p$-part
of $\Comp A$ is trivial by Theorem 1 in \cite{ELL}.

\section{Appendix: Locally toric rings}\label{sec-loctor}
\subsection{Resolution of locally toric singularities}

\begin{definition}\label{def-loctor}
Let $A$ be a complete Noetherian local  $R$-algebra. We say that
$A$ is a locally toric ring if it is of the form
$$ A \cong R[[S]]/ \mathscr{I} = R[[\chi^s : s \in S]]/ \mathscr{I}, $$
where $S$ is a finitely generated integral saturated sharp monoid
and the ideal $\mathscr{I}$ is generated by an element of the form
$ \pi - \varphi $ with
$$ \varphi \equiv 0 \mod (\chi^s : s \in S, s \neq 0). $$
\end{definition}

We have made a slight modification of Definition 2.6.2 in
\cite{Kir} in order to treat the cases where $R$ is of equal,
resp.~mixed characteristic  in a uniform way.
 The condition that $S$ is sharp seems to be missing in \cite{Kir}
 (without this condition, the ring $R[[S]]$ is not well-defined).
 The condition that $S$ is saturated is equivalent to the
 assumption in \cite[2.6.2]{Kir} that $A$ is normal. The condition
 that $S$ is finitely generated, integral and sharp
 implies that the abelian group $S^{\mathrm{gp}}$ is a free
 $\Z$-module of finite rank and that the natural morphism $S\to
 S^{\mathrm{gp}}$ is injective. Thus $S$ is a submonoid of a free
 $\Z$-module of finite rank, as required in \cite[2.6.2]{Kir}.
 Definition \ref{def-loctor} is equivalent to Kato's definition of
 a toric singularity in \cite[2.1]{kato}, by \cite[3.1]{kato}.

\sss  We shall only be concerned here with the case where $A$ has
Krull dimension $2$; then $S^{\mathrm{gp}}$ has rank two. By
\cite[27.3]{Lip} there exists a minimal desingularization $ \rho :
X \to \Spec A$ which, if $A$ is locally toric, turns out to be
entirely determined by the monoid $S$. We explain briefly the main
ingredients in this procedure, following the treatment in
\cite[Ch.\,2]{Kir}.

\sss  Set $M=S^{\mathrm{gp}}$ and denote by $ N $ the dual lattice
$ M^{\vee} = \Hom(M,\Z) $. We put $ M_{\R} = M \otimes_{\Z} \R $
and $ N_{\R} = N \otimes_{\Z} \R $. The monoid $S$ yields a cone $
\sigma^{\vee} = S \otimes_{\N} \R_{\geq 0} \subset M_{\R} $ and a
dual cone $ \sigma = (\sigma^{\vee})^{\vee} \subset N_{\R} $.
 There exists a fan $ F = \{\sigma_i\}_{i \in I_F} $ of cones in
$N_{\R} $ that is a regular refinement of $\sigma$.
 Moreover, since $M$ has rank two,
there exists a \emph{minimal} regular refinement $F_{\min}$, in
the sense that any other regular refinement of $\sigma$ is also a
refinement of $ F_{\min} $.

\sss  For each $ \sigma_i \in F_{\min} $, put $ S_{\sigma_i} =
(\sigma_i)^{\vee} \cap M $ and define $ A_{\sigma_i} =
A[S_{\sigma_i}] $ (the notation is ambiguous; this is not the
monoid $A$-algebra associated to $S_{\sigma_i}$, but the subring
of the quotient field of $A$ obtained by joining to $A$ the
elements $\chi^s$ with $s$ in $S_{\sigma_i}$). Note that $ A_{\sigma_i} $ will in general not be local.

\begin{lemma}\label{lemma-torreg}
Let $ \sigma_i \subset N_{\R} $ be part of a regular refinement of $\sigma$. Then $ A_{\sigma_i} $ is regular.
\end{lemma}
\begin{proof}
 This follows from \cite[2.6.8]{Kir}, but the proof that is given
 there is not complete (one cannot directly carry over the methods
 from the classical theory of toric varieties over an algebraically closed
 field, since the equation $\pi-\varphi$ is not homogeneous with respect to the torus
 action). 
 Instead, one can invoke the logarithmic approach in
 \cite[10.4]{kato}. We will give below a proof by explicit
 computation for tame cyclic quotient singularities (Proposition \ref{prop-mindesing}); this is the only case we used in this chapter.
\end{proof}
\if false
 We first prove regularity of $R[[S]][S_{\sigma_i}]$.
This is equivalent to regularity of $R[[S]]\{S_{\sigma_i}\}$, by
which we mean the completion of $R[[S]][S_{\sigma_i}]$ with
respect to the ideal $J=(\pi,\chi^s, s\in S)$, because $J$ is
contained in all the maximal ideals of $R[[S]][S_{\sigma_i}]$. But
$R[[S]]\{S_{\sigma_i}\}$ is also the completion of
$R[S][S_{\sigma_i}] = R[S_{\sigma_i}] $ with respect to the ideal
$J'=J\cap R[S_{\sigma_i}]$,  and $R[S_{\sigma_i}]$ is regular
(this $R$-algebra is isomorphic to a polynomial ring $R[x,y]$ in
two variables $x=\chi^{s_1}$, $y=\chi^{s_2}$ by regularity of
$\sigma_i$).

Thus, to see that $ A_{\sigma_i} $ is  regular, it is enough to
show that $\pi-\varphi$ is not contained in the square of any
maximal ideal of $R[[S]][S_{\sigma_i}]$; then it is part of a
regular system of parameters at every point where it vanishes.

 Every maximal ideal of $R[[S]][S_{\sigma_i}]$ contains the ideal
$J$, and the ring morphism
$$\theta:R[S_{\sigma_i}]/(J')^2\to R[[S]][S_{\sigma_i}]/J^2$$ is an
isomorphism.
 The residue class of  $\pi-\varphi$ modulo $J^2$ is of the form
$$\theta(\pi- \mod (J')^2).$$
 were
  contained in the square of a maximal ideal of
  $R[[S]][S_{\sigma_i}]$, then it would also be contained in the
  square of a maximal ideal of $R[S_{\sigma_i}]=R[x,y]$.
\fi

\sss The natural maps $ \Spec A_{\sigma_i} \to \Spec A $ are birational,
and glue to give a desingularization
$$ \rho_{F_{\min}} : X_{F_{\min}} = \cup_i \Spec A_{\sigma_i} \to \Spec A. $$
The minimality of $F_{\min} $ implies that $\rho_{F_{\min}}$ is
the minimal toric desingularization of $\Spec A$. In fact, it is
also the minimal resolution: this can be seen, for instance, by
computing the self-intersection numbers of the exceptional curves
in the resolution. In the case we're interested in, these
self-intersections will be at most $-2$, so that none of the
exceptional curves in the resolution satisfies Castelnuovo's
contractibility criterion and the resolution is indeed minimal
(see Proposition \ref{prop-mindesing}).

\subsection{Tame cyclic quotient singularities}\label{subsec-tamecyclic}
\sss  We consider the regular local ring
\begin{equation*}
A = R[[t_1,t_2]]/(\pi - w \cdot t_1^{m_1} t_2^{m_2}),
\end{equation*}
where $ w$ is a unit in  $ R[[t_1,t_2]] $. Let $ n \in \mathbb{N}'
$ be an integer prime to $m_1$ and $m_2$, and let $r$ be the
unique integer $ 0 < r < n $ such that $ m_1 + r m_2 \equiv 0 $
modulo $n$.

  We let $ \mu_n(k) $ act on the $R$-algebra $ R[[t_1,t_2]] $
  by $ (\xi,t_1) \mapsto \xi t_1 $ and $ (\xi,t_2) \mapsto \xi^r t_2 $ for any $ \xi \in \mu_n(k) $. We also assume
  that the unit $ w $ is $\mu_n(k)$-invariant, so that there is an induced $\mu_n(k)$-action on $A$.

\begin{lemma}\label{lemm-torsing}
We set
$$ M = \{(s_1,s_2) \in \Z^2 \,|\, n \mbox{ divides } s_1 + r s_2 \} = (n,0) \Z + (-r,1) \Z \subset \Z^2 $$
and $ S = M \cap (\Z_{\geq 0})^2 $. Then $B = A^{\mu_n(k)}$ is a
locally toric ring with respect to the monoid $S$. More precisely,
$$ B \cong R[[t_1^{s_1} t_2^{s_2} \,|\, (s_1,s_2) \in S]]/(\pi - w \cdot t_1^{m_1} t_2^{m_2}). $$
\end{lemma}
\begin{proof}
Since $n$ is prime to $p$, we have an isomorphism of $R$-algebras
$$B\cong R[[t_1,t_2]]^{\mu_n(k)} /(\pi - w
\cdot t_1^{m_1} t_2^{m_2}).$$ The lemma then follows from the fact
that the $R$-algebra $ R[[t_1,t_2]]^{\mu_n(k)} $ is topologically
generated by invariant monomials, which are exactly the elements
$t_1^{s_1}t_2^{s_2}$ with $(s_1,s_2) \in S$.
\end{proof}

\sss The $R$-scheme $\Spec B$ is called a {\em tame cyclic
quotient singularity}. Thanks to Lemma \ref{lemm-torsing}, we can
use the theory of locally toric singularities to construct the
minimal resolution of $\Spec B$. First, we need to introduce some
notation.

\begin{definition}\label{def-HJ}
For relatively prime integers $ a,b \in \mathbb{Z}_{>0} $ with $ a
> b $ we use the compact notation  $ a/b = [z_1, \ldots, z_L]_{\mathrm{HJ}}
$ for the Hirzebruch-Jung continued fraction expansion
$$a/b=z_1-\frac{1}{z_2-\frac{1}{\ldots-\frac{1}{z_L}}}$$
with $z_i\in \mathbb{Z}_{\geq 2}$.
\end{definition}

\begin{prop}\label{prop-mindesing}
Let $ \rho : X = X_{F_{\min}} \to \Spec B $ be the minimal
 toric desingularization of $\Spec B$. The special fiber $ X_k = \sum_{i=0}^{L+1}
\mu_i E_i $ is a strict normal crossings divisor and, renumbering
the irreducible components of $X_k$ in a suitable way, we can
arrange that the following properties hold:
\begin{enumerate}
\item $E_0$ and $E_{L+1}$ are the strict transforms of the
irreducible components of the special fiber of $\Spec B$.

\item $E_i \cong \mathbb{P}^1_k$ for each $ 1 \leq i \leq L $.

\item For each $ 1 \leq i \leq L $, $E_i$ intersects $E_{i-1}$ and
$E_{i+1}$  in a unique point, and no other components of $X_k$.
Moreover, $E_0$ (resp.~$E_{L+1}$) intersects $E_1$ (resp.~$E_L$)
 in a unique point, and no other components of $X_k$.

\item $ E_i^2 \leq - 2 $ for each $ 1 \leq i \leq L $.
\end{enumerate}
In particular, $\rho$ is the minimal resolution of $\Spec B$.
\end{prop}
\begin{proof}
Let $n$ and $r$ be as above, and write $ n/r =
[b_1,\ldots,b_L]_{\mathrm{HJ}} $. The
 monoid $S$ from Lemma \ref{lemm-torsing} defines the lattice
  $$ N= (0,1) \mathbb{Z} + \frac{1}{n}(1,r) \mathbb{Z}, $$
 which is dual to $M=S^{\mathrm{gp}}$, and the cone $\sigma=\R^2_{\geq 0}$ in $N\otimes_{\Z}\R$.
  The minimal regular refinement
of $ \sigma $ is the union of subcones $ \sigma_{0}, \ldots,
\sigma_L $ that can be computed as in \cite[\S2.6]{fulton}.
 If we put $ e_0 = (0,1) $, $e_1 = \frac{1}{n}(1,r)$, and inductively define $ e_{i+1} = b_i e_i - e_{i-1} $
  (in particular $ e_{L+1} = (1,0) $), then the cone $ \sigma_i $ is generated by $ e_i$ and $e_{i+1}$.

We put $ u_0 = t_1^n $ and $ v_0 = t_1^{-r} t_2 $, and define
inductively elements $ u_i = (v_{i-1})^{-1} $  and $ v_i = u_{i-1}
(v_{i-1})^{b_i} $ in the fraction field of $B$, for $ 1 \leq i
\leq L $. We moreover put $ \mu_0 = m_2 $ and $ \mu_1 =
(m_1+rm_2)/n $ and define inductively integers $ \mu_{i+1} = b_i
\mu_i - \mu_{i-1} $. In particular $\mu_i \in \NN $ for all $i$
and $\mu_{L+1} = m_1$. Moreover, an easy induction argument shows
that $ t_1^{m_1} t_2^{m_2} = u_i^{\mu_{i+1}} v_i^{\mu_i} $ for all
$i$.

 We identify $\Z^2$ with the group of monomials in the variables $(t_1,t_2)$ via the map $(a_1,a_2)\mapsto
 t_1^{a_1}
 t_2^{a_2}$.
 Then for each $ i \in \{0, \ldots, L \} $, the elements $u_i$ resp.~$v_i$ are dual to $e_i$ resp.~$ e_{i+1} $ and generate the monoid $S_{\sigma_i}$. Moreover, we get
$$ B_{\sigma_i} = R[[S]][S_{\sigma_i}]/(\pi - w \cdot u_i^{\mu_{i+1}} v_i^{\mu_i}). $$

 First, we show that $\Spec B_{\sigma_i}$ is regular.
 The scheme $\Spec B$ has an isolated singularity at its unique
 closed point $x$, and $\rho$ is an isomorphism over the
 complement of this closed point. Thus we only need to prove
 regularity of $\Spec B_{\sigma_i}$ at the points lying above $x$.
 In such points we either have $u_i=0$ or $v_i=0$, so that the
 result follows from the fact that $\Spec B_{\sigma_i}$ is two-dimensional and the zero loci of $u_i$ and
 $v_i$ are regular one-dimensional schemes. More precisely, the
 zero locus of $v_i$ is isomorphic to $\Spec k[[u_0]]$ if $i=0$
 and to $\Spec k[u_i]$ else; likewise, the zero locus of $u_i$ is
 isomorphic to $\Spec k[[v_L]]$ if $i=L$ and to $\Spec k[v_i]$
 else.

  The ideal $(u_i,v_i)$ is maximal in $B_{\sigma_i}$, so that
  $u_i,v_i$ form a regular system of parameters and the special
  fiber of $\Spec B_{\sigma_i}$ is a divisor with strict normal
  crossings. Thus we can conclude that $X$ is regular and that
  $X_k$ is a strict normal crossings divisor of the form described in (1) and (3).   The relation $ u_i = (v_{i-1})^{-1} $
implies that the components $E_1,\ldots,E_{L}$ are all isomorphic
to $\mathbb{P}^1_k$, so that (2) holds as well. Moreover, the
component $E_j$ has multiplicity $\mu_j$ in $X_k$, for each $j$ in
$\{0,\ldots,L+1\}$, so that
$$0=X_k\cdot E_i=(\sum_{j=0}^{L+1}\mu_j E_j)\cdot
E_i=\mu_{i-1}+\mu_iE_i^2+\mu_{i+1}$$ for all $1\leq i\leq L$. This
implies that the self-intersection number of $E_i$ is equal to
$-b_i$. In particular, each of these exceptional components has
self-intersection number at most $-2$, such that (4) holds and
$\rho$ is minimal.
\end{proof}

\if false
 We first note that $ C_{\sigma_i} = R[[S]][S_{\sigma_i}]
$ is regular.  To see this, let $\frak{m}_{\sigma_i}$ denote the
maximal ideal of $C_{\sigma_i}$ generated by $\pi$, $u_i$ and
$v_i$. The completion of $ R[[S]][S_{\sigma_i}] $ at
$\frak{m}_{\sigma_i}$ equals $ R[[S_i]] = R[[u_i,v_i]] $, thus it
is regular. This implies that  $ C_{\sigma_i} $ is regular at
$\frak{m}_{\sigma_i}$, and thus regular everywhere because the
torus $\mathbb{G}_{m,R}$ acts transitively on $\Spec
C_{\sigma_i}$.

It follows that the completion of $B_{\sigma_i}$ in $M_{\sigma_i}$
is $R[[u_i,v_i]]/(\pi - w \cdot u_i^{\mu_{i+1}} v_i^{\mu_i})$. In
particular, the special fiber is a normal crossings divisor at the
origin of each chart.

Now we look more specifically at  the special fiber in each chart.
We will show, for any $ 1 \leq i \leq L $ that $
B_{\sigma_{i-1}}/(u_{i-1}) \cong k[v_{i-1}] $ and $
B_{\sigma_i}/(v_i) \cong k[u_i] $. Because we have the relation $
u_i = (v_{i-1})^{-1} $ this glues to a $ \mathbb{P}^1_k$.

We will proceed inductively, starting with $ B_{\sigma_{0}} $. Note that $S_0$ is the monoid consisting of monomials $u_0^iv_0^j$ with $(i,j) \in S_0 = (\mathbb{Z}_{\geq0})^2 $, and $S$ is the submonoid $S_0$ consisting of those $(i,j) $ such that $ 0 \leq j \leq  (n/r)i $ holds. Introducing dummy variables $u,v$, we see that we can describe $ B_{\sigma_{0}} $ as follows
$$ B_{\sigma_{0}} = R[[S]][S_0]/(\pi - w \cdot u^{\mu_1} v^{\mu_0}, u - u_0, u_0 v_0 - uv). $$
Thus, dividing out $u_0$ yields the ring
$$ B_{\sigma_{0}}/(u_0) = k[[S]][v_0]/(u, uv, u^{\mu_1} v^{\mu_0}) $$
and we would like to see that this is isomorphic to $ k[v_0] $.

Having dealt with the $i=0$ case, consider now $B_{\sigma_1}$. Here $ u_1 = v_0^{-1} = t_1^r t_2^{-1} $ and
$$ v_1 = u_0 v_0^{b_1} = t_1^{n - rb_1} t_2^{b_1} = t_1^{-r_1} t_2^{b_1}. $$
We either have $r=1$, $b_1=n$ and $r_1=0$, in which case $ v_1 \in S $. Otherwise, $r>1$, $r_1 \geq 1$ and $ v_1 \notin S $. In either case, $ u_1 \notin S $.

This shows that the special fiber is a strict normal crossings
divisor. Assertions (1) - (4) now follow immediately from this
description, and since the self intersections are the integers
$b_i$, it is clear that this resolution of singularities is
minimal. \fi

\chapter{Component groups and non-archimedean
uniformization}\label{chap-uniform}
 In this chapter, we will study
the behaviour of the torsion part of the N\'eron component group
of a semi-abelian $K$-variety under ramified extension of the base
field $K$. Our main goal is to prove the rationality of the
component series (Theorem \ref{thm-compsersab}). We discussed the
case of an abelian $K$-variety in \cite{HaNi-comp}; in that case,
the component group is finite. The main complication that arises
in the semi-abelian case is the fact that it is difficult in
general to identify the torsion part of the component group in a
geometric way. This problem is related to the {\em index} of the
semi-abelian $K$-variety, an invariant that we introduce in
Section \ref{subsec-index}. For tori, the torsion part of the
component group has a geometric interpretation in terms of the
dual torus, and we can explicitly
compute the index from the character group.
  ~The case of a semi-abelian variety is
substantially more difficult; there we need to construct a
suitable uniformization, which is no longer an algebraic group but
a rigid analytic group. In order to deal with N\'eron component
groups of rigid analytic groups, we will use the cohomological
theory of Bosch and Xarles \cite{B-X}, that we recall and extend
in Section \ref{sec-BX}. We correct an error in their paper, which
was pointed out by Chai, and we show that all of the principal
results in \cite{B-X} remain valid.

 \section{Component groups of smooth sheaves}\label{sec-BX}
\subsection{The work of Bosch and Xarles}
\sss In \cite{B-X}, Bosch and Xarles developed a powerful
cohomological approach to the study of component groups of abelian
$K$-varieties. They interpret the N\'eron model in terms of a
push-forward functor from the rigid smooth site on $\Sp K$ to the
formal smooth site on $\Spf R$, and they show that the N\'eron
component group of an abelian variety can be recovered from the
smooth sheaf associated to the N\'eron model. Combining this
interpretation with non-archimedean uniformization of abelian
varieties, they deduce several deep results on the structure of
the component group.

\sss
 Unfortunately, it is known that the perfect residue field case of Lemma 4.2  in
\cite{B-X} is not correct (see \cite[4.8(b)]{chai}), and the
proofs of some of the main results in \cite{B-X}
 rely on this lemma.
 We will now show that one can replace the
 erroneous lemma by another statement that suffices to prove the
 validity of all the other results in \cite{B-X}.

\begin{prop}\label{prop-bx}
Assume that $k$ is algebraically closed. Let
$$0\longrightarrow T\longrightarrow G\longrightarrow H\longrightarrow 0$$ be a short exact sequence of
semi-abelian $K$-varieties such that $T$ is a torus.
 We denote by
$f:\mathscr{G}\rightarrow \mathscr{H}$ the unique morphism of
N\'eron $lft$-models that extends the morphism $G\rightarrow H$.
 Then the following properties hold.
\begin{enumerate}
\item The map
$$f^o(R):\mathscr{G}^o(R)\rightarrow \mathscr{H}^o(R)$$ is
surjective. \item The sequence of component groups
\begin{equation}\label{eq-phi}
\Comp T\rightarrow \Comp G\rightarrow \Comp H\rightarrow
0\end{equation} is exact.
\end{enumerate}
\end{prop}
\begin{proof}
(1) By \cite[4.3]{chai}, the Galois cohomology group $H^1(K,T)$
vanishes, so that the sequence
\begin{equation}\label{eq-chai}0\rightarrow T(K)\rightarrow
G(K)\rightarrow H(K)\rightarrow 0\end{equation} is exact.
 In particular, the map
 $$G(K)=\mathscr{G}(R)\to \mathscr{H}(R)=H(K)$$ is surjective. Since $\mathscr{G}(R)/\mathscr{G}^o(R)$ is finitely generated by
 \cite[3.5]{HaNi}, it
 follows from \cite[9.6.2]{neron} that
 $$f^o(R):\mathscr{G}^o(R)\rightarrow \mathscr{H}^o(R)$$ is
 surjective.

\medskip
\noindent (2)  It follows at once from the exactness of
\eqref{eq-chai} that $\Comp G\to \Comp H$ is surjective. Then by
an elementary diagram chase one deduces from point (1) that the
sequence
$$\Comp T=T(K)/\mathscr{T}^o(R)\rightarrow \Comp G=G(K)/\mathscr{G}^o(R) \rightarrow \Comp H=H(K)/\mathscr{H}^o(R)\rightarrow 0 $$
is exact.
\end{proof}

\sss \label{sss-BXcorrect} Let us check that replacing
\cite[4.2]{B-X} (perfect residue field case) by Proposition
\ref{prop-bx} suffices to prove all the subsequent results in
\cite{B-X} (note that one can immediately reduce to our setting
where $R$ is complete and $k$ is separably closed, since the
formation of N\'eron models commutes with base change to the
completion of a strict henselization). The result \cite[4.2]{B-X}
was applied at the following places.
\begin{itemize}
\item In the proof of \cite[4.11(i)]{B-X} and \cite[5.8(ii)]{B-X},
the result in \cite[4.2]{B-X} was used to prove the exactness of
the sequence
$$\begin{CD}\Comp{T_{\spl}}@>>> \Comp{T}@>>> \Comp{T'}@>>> 0 \end{CD}$$
where $T$ is a $K$-torus and $T'=T/T_{\spl}$.
 This also follows from the proof of \cite[10.1.7]{neron}, since
$T'$ is anisotropic. If $k$ is perfect, it follows from
Proposition \ref{prop-bx}.
 \item The proof of \cite[4.11(ii)]{B-X} (perfect residue field case); this result
is a special case of Proposition \ref{prop-bx}(2).
\end{itemize}

\subsection{Identity component and component group of a smooth sheaf}
\sss  We denote by $(\Spf R)_{\sm}$ and $(\Sp K)_{\sm}$ the small
smooth sites over $\Spf R$ and $\Sp K$, respectively
\cite[\S3]{B-X}. For every smooth affinoid $K$-algebra $A$ and
every
 abelian sheaf $F$ on $(\Sp K)_{\sm}$, we will write $F(A)$
 instead of $F(\Sp A)$, and we will use the analogous notation for
 smooth topological $R$-algebras of finite type and abelian
 sheaves on $(\Spf R)_{\sm}$.

\sss \label{sss-smlift} An important property of the site $(\Spf
R)_{\sm}$ is the following: if $\mX$ is a non-empty smooth formal
$R$-scheme, then the structural morphism  $\mX\to \Spf R$ has a
section. This follows from the infinitesimal lifting criterion for
smoothness and the fact that
 the $k$-rational points on $\mX\times_R k$ are dense because $k$ is separably closed
\cite[2.2.13]{neron}.

\sss \label{sss-neronsheaf} The generic fiber functor that
associates to a (smooth) formal $R$-scheme its (smooth) rigid
generic fiber over $K$ induces a morphism of sites
$$j:(\Sp K)_{\sm}\to (\Spf R)_{\sm}.$$ For every abelian sheaf $F$ on
$(\Sp K)_{\sm}$, we define the N\'eron model $\mathscr{F}$ of $F$
by
$$\mathscr{F}=j_* F$$ as in \cite[\S3]{B-X}. This is an abelian sheaf on $(\Spf R)_{\sm}$.

\sss Bosch and Xarles define in \cite[\S4]{B-X} the {\em identity
component} $\mathscr{F}^o$ of an abelian sheaf $\mathscr{F}$ on
$(\Spf R)_{\sm}$, which is a subsheaf of $\mathscr{F}$, and the
{\em component sheaf}
$$\Comp{\mathscr{F}}=\mathscr{F}/\mathscr{F}^o.$$ The identity component and component group are functorial in
$\cF$. The sheaf property of $\cF$ is never used in the
construction of the identity component, so that we can immediately
extend this definition to abelian presheaves on $(\Spf R)_{\sm}$,
as follows.

\sss Let $\cF$ be an abelian presheaf on $(\Spf R)_{\sm}$.  We
define $\mathscr{F}^o(R)$ as the subgroup of $\mathscr{F}(R)$
consisting of elements $\sigma$ such that there exists a smooth
connected formal $R$-scheme $\mX$, an element $\tau$ in
$\mathscr{F}(\mX)$ and points $x_0$ and $x_1$ in $\mX(R)$ such
that $x_0^*\tau=0$ and $x_1^*\tau=\sigma$ in $\mathscr{F}(R)$.
 If $\mY$ is any smooth formal
$R$-scheme, then an element of $\mathscr{F}(\mY)$ belongs to
$\mathscr{F}^o(\mY)$ if and only if its image in $\mathscr{F}(R)$
lies in $\mathscr{F}^o(R)$ for every $R$-morphism $\Spf R\to \mY$.
We will say that $\cF$ is connected if $\cF^o=\cF$. If $\cF$ is a
sheaf, then so is $\cF^o$.

\begin{prop}\label{prop-compcrit}
If $\mathscr{F}$ is an abelian presheaf on $(\Spf R)_{\sm}$, $\mX$
is a connected smooth formal $R$-scheme and $x$ is an element of
$\mX(R)$, then an element $\sigma$ of $\cF(\mX)$ belongs to
$\cF^o(\mX)$ if and only if $x^*\sigma$ belongs to $\cF^o(R)$.
\end{prop}
\begin{proof} The ``only if'' part is a direct consequence of the
definition of the identity component $\cF^o$, so that it suffices
to prove the converse implication. Let $y$ be a point of $\mX(R)$.
We must show that $y^*\sigma$ lies in $\cF^o(R)$. Denote by
$$f:\mX\to \Spf R$$ the structural morphism, and set
$$\sigma_0=\sigma-f^*x^*\sigma\quad \in \cF(\mX).$$ Then
$x^*\sigma_0=0$, so that $y^*\sigma_0$ lies in $\cF^o(R)$. But
$$y^*\sigma_0=y^*\sigma-x^*\sigma$$ and since $x^*\sigma$ lies in
$\cF^o(R)$, we find that $y^*\sigma$ lies in $\cF^o(R)$, as well.
\end{proof}

\begin{prop}\label{prop-compgr}
For every abelian sheaf $\mathscr{F}$ on $(\Spf R)_{\sm}$, the
component sheaf $\Comp{\mathscr{F}}$ is the constant sheaf on
$(\Spf R)_{\sm}$ associated to the abelian group
 $$\mathscr{F}(R)/\mathscr{F}^o(R).$$
\end{prop}
\begin{proof}
Let $\mX$ be a non-empty connected smooth formal $R$-scheme. We
will show that the morphism
$$i:\mathscr{F}(R)/\mathscr{F}^o(R)\to
\mathscr{F}(\mX)/\mathscr{F}^o(\mX)$$ induced by the structural
morphism $f:\mX\to \Spf R$ is an isomorphism.
 Since $\mX$ is smooth and $k$ is separably closed, the morphism $f$ has a
 section $x:\Spf R\to \mX$, which induces a section
 $$s:\mathscr{F}(\mX)/\mathscr{F}^o(\mX)\to
 \mathscr{F}(R)/\mathscr{F}^o(R)$$ of $i$.
 If $\tau$ is an
 element of $\mathscr{F}(\mX)$, then $$x^*(\tau-f^*x^*\tau)=0$$ in
 $\mathscr{F}(R)$, so that $\tau-f^*x^*\tau$ must lie in
 $\mathscr{F}^o(\mX)$. This implies that
 $$\tau \!\!\mod \mathscr{F}^o(\mX)=(i\circ s)(\tau \!\!\mod \mathscr{F}^o(\mX)).$$
 Thus $s$ is inverse to $i$.
\end{proof}

\sss With a slight abuse of notation, we will usually write
$\Comp{\mathscr{F}}$ for the abelian group
 $\mathscr{F}(R)/\mathscr{F}^o(R)$. When $F$ is an abelian sheaf
 on $(\Sp K)_{\sm}$, we write $\Comp{F}$ for $\Comp{j_*F}$.

\sss If $G$ is a smooth commutative rigid $K$-group, then the
 associated presheaf  on $(\Sp
K)_{\sm}$ is a sheaf \cite[3.3]{B-X}, which we'll denote again by
$G$. If $G$ admits a formal N\'eron model $\mathscr{G}$ in the
sense of \cite{bosch-neron}, then $\mathscr{G}$ represents the
N\'eron model $j_* G$ on $(\Spf R)_{\sm}$, and $\mathscr{G}^o$
represents the identity component $(j_* G)^o$. It follows that the
 component group $\Comp{G}$ of the abelian sheaf $G$ is
 canonically isomorphic to the group
 $\mathscr{G}_k/\mathscr{G}^o_k$ of connected components of
 $\mathscr{G}_k$.

\subsection{Some basic properties of the component group}
\begin{lemma}\label{lemm-smtop}\item
\begin{enumerate}
\item
 If $\cF$ is an abelian presheaf on $(\Spf R)_{\sm}$ and
$\cF\to \cF'$ is a sheafification, then $\cF(R)\to \cF'(R)$ is an
isomorphism. \item The functor $\cF\to \cF(R)$ from the category
of abelian sheaves on $(\Spf R)_{\sm}$ to the category of abelian
groups is exact.
\end{enumerate}
\end{lemma}
\begin{proof}
(1) Injectivity of $\cF(R)\to \cF'(R)$ follows immediately from
the fact that every surjective smooth morphism of formal schemes
$\mX\to \Spf R$ has a section, so that an element of $\cF(R)$
vanishes as soon as it vanishes on some smooth cover of $\Spf R$.
It remains to prove surjectivity. Any element $\sigma'$ of
$\cF'(R)$ can be represented by an element
 $\sigma$ of $\cF(\mX)$ where $\mX$ is a non-empty smooth formal
 $R$-scheme and $\sigma$ satisfies the gluing condition with
 respect to the smooth cover $\mX\to \Spf R$. If $x$ is any point in
 $\mX(R)$, then the image of $x^*(\sigma)\in \cF(R)$  by
 the morphism $\cF(R)\to \cF'(R)$ equals $\sigma'$.

(2) Taking sections of an abelian sheaf  always defines a left
exact functor; right exactness
  is proven by applying (1) to the
  sheafification of the image presheaf of a surjective morphism of sheaves.
\end{proof}

\begin{prop}\label{prop-sheafify} Let $\cF$ be an abelian presheaf on $(\Spf R)_{\sm}$, and denote
by $\cF\to \cF'$ its sheafification. \begin{enumerate}  \item If
$\cF$ is connected, then so is $\cF'$. \item The identity
component $(\cF')^o$ is the sheafification of $\cF^o$. \item The
component sheaf $\Comp{\cF'}$ is the sheafification of the
quotient presheaf $\cF/\cF^o$.
\end{enumerate}
\end{prop}
\begin{proof}
(1) We know by Lemma \ref{lemm-smtop} that the morphism $\cF(R)\to
\cF'(R)$ is surjective, so that this is a direct consequence of
Proposition \ref{prop-compgr} and the functoriality of the
identity component.

(2) The sheafification $\cG$ of $\cF^o$ is a subsheaf of $\cF'$,
and by $(1)$, it is contained in $(\cF')^o$. We will show that
$\cG=(\cF')^o$.
%
%

\smallskip
{\em Step 1.} Let $\mX$ be a connected smooth formal $R$-scheme,
 let $\sigma$ be an element in $\cF(\mX)$ and assume that the
 image $\tau$ of $\sigma$ in $\cF'(\mX)$ lies in $(\cF')^o(\mX)$.
 Assume, moreover, that there exists a point $x$ in $\mX(R)$ such
 that $x^*\tau$ lies in $\cG(R)$.
 We will prove that $\tau$
 lies in $\cG(\mX)$.

By Lemma \ref{lemm-smtop}, we can find
 an element $\rho$ in $\cF^o(R)$ such that the image of $\rho$ in
 $(\cF')^o(R)$ is equal to $x^*\tau$.
We denote by $h$ the structural morphism $\mX\to \Spf R$, and we
set $$\sigma_0=h^*(x^*\sigma-\rho)$$ in $\cF(\mX)$.  Then
$\sigma_0$ lies in the kernel of $\cF\to \cF'$, so that we may
assume that $x^*\sigma\in \cF^o(R)$ by replacing $\sigma$ by
$\sigma-\sigma_0$. This implies that $\sigma$ lies in
$\cF^o(\mX)$, by Proposition \ref{prop-compcrit}, so that $\tau$
lies in $\cG(\mX)$.

\smallskip
{\em Step 2.} Let $\mX$ be a connected smooth formal $R$-scheme,
and let $\tau$ be an element of $(\cF')^o(\mX)$. Assume that there
exists a point $x_1$ in $\mX(R)$ such that $x_1^*\tau$ lies in
$\cG(R)$. We will prove that $\tau$ lies in $\cG(\mX)$.

 Let $x_2$ be a point of $\mX(R)$. For each $i$ in $\{1,2\}$, we can find a smooth
 morphism of connected formal $R$-schemes $\mY_i\to \mX$ whose image contains
 $x_i$ and such that the restriction $\tau_i$ of $\tau$ to $\mY_i$ lifts to an
 element $\sigma_i$ of $\cF(\mY_i)$.  Since $\mX$ is
connected, the intersection of the images of $\mY_1\to \mX$ and
$\mY_2\to \mX$ is non-empty, and thus contains an $R$-point $x_3$.

 By Step 1, we know that
 $\tau_1$ lies in $\cG(\mY_1)$, since we can lift $x_1$ to a point
 $y_1$ in $\mY_1(R)$ and $y_1^*\tau_1=x_1^*\tau$ lies in $\cG(R)$.
Thus $x_3^*\tau$ lies in $\cG(R)$, because $x_3$ lies in the image
of $\mY_1\to \mX$. Again applying Step 1, we find that $\tau_2$
lies in $\cG(\mY_2)$. This means that $\tau$ is a section of $\cG$
locally at every $R$-point of $\mX$ with respect to the smooth
topology, so that $\tau$ must lie in $\cG(\mX)$.

\smallskip
{\em Step 3.} Now we prove that $\cG=(\cF')^o$. Step 2 implies at
once that $\cG(R)=(\cF')^o(R)$: for every element $\rho$ of
$(\cF')^o(R)$, we can find a connected smooth formal $R$-scheme
$\mX$, an element $\tau$ in $(\cF')^o(\mX)$ and points $x_0,\,x_1$
in $\mX(R)$ such that $x_0^*\tau=0$ and $x_1^*\tau=\rho$. By Step
2, we know that $\tau$ lies in $\cG(\mX)$, so that $\rho$ must lie
in $\cG(R)$. Now, again by Step 2, we see that
$\cG(\mY)=(\cF')^o(\mY)$ for every connected smooth formal
$R$-scheme, which implies that $\cG=(\cF')^o$.

\smallskip
(3) This follows from (2) and exactness of the sheafification
functor \cite[II.2.15]{milne}.
\end{proof}

\begin{lemma}\label{lemm-smoothtop}  Let $K'$ be a finite
 extension of $K$ with valuation ring $R'$, and let $\mX$ be a
 smooth formal $R'$-scheme. Then we can cover $\mX$ by open formal
 subschemes $\mU$ with the following property: there exist a
 connected smooth formal $R$-scheme
 $\mY$ and a smooth surjective morphism of $R'$-schemes $h:\mY\times_R R'\to \mU$
 such that, for every point $x$ of $\mU(R')$, there exists a point
 $y$ in $\mY(R)$ whose image in $(\mY\times_R
 R')(R')$ is mapped to $x$ by the morphism $h$.
\end{lemma}
\begin{proof}
 We denote by $k'$ the residue field of $R'$. This is a finite purely inseparable extension of the separably closed field $k$.
  We choose a basis  $e_1,\ldots,e_d$ for the $R$-module
$R'$.

 Shrinking $\mX$, we may
assume that $\mX$ is connected and admits an \'etale $R'$-morphism
to
$$\mathbb{B}^m_{R'}=\Spf R'\{X_1,\ldots,X_m\},$$ for some integer
$m\geq 0$.  We consider the morphism of formal $R'$-schemes
$$g:\mathbb{B}^{md}_{R'}=\Spf R'\{X_{i,j}\,|\,i=1,\ldots,m,\, j=1,\ldots,d\}\to
\mathbb{B}^m_{R'}$$ defined by
$$X_i\mapsto \sum_{j=1}^d e_jX_{i,j}.$$ This morphism is clearly
smooth and surjective. Moreover, over every $R'$-point of
$\mathbb{B}^m_{R'}$ we can find a point of
$\mathbb{B}^{md}_{R'}(R')$ whose $X_{i,j}$-coordinates lie in $R$.

 We set
 $$\mY'=\mathbb{B}^{md}_{R'}\times_{\mathbb{B}_{R'}^m}\mX.$$
 The second projection morphism $$h:\mY'\to \mX$$ is smooth and
 surjective, and the first projection morphism $$\mY'\to
 \mathbb{B}^{md}_{R'}$$ is \'etale.
Since the morphism
$$\Spec R'/\frak{m}^n R'\to \Spec R/\frak{m}^n$$ is finite, radicial and surjective for
every integer $n>0$, the invariance of the \'etale site under
 such morphisms \cite[IX.4.10]{sga1} implies that there exists an
 \'etale morphism of formal $R$-schemes
$$\mY\to \mathbb{B}_R^{md}=\Spf
 R\{X_1,\ldots,X_{md}\}$$ together with an isomorphism of formal
 $\mathbb{B}^{md}_{R'}$-schemes
 $$\mY'\to \mY\times_R R'.$$

Now let $x$ be any point of $\mX(R')$. We will construct a point
$y$ in $\mY(R)\subset \mY'(R')$ whose image in $\mX(R')$ is $x$.
 Let $b$ be a point of $\mathbb{B}_{R}^{md}(R)$ with the same image
as $x$ in $\mathbb{B}^m_{R'}(R')$. The couple $(x,b)$ defines a
point $y$ in
$$\mY(R')=\mathbb{B}^{md}_{R'}(R')\times_{\mathbb{B}_{R'}^m(R')}\mX(R').$$
Since $\mY\to \mathbb{B}^{md}_R$ is \'etale, the reduction $y_0$
of
 $y$ in $\mY'(k')$ lies in $\mY(k)$, because the reduction of $b$
 is $k$-rational and $k'$ is purely inseparable over $k$.
 Moreover, the point $b$ can be
lifted in a unique way to a point $z$ of $\mY(R)$ whose reduction
in $\mY(k)$ coincides with $y_0$. Repeating this uniqueness
argument after base change to $R'$, we see that $z$ must coincide
with $y$. In particular, $y$ lies in $\mV(R)\subset \mV'(R')$.
\end{proof}

\begin{lemma}\label{lemm-exacth}
Let $K'$ be a finite
 extension of $K$ with valuation ring $R'$. We denote the
 morphism $\Spf R'\to \Spf R$ by $h$. Then the functor $h_*$ from the category of smooth abelian sheaves on $\Spf R'$
  to the category of smooth abelian sheaves on $\Spf R$ is exact.
\end{lemma}
\begin{proof}
A direct image functor such as $h_*$ is always left exact, because
it has a right adjoint $h^*$. It follows from Lemma
\ref{lemm-smoothtop} that every smooth formal $R'$-scheme is
Zariski-locally of the form $\mX\times_R R'$, for some smooth
formal $R$-scheme $\mX$.
 This easily implies that $h_*$ is right exact.
\end{proof}

\begin{prop}\label{prop-compweilres} Let $K'$ be a finite
 extension of $K$ with valuation ring $R'$. We denote the
 morphism $\Spf R'\to \Spf R$ by $h$.
 For every abelian sheaf $\cF$ on $(\Spf R')_{\sm}$, the
 following properties hold.
 \begin{enumerate}
 \item If $\cF$ is connected, then $h_*\cF$ is connected.
\item The natural morphism $h_*(\cF^o)\to h_*\cF$ induces an
isomorphism $h_*(\cF^o)\to (h_*\cF)^o$. \item The natural morphism
$h_*\cF\to h_*\Comp{\cF}$ induces an isomorphism
$$\Comp{h_*\cF}\to h_*\Comp{\cF}.$$
\end{enumerate}
\end{prop}
\begin{proof}
(1) Let $\sigma$ be an element of $(h_*\cF)(R)=\cF(R')$. Since
$\cF$ is connected, we can find a connected smooth formal
$R'$-scheme $\mX$, a section $\tau$ in $\cF(\mX)$ and points
$x_0$, $x_1$ in $\mX(R')$ such that $x_0^*\tau=0$ and
$x_1^*\tau=\sigma$. By Lemma \ref{lemm-smoothtop}, we can find for
each $i\in \{0,1\}$ an open neighbourhood $\mU_i$ of $x_i$ in
$\mX$, a connected smooth formal $R$-scheme $\mY_i$
 and a smooth surjective morphism of formal
$R'$-schemes $$\mY_i\times_R R'\to \mU_i$$ such that every
$R'$-point on $\mU_i$ lifts to an $R$-point of $\mY_i$. We choose
for each $i$ a  point $y_i$ on $\mY_i(R)$ whose image in
$\mU_i(R')$ is $x_i$.
 We write $\tau_i$ for the restriction of
$\tau$ to $\mY_i\times_R R'$, and for the corresponding element of
$(h_*\cF)(\mY_i)$.

 Since $\mX$ is smooth and connected, we can find a point $x_2$ of
 $\mX(R')$ that lies in the intersection of $\mU_0$ and $\mU_1$.
 We can lift this point to a point $y_{2,i}$ in $\mY_i(R)$, for
 $i=0,1$. Then $y_{2,0}^*\tau_0$ lies in $(h_*\cF)^o(R)$ by
 Proposition
 \ref{prop-compcrit}, because $\mY_0$ is smooth and connected and
 $y_0^*\tau_0=x_0^*\tau=0$. But
 $$y_{2,0}^*\tau_0=x_2^*\tau=y_{2,1}^*\tau_1$$
 so that $y_{2,1}^*\tau_1$ and $y_1^*\tau_1=\sigma$ must also lie in
 $(h_*\cF)^o(R)$.

(2) By (1) and left exactness of $h_*$, it is enough to show the
following property: if $\mX$ is a connected smooth formal
$R$-scheme and $\sigma$ is an element of $(h_*\cF)^o(\mX)$, then
the corresponding element of $\cF(\mX\times_R R')$ belongs to
$\cF^o(\mX\times_R R')$. To prove this property, it suffices to
consider the case $\mX=\Spf R$, which follows immediately from the
definition of the identity component and the fact that
 $\mY\times_R R'$ is
connected for every connected formal $R$-scheme $\mY$.

(3) This follows from (2) and right exactness of $h_*$.
 \end{proof}

\subsection{The trace map}

\sss  Let
 $K'$ be a finite separable extension of $K$, with valuation ring $R'$.  Then we have a commutative diagram of morphisms of
sites
$$\begin{CD}(\Sp K')_{\sm} @>j'>> (\Spf R')_{\sm}
\\ @Vh_KVV @VVhV
\\ (\Sp K)_{\sm} @>j>> (\Spf R)_{\sm}.
\end{CD}$$

\sss \label{sss-trace} Let $F$ be an abelian sheaf on $(\Sp
K)_{sm}$. Since $(h_K)_*$ is left adjoint to $h_K^*$, we have a
tautological morphism
$$\tau:F\to (h_K)_*(h_K)^*F.$$ Applying the functor $j_*$, this yields
a morphism \begin{equation}\label{eq-shweilres0} j_*F\to
h_*j'_*h_K^*F\end{equation} of smooth abelian sheaves on $\Spf R$.
 In
\cite[2.3]{HaNi-comp}, we defined a {\em trace map}
$$\tr:(h_K)_* h_K^* F\rightarrow F$$ such
that the composition $\tr\circ \tau$ is multiplication by
$d=[K':K]$. Applying the functor $j_*$ to $\tr$, we obtain a
morphism of smooth abelian sheaves
\begin{equation}\label{eq-shweilres}  h_*j'_*h^*_KF\cong j_*(h_K)_*h_K^* F \to
j_*F\end{equation} on $\Spf R$.

\sss Now we apply the functor $\Comp{\cdot}$ to the morphisms
\eqref{eq-shweilres0} and \eqref{eq-shweilres}. This yields
morphisms of component groups \begin{eqnarray*}\alpha&:&\Comp F\to
\Comp{h_K^*F} \\ \tr &:& \Comp{h_K^*F}\to \Comp F,
\end{eqnarray*} where we used Proposition \ref{prop-compweilres}
to identify the component group of $h_*j'_*h^*_KF$ with
$\Comp{h_K^*F}$. The composition $\tr\circ \alpha$ is
multiplication by $d$. In particular, for every smooth commutative
rigid $K$-group, we obtain a trace map $$\Comp{G\times_K K'}\to
\Comp{G}$$ such that the precomposition with the base change
morphism
$$\Comp{G}\to \Comp{G\times_K K'}$$ is multiplication by $d$ on $\Comp{G}$.

\begin{prop}\label{prop-trace}
Let $K'$ be a finite separable extension of $K$, and let $G$ be a
smooth commutative rigid $K$-group. Then the kernel of the base
change morphism
$$\Comp{G}\to \Comp{G\times_K K'}$$ is killed by $[K':K]$.
\end{prop}
\begin{proof}
This in an immediate consequence of the existence of the trace
map.
\end{proof}

\section{The index of a semi-abelian
$K$-variety}\label{subsec-index}
\subsection{Definition of the index}
\begin{prop}\label{prop-split} Let $G$ be a semi-abelian
$K$-variety, denote by $G_{\spl}$ its maximal split subtorus, and
set $H=G/G_{\spl}$. Then $\Comp{G_{\spl}}$ is a free $\Z$-module
of rank $\rho_{\spl}(G)$, and the sequence
\begin{equation}\label{eq-split} 0\longrightarrow
\Comp{G_{\spl}}\longrightarrow \Comp G\longrightarrow \Comp
H\longrightarrow 0\end{equation} is exact. Moreover, $\Comp H$ is
finite, and the rank of
 $\Comp{G}$ equals $\rho_{\spl}(G)$.
\end{prop}
\begin{proof}
 It follows from the example in Section \ref{sss-torcomp} of Chapter \ref{chap-preliminaries} that $\Comp{G_{\spl}}$ is
free of rank $\rho_{\spl}(G)$. Since $H$ is an extension of an
abelian variety by an anisotropic torus, the N\'eron $lft$-model
of $H$ is quasi-compact, so that $\Comp{H}$ is finite.
 There are no non-trivial morphisms of
algebraic $K$-groups from $H$ to $\mathbb{G}_{m,K}$, so that the
conditions of \cite[10.1.7]{neron} are fulfilled. The exactness of
\eqref{eq-split} is shown in the proof of \cite[10.1.7]{neron}. It
follows that $\Comp{G}$ has rank $\rho_{\spl}(G)$.
\end{proof}

\sss We keep the notations of Proposition \ref{prop-split}. Since
$\Comp{H}$ is finite, the injective morphism of free $\Z$-modules
$$\Comp{G_{\spl}}\to \Comp{G}_{\free}$$ has finite cokernel.

\begin{definition}
We define the index of $G$, denoted by $i(G)$, as
$$i(G)=|\coker(\Comp{G_{\spl}}\to \Comp{G}_{\free})|.$$
\end{definition}

\sss The torsion part of $\Comp G$ is isomorphic to the component
group $\Comp H$ if and only if the index $i(G)$ is one. This
happens, for instance, if $G$ is the product of $G_{\spl}$ and
$H$, but $i(G)$ can be different from one in general, as shown by
the example below. In the next section, we will study the
behaviour of the index under finite extensions of $K$.

\subsection{Example: The index of a $K$-torus}\label{ss-torus}
\sss Let $T$ be a $K$-torus. We will compute the index of $T$ in
 terms of the character module $X(T)$ of $T$.
 Let $K'$ be a splitting field of
 $T$, and put $\Gamma=\Gal(K'/K)$. We consider the trace map
 $$\mathrm{tr}:X(T)\to X(T)^\Gamma:x\mapsto \sum_{g\in \Gamma}g*x.$$
By the proof of \cite[3.5]{HaNi-comp}, the maximal split subtorus
$T_{\spl}$ of $T$ has character module $X(T)/\ker(\mathrm{tr})$,
so that we have a canonical isomorphism $$\Comp{T_{\spl}}\cong
(X(T)/\ker(\mathrm{tr}))^{\vee}.$$ On the other hand, we can look
at the maximal anisotropic subtorus $T_{\mathrm{a}}$ of $T$. It
has character module $X(T)/X(T)^\Gamma$, and the quotient
$T/T_{\mathrm{a}}$ is a split $K$-torus with character module
$X(T)^\Gamma$. It is the dual of the maximal split subtorus of the
dual torus of $T$.

\sss We will see in Proposition \ref{prop-bx2} that the sequence
$$\Comp{T_{\mathrm{a}}}\longrightarrow \Comp{T}\longrightarrow
\Comp{T/T_{\mathrm{a}}}=(X(T)^\Gamma)^{\vee}\longrightarrow 0$$ is
exact. Since $\Comp {T_{\mathrm{a}}}$ is finite, the induced
morphism
$$\Comp{T}_{\free}\to (X(T)^\Gamma)^{\vee}$$ is an isomorphism.
 Thus the index $i(T)$ of $T$ is equal to the cardinality of the
cokernel of the injective morphism of abelian groups
$$ X(T)^\Gamma\to X(T)/\ker(\mathrm{tr})$$ that is
induced by the inclusion $X(T)^\Gamma\subset X(T)$.

\sss This index can be different from one. For instance, let $K$
be the field $\C((t))$ of complex Laurent series and denote by
$K'$ the degree two Galois extension $\C((\sqrt{t}))$ of $K$. Let
$T$ be the torus corresponding to the character module $X(T)=\Z^2$
with an action of $\Gal(K'/K)\cong \mu_2(\C)$
 given by
 $$(-1)*v=\left(\begin{array}{cc}
  0 & 1 \\ 1 & 0\end{array}\right)\cdot v.$$ Then $X(T)^\Gamma$ is the
  submodule of $\Z^2$
generated by $(1,1)$ and $\ker(\mathrm{tr})$ is the submodule of
$\Z^2$ generated by $(1,-1)$, so that $i(G)=2$.

\section{Component groups and base change}
\subsection{Uniformization of semi-abelian varieties}
\sss Our aim is to study the behaviour of the torsion part of the
component group of a semi-abelian $K$-variety $G$ under finite
extensions of the field $K$.
 To extend the results for abelian varieties from \cite{HaNi-comp}, we need a suitable
 notion of uniformization for $G$.

 \sss \label{sss-sabuniform} Let
 $$ 0\to M\to E^{\an}\to (G_{\ab})^{\an}\to 0
$$ be the uniformization of the abelian $K$-variety $G_{\ab}$, with $E$ a semi-abelian
$K$-variety with potential good reduction and $M$ an \'etale
lattice in $E$ of rank $\rho(E)$; see Chapter
\ref{chap-preliminaries}, \eqref{sss-uniformAV}.
 We put
$$H=E^{\an}\times_{(G_{\ab})^{\an}} G^{\an}.$$
This is a smooth rigid $K$-group that fits into the short exact
sequences of rigid $K$-groups
\begin{equation}\label{eq-seq1} 0\to M \to H
\to G^{\an}\to 0, \end{equation} and
\begin{equation}\label{eq-seq2} 0\longrightarrow G^{\an}_{\tor} \to H
\to E^{\an}\longrightarrow 0. \end{equation}
 We will call the sequence \eqref{eq-seq1} the non-archimedean
 uniformization of $G$.

 \sss It is important to keep in mind that $H$
 is usually not algebraic, so that the theory of N\'eron
 $lft$-models in \cite{neron} cannot be applied. One can deduce from \cite[1.2]{bosch-neron} and Proposition \ref{prop-bounded} below that the rigid
 $K$-group $H$ admits a quasi-compact formal N\'eron model in the
 sense of \cite[1.1]{bosch-neron} if and only if the split
 reductive ranks of $G_{\tor}$ and $E$ are zero. It seems
 plausible that $H$ always has a formal N\'eron model, but we will not prove this.
 Instead, we will work with the smooth sheaf associated to $H$ and
 use the theory of component groups of smooth abelian sheaves developed in Section \ref{sec-BX}.

\subsection{Bounded rigid varieties and torsors under analytic tori}\label{subsec-bounded}

\if false
\begin{prop}
Assume that $K$ is complete. Let $T$ be the analytification of an
anisotropic $K$-torus, and let
$$1\rightarrow T\rightarrow G \rightarrow H\rightarrow 1$$ be an
exact sequence of rigid $K$-groups. If $H$ is bounded, then $G$ is
bounded, as well.
\end{prop}
\begin{proof}
 The morphism $G\to H$ is flat, so that the image of an affinoid
 open subvariety of $G$ is a quasi-compact open
 subvariety of $H$ \cite[5.11]{formrigII}. Since $H$ is bounded,
 we can find a quasi-compact open subvariety $U$ of $G$ whose
 image in $H$ contains the set $H(K)$ and such that $U(K)$ is
 non-empty.

 The morphism $G\to Q$ is flat, so that the image of an affinoid
 open subvariety of $G$ is a quasi-compact open
 subvariety of $Q$ \cite[5.11]{formrigII}. Since $Q$ is bounded,
 we can find a quasi-compact open subvariety $U$ of $G$ whose
 image in $Q$ contains the set $Q(K)$.
 If we set $$G'=G\times_{Q} U,$$ then $G'$ is isomorphic to
 $$H\times_K U.$$ Since $H$ is bounded, we can find  a quasi-compact open
 subvariety $H_0$ of $H$ that contains all the $K$-rational points of
 $H$.
 Then $V=H_0\times_K U$ is a quasi-compact open subvariety of
 $$H\times_K U\cong G'.$$  The image of $V$ under the projection
 morphism
$$G'=G\times_{Q} U\to G$$ contains all the $K$-rational points of
$G$, because every point of $Q(K)$ lifts to a (not necessarily
$K$-rational) point of $U$. Since $V$ is quasi-compact, we can
cover its image, and therefore $G(K)$, by a finite number of
affinoid open subvarieties of $G$. Thus $G$ is bounded.
\end{proof}
\fi

\begin{prop}\label{prop-tortors}
Let $T$ be a split algebraic $K$-torus with
 N\'eron $lft$-model $\mathscr{T}$. Let $\mH$ be a
smooth connected formal $R$-scheme, and let $G$ be a
$T^{\an}$-torsor on $H=\mH_K$. Let $g$ be a point of $G(K)$, and
let $c$ be an element of the component group $\Comp T$. Then there
exists a unique map
$$\psi:G(K)\to \Comp T$$ with the following properties.
\begin{enumerate}
\item We have $\psi(g)=c$. \item The map $\psi$ is equivariant
with respect to the action of $T(K)$ on $G(K)$ and $\Comp
T=T(K)/\mathscr{T}^o(R)$. \item For every smooth connected formal
$R$-scheme $\X$, every $R$-morphism $\X\to \mH$ and every morphism
of rigid $H$-varieties
$$f:\X_K\to G,$$ the map
$$\psi\circ f:\X_K(K)\to \Comp T$$ is constant.
\end{enumerate}
Moreover, for every finite subset $\Phi_0$ of $\Comp{T}$, the set
$\psi^{-1}(\Phi_0)$ is contained in a quasi-compact open
subvariety of $G$.
\end{prop}
\begin{proof}
  We may assume that $c=0$, since we can always compose $\psi$ with a translation. By \cite[4.2]{B-X}, we can cover $\mH$ by
open formal subschemes $\mU$ such that the torsor $G$ is trivial
over $\mU_K$. Any intersection of such opens $\mU$ will contain a
$K$-valued point in its generic fiber, because $\mH$ is smooth and
connected. Thus we may assume that $G=H\times_K T^{\an}$. Let
$\mathscr{T}$ be the N\'eron $lft$-model of $T$ and denote for
each $t$ in $T(K)$ by $\overline{t}$ the residue class of $t$ in
$\Comp{T}=T(K)/\mathscr{T}^o(R)$. Then the map
$$\psi:G(K)=H(K)\times T(K)\to \Comp T:(h,t)\mapsto \overline{t}$$
is clearly the unique map satisfying properties (1)-(3) in the
statement. If we denote by $\mathscr{T}_0$ the union of the
connected components of $\mathscr{T}$ that belong to $\Phi_0$, and
by $T_0$ the generic fiber of the $\frak{m}$-adic completion of
$\mathscr{T}_0$, then $T_0$ is a quasi-compact open subvariety of
$T^{\an}$ and
 the set $\psi^{-1}(\Phi_0)$ is contained in the quasi-compact open subvariety $H\times_K T_0$ of $G$.
\end{proof}

\begin{definition}
We say that a rigid $K$-variety $X$ is bounded if $X$ admits a
quasi-compact open rigid subvariety that contains all the
$K$-rational points of $X$.
\end{definition}

\sss It follows from \cite[1.2]{bosch-neron} that a smooth rigid
$K$-group is bounded if and only if it admits a quasi-compact
formal N\'eron model in the sense of \cite[1.1]{bosch-neron}. If
$G$ is a semi-abelian $K$-variety, then $G^{\an}$ is bounded if
and only if it is the extension of an abelian $K$-variety by an
anisotropic $K$-torus.

\begin{prop}\label{prop-bounded}
Let $H$ be a bounded smooth  rigid $K$-variety and let $T$ be an
anisotropic $K$-torus. Then every $T^{\an}$-torsor $G$ on $H$ is
bounded.
\end{prop}
\begin{proof}
 Let $\mH$ be a formal weak N\'eron model of $H$. Recall that this
 means that $\mH$ is a quasi-compact smooth formal $R$-scheme, endowed with an
 open immersion of rigid $K$-varieties
  $\mH_{K}\to H$ whose image contains all the $K$-rational
  points of $H$. Every bounded smooth rigid $K$-variety admits a formal weak N\'eron model, by \cite[3.3]{bosch-neron} (recall that we assume all rigid $K$-varieties to be
  quasi-separated). If we let $\mU$ run through the connected components of $\mH$, then the set $H(K)$ is covered by the quasi-compact
  open subvarieties $\mU_K$ of $H$. Thus it suffices to prove the
  result after base change from $H$ to any of these subvarieties
  $\mU_K$. Therefore, we may assume that $H$ has a connected
  smooth formal $R$-model.

 We may also assume that there exists a $K$-rational point $x$ on $G$,
since $G$ is obviously bounded if $G(K)$ is empty.
 Let $K'$ be a finite Galois extension of $K$ such that
$T'=T\times_K K'$ is split, let $R'$ be the integral closure of
 $R$ in $K'$,
 and set $G'=G\times_K K'$ and $H'=H\times_K K'$. Then $H'$ has a
 connected smooth formal $R'$-model, and
$G'$ is a $(T')^{\an}$-torsor over $H'$. Let $$\psi:G'(K')\to
\Comp{T'}$$ be the function from Proposition \ref{prop-tortors}
such that $\psi(x)=0$. We will prove that $\psi$ sends each
element of  $G(K)\subset G'(K')$ to $0$. Then $G(K)$ is contained
in a quasi-compact open subvariety of $G'$, by Proposition
\ref{prop-tortors}, and this implies that $G$ is bounded.

 The canonical isomorphism
 $$\Comp{T'}\cong X(T)^{\vee}\otimes_{\Z}(K')^*/(R')^*$$ from \eqref{eq-torcomp} in Chapter \ref{chap-preliminaries} is
 equivariant with respect to the Galois action of $\Gal(K'/K)$.
 The Galois action on the value group $(K')^*/(R')^*$ is trivial, and $X(T)^{\Gal(K'/K)}=0$ because $T$ is
 anisotropic. Thus $\Comp{T'}^{\Gal(K'/K)}=0$. On the other hand,
 the uniqueness of the map $$\psi:G'(K')\to \Comp{T'}$$ easily implies that it is
 $\Gal(K'/K)$-equivariant: if $\sigma$ is an element of
 $\Gal(K'/K)$, then $\psi^{\sigma}:=\sigma\circ \psi\circ
 \sigma^{-1}$ maps $x$ to $0$ and still satisfies properties (2)
 and (3) in Proposition \ref{prop-tortors}. Thus $\psi$ maps each element of
 $G(K)$ to an element of $\Comp{T'}^{\Gal(K'/K)}=0$.
\end{proof}

\if false
 We will construct a quasi-compact open subvariety $G'_0$
of $G'$ such that the image of $G'_0$ in $G$ contains all the
$K$-rational points of $G$. The existence of $G'_0$ implies that
$G$ is bounded.

 Let $\mU'$ be a connected component of $\mH'$ and set
 $U'=\mU'_{K'}$. Let $g$ be a point of
 $$(G'\times_{H'}U')(K')\cap G(K)\subset G'(K').$$
 We denote by
$$\psi:G'\times_{H'}U'\to \Comp{T'}$$ the map
from
 Proposition \ref{prop-tortors} such that $\psi(g)=0$.

By \cite[4.2]{B-X}, we can cover $\mH'$ by finitely many connected
open formal subschemes
 $\mU'$ such that the $(T')^{\an}$-torsor $G'$ is split over the
 generic fiber
  $U'=\mU'_{K'}$. We choose a section $s:U'\to G'\times_{H'}U'$ of the projection
  $G'\times_{H'}U'\to U'$ such that the image of $s$ contains a point $g$ of $G(K)\subset G'(K')$.

 Let

  $\mathscr{T}$ and
  $\mathscr{T}'$ the N\'eron $lft$-models of $T$ and $T'$,
  respectively, and by $\mathscr{T}'_0$ the union of the connected components
  of $\mathscr{T}'$ that intersect the image of the base change
  morphism $$\mathscr{T}\times_R R'\to \mathscr{T}'.$$ We set
  $T'_0=(\mathscr{T}'_0)_{K'}$. This is a quasi-compact rigid
  $K'$-variety, because $\mathscr{T}$, and thus $\mathscr{T}'_0$, are quasi-compact.

 The action of $T'$ on $G'$ gives rise to an open immersion of rigid $K'$-varieties
 $$T'_0\times_{K'}U'\to G':(u,t)\mapsto t\cdot s(u).$$ We claim that
 the image of this morphism contains the set $$\mathrm{Im}(G(K)\to G'(K'))\cap (G'\times_{H'}U')(K').$$
 Then we can take for  $G'_0$ the union of the quasi-compact open subvarieties
 $T'_0\times_{K'}U'$ of $G'$, where $\mU$ ranges over our finite open
 cover of $\mH'$.

  Thus it suffices to prove our claim.

\end{proof}
\fi
\subsection{Behaviour of the component group under base
change}
\begin{prop}\label{prop-bx2} Assume that $k$ is algebraically closed. Let
$$
0\longrightarrow T\longrightarrow G\longrightarrow
H\longrightarrow 0$$ be a short exact sequence of semi-abelian
$K$-varieties, such that $T$ is a torus.
 Then the sequence of component groups
\begin{equation}\label{eq-phi2}
\Comp T\rightarrow \Comp G\rightarrow \Comp H\rightarrow
0\end{equation} is exact, and the kernel of $\Comp T\rightarrow
\Comp G$ is torsion.
\end{prop}
\begin{proof}
We already proved exactness of the sequence \eqref{eq-phi2} in
Proposition \ref{prop-bx}, so it remains to prove that the kernel
of $\Comp T\rightarrow \Comp G$ is torsion. It is enough to show
that
$$\rank(\Comp G)=\rank(\Comp T)+\rank(\Comp H).$$ This follows from Proposition \ref{prop-split} and
Lemma \ref{lemm-add} in Chapter 2.
\end{proof}

\begin{cor}\label{cor-trace}
Let $G$ be a semi-abelian $K$-variety, and let $L$ be a finite
separable extension of $K$ such that $G\times_K L$ has good
reduction. Then $\Comp{G}_{\tors}$ is killed by $[L:K]$.
\end{cor}
\begin{proof}
 The torus $G_{\tor}\times_K L$ is split, and the abelian variety
 $G_{\ab}\times_K L$ has good reduction. Thus it follows from
 Proposition \ref{prop-bx2} that $\Comp{G\times_K L}$ is free, so
 that $\Comp{G}_{\tors}$ must be contained in the kernel of the
 base change morphism
 $$\Comp{G}\to \Comp{G\times_K L}.$$ This kernel is killed by
 $[L:K]$, by Proposition \ref{prop-trace}.
\end{proof}

\sss We will need the following generalization of Proposition
\ref{prop-bx}. The main example to keep in mind is the following:
$H$ is the analytification of a semi-abelian $K$-variety, and $G$
is an analytic extension of $H$ by an algebraic $K$-torus $T$.
Then for every finite separable extension $L$ of $K$, the rigid
$K$-group $H\times_K L$ has a formal N\'eron model in the sense of
\cite[1.1]{bosch-neron}, by \cite[6.2]{bosch-neron}.

\begin{prop}\label{prop-exact2} Assume that $k$ is algebraically
closed.
 Let \begin{equation}\label{eq-rigexact} 0\longrightarrow
T^{\an}\longrightarrow G\longrightarrow H\longrightarrow 0
\end{equation}
be an exact sequence of smooth commutative rigid $K$-groups, where
$T^{\an}$ is the rigid analytification of an algebraic $K$-torus
$T$. Assume that $H$ has a formal N\'eron model
 and
 that the abelian sheaf $G/T^{\an}_{\spl}$ on  $(\Sp
K)_{sm}$ is representable by a smooth commutative rigid $K$-group
$G'$. Then the following properties hold.
\begin{enumerate}
\item \label{itm:exact} The sequence of component groups
\begin{equation}\label{eq-compgr}
 \Comp T\longrightarrow
\Comp G \longrightarrow \Comp H\longrightarrow 0
\end{equation}
is exact. \item \label{itm:torsion} If $H \times_K L$ admits a
formal N\'eron model for every finite separable extension $L$ of
$K$, then
 the kernel of $\Comp T\rightarrow \Comp G$ is finite.
 \end{enumerate}
\end{prop}
\begin{proof}
 The main problem is that we no longer know
 if $G$ admits a formal N\'eron model, so that we cannot apply the
 arguments in the proof of Proposition \ref{prop-bx2} in a direct
 way. Instead of trying to prove the existence of a formal N\'eron model for $G$, we will adapt the arguments of Proposition \ref{prop-bx2} to the  N\'eron
 model $j_*G$ of the smooth sheaf $G$ -- see \eqref{sss-neronsheaf}.


\eqref{itm:exact} The sequence \eqref{eq-rigexact} defines an
exact sequence of \'etale sheaves on $\Sp K$, and
$$H^1_{\mathrm{\acute{e}t}}(\Sp K,T^{\an})=H^1(K,T)=0$$ by
\cite[4.3]{chai}. It follows that the sequence
\begin{equation}\label{eq-rigexact2}
0\longrightarrow T(K)\longrightarrow G(K)\longrightarrow
H(K)\longrightarrow 0
\end{equation} is exact. We denote by $\mathscr{G}=j_*G$ and $\cH=j_*H$ the N\'eron
models of $G$ and $H$ on $(\Spf R)_{\sm}$. Looking at the proof of
Proposition \ref{prop-bx} and applying Proposition
\ref{prop-compgr}, we see that it is enough to show that the map
$$\mathscr{G}^o(R)\to \mathscr{H}^o(R)$$ is surjective.

 First, we reduce to the case where
$T^{\an}$ is bounded. It is shown in  \cite[4.2]{B-X} (split case)
that $R^1j_*T^{\an}_{\spl}=0$. Thus, applying
 the functor $j_*$ to the exact sequence
$$0\to T^{\an}_{\spl}\to G\to G'\to 0$$ of abelian sheaves on $(\Sp
 K)_{\mathrm{sm}}$,
   we get an exact
 sequence
 $$0\to j_*T^{\an}_{\spl}\to \mathscr{G}\to \mathscr{G}'\to 0 $$ of abelian sheaves on $(\Spf
 R)_{\mathrm{sm}}$. In
 particular, $\mathscr{G}\to \mathscr{G}'$ is an epimorphism. Then
 \cite[4.8]{B-X} implies that $\mathscr{G}^o\to
(\mathscr{G}')^o$ is an epimorphism, as well. It follows from
Lemma \ref{lemm-smtop} that $\mathscr{G}^o(R)\to
(\mathscr{G}')^o(R)$ is surjective. Therefore, it is enough to
prove that the map $(\mathscr{G}')^o(R)\to \mathscr{H}^o(R)$ is
surjective, so that we may replace the sequence
\eqref{eq-rigexact} by the exact sequence
$$0\to (T/T_{\spl})^{\an}\to G'\to H\to 0$$
of abelian sheaves on $(\Sp
 K)_{\mathrm{sm}}$. This means that we may assume that $T_{\spl}$ is
 trivial, and thus that $T^{\an}$ is bounded. Then $G'=G$.

Now, we reduce to the case where $H$ is quasi-compact. By our
assumptions, the N\'eron model $\cH$ of $H$ is representable by a
quasi-compact smooth formal $R$-scheme, that we denote again by
$\cH$. We denote by $\cH^o$ its identity component. The generic
fiber of $\mathscr{H}^o$
 is a quasi-compact open rigid subgroup of $H$, which we denote by
 $H^o$. It is
 easily seen that  $\mathscr{H}^o$ is a formal N\'eron
 model for $H^o$.

The inverse image $\widetilde{G}$ of $H^o$ in $G$ is an open rigid
subgroup of $G$ that fits into a short exact sequence of smooth
rigid $K$-groups
$$0\rightarrow T^{\an}\rightarrow \widetilde{G}\rightarrow
H^o\rightarrow 0.$$ Since the functor $j_*$ commutes with fibred
products, the N\'eron model
$\widetilde{\mathscr{G}}=j_*\widetilde{G}$ is isomorphic to
$\mathscr{G}\times_{\mathscr{H}}\mathscr{H}^o$. Clearly, the
morphism
 $\widetilde{\mathscr{G}}\to \mathscr{G}$ induces an isomorphism
 between the respective identity components. Thus we may assume
 that $H=H^o$ and $G=\widetilde{G}$. In particular, we may assume
 that $H$ is quasi-compact.

We've reduced to the case where both $T^{\an}$ and $H$ are
bounded. Then $G$ is bounded by Proposition \ref{prop-bounded}, so
that its N\'eron model $\cG$ is represented by a quasi-compact
smooth formal $R$-scheme \cite[1.2]{bosch-neron}.
 Exactness of \eqref{eq-rigexact2} implies that the natural map
$\mathscr{G}(R)\to \mathscr{H}(R)$ is surjective. Then it follows
from \cite[9.6.2]{neron} that $\mathscr{G}^o(R)\to
\mathscr{H}^o(R)$ is surjective, which is what we wanted to prove.
 To be precise, \cite[9.6.2]{neron} is formulated for algebraic schemes, but the proof of \cite[9.6.2]{neron} immediately carries over to the
formal scheme case, since it only involves the Greenberg schemes
of the algebraic schemes $\mathscr{G}\times_R (R/\mathfrak{m}^n)$
and $\mathscr{H}\times_R (R/\mathfrak{m}^n)$ for $n>0$.

\eqref{itm:torsion} Let $L$ be a finite separable extension of $K$
such that the
 torus $T\times_K L$ is split. Then the square
 $$\begin{CD} \Comp T@>>> \Comp G
\\ @VVV @VVV
\\ \Comp {T\times_K L}@>>> \Comp {G\times_K L}
\end{CD}$$ commutes, and the kernel of $\Comp T\to \Comp {T\times_K L}$ is
torsion by \cite[5.3]{HaNi-comp}. Thus it is enough to show that
$$\Comp {T\times_K L}\to \Comp {G\times_K L}$$ is injective. Therefore, we may
assume that $T$ is split.

 We set $\widetilde{G}=G\times_H H^o$ and $\widetilde{\mathscr{G}}=j_*\widetilde{G}$ as in \eqref{itm:exact}.
  Then $\widetilde{\mathscr{G}}$ is a subsheaf of $\mathscr{G}$
  and these sheaves have the same identity component.
 It follows that $\Comp{\widetilde{G}}$ is a subgroup
 of $\Comp{G}$ that contains the image of $\Comp {T}$ in $\Comp
 {G}$. Thus
  replacing $H$ by $H^o$ and $G$ by $\widetilde{G}$ does not affect the
  kernel of
$\Comp {T}\to \Comp {G}.$
   Therefore, we may assume that $H=H^o$. In
particular, $H$ admits a smooth connected formal $R$-model.

 We denote by $\psi$ the function associated to the
 $T^{\an}$-torsor $G\to H$ as in Proposition \ref{prop-tortors},
 normalized by $\psi(e_G)=0$ (here $e_G$ denotes the identity
 point of $G$).
 Let $t$ be an element of $T(K)$ whose class $\overline{t}$ in
$\Comp T=T(K)/\mathscr{T}^o(R)$ belongs to the kernel of $\Comp
T\to \Comp G$. We denote again by $t$ the image of $t$ in $G(K)$.
Then it follows from property (2) in Proposition
\ref{prop-tortors} that $\psi(t)=\overline{t}$.  Thus it is enough
to show that $\psi(t)=0$.

By definition of $(j_*G)^o$, there exist a
  smooth connected formal $R$-scheme $\X$, a morphism of rigid $K$-varieties
$$f:\X_K\to G$$
and elements $x_0$ and $x_1$ of $\X_K(K)$ such that $f$ maps $x_0$
to $e_G$ and $x_1$ to $t$. By the universal property of the formal
N\'eron model, the induced morphism $\X_K\to H$ extends uniquely
to an $R$-morphism $\X\to \mathscr{H}$. Thus we can apply property
(3) of Proposition \ref{prop-tortors}, and we see that
$\psi(t)=\psi(e_G)=0$, as required.
\end{proof}

\sss \label{sss-condition} Let $e$ be a positive integer, and let
$G$ be an abelian $K$-variety. We denote by $B$ the abelian
$K$-variety with potential good reduction that appears in the
non-archimedean uniformization of $G$ (see \eqref{sss-uniformAV}
in Chapter \ref{chap-preliminaries}). Then we consider the
following condition on the couple $(G,e)$:

\begin{center}
For every finite separable extension $K'$ of $K$ of degree prime
to $e$, the base change morphism $\Comp{B}\to \Comp{B\times_K K'}$
is an isomorphism.
\end{center}

\sss \label{sss-satisf}This condition is satisfied, for instance,
in each of the following cases:
 \begin{itemize}
 \item $G$ has potential
multiplicative reduction and $e$ is any positive integer (in this
case, $B$ is trivial); \item $G$ is tamely ramified and $e$ is a
multiple of the degree of the minimal extension $L$ of $K$ such
that $G\times_K L$ has semi-abelian reduction
\cite[5.5]{HaNi-comp}; \item $B$ is isomorphic to the Jacobian of
a smooth projective $K$-curve $C$ of index one, and $e$ is a
multiple of
 the
stabilization index $e(C)$ (Proposition \ref{prop-compfu} in
Chapter \ref{chap-jacobians}).
\end{itemize}
Moreover, if $(G,e)$ and $(G',e')$ are two couples satisfying
\eqref{sss-condition}, then the couple $$(G\times_K
G',\lcm\{e,e'\})$$ also satisfies \eqref{sss-condition}, since the
formation of N\'eron models commutes with finite products.

\begin{prop}\label{prop-goodred} Assume that $k$ is algebraically
closed. Let $G$ be a semi-abelian $K$-variety, with toric part
$G_{\tor}$ and abelian part $G_{\ab}$. Suppose that $G_{\ab}$ has
potential good reduction, and denote by $L$ the minimal extension
 of $K$ in $K^s$ such that $G\times_K L$ has semi-abelian reduction.
 Then the group $\Comp G_{\tors}$ is killed by $[L:K]$. For every
finite separable extension $K'$ of $K$ of degree $d$ prime to
$[L:K]$,
 the base change morphism
$$\alpha:\Comp G\rightarrow \Comp{G\times_K K'}$$ is injective.

Let $e>0$ be a multiple of $[L:K]$ such that $(G_{\ab},e)$
satisfies condition \eqref{sss-condition}. If $d=[K':K]$ is prime
to $e$,  then image of the morphism
$$\alpha_{\free}:\Comp{G}_{\free} \to \Comp{G\times_K K'}_{\free}
$$
equals $d\cdot \Comp{G\times_K K'}_{\free}$, and the morphism
$$\alpha_{\tors}:\Comp{G}_{\tors} \rightarrow \Comp{G\times_K K'}_{\tors}$$ is an
isomorphism. In particular, the cokernel of $\alpha$ is isomorphic
to $(\Z/d\Z)^{\rho_{\spl}(G)}$.
\end{prop}
\begin{proof}
For notational convenience, we'll denote by $(\cdot)'$ the base
change functor from $K$ to $K'$. In particular,  $G'=G\times_K
K'$. We proved the injectivity of $\alpha$ in
\cite[5.5]{HaNi-comp}, and we showed in Corollary \ref{cor-trace}
that $\Comp{G}_{\tors}$ is killed by $[L:K]$. Since, by
Proposition \ref{prop-semiab}, $G'$ acquires semi-abelian
reduction on the degree $[L:K]$ extension $K'\otimes_K L$ of $K'$,
Corollary \ref{cor-trace} also implies that $\Comp {G'}_{\tors}$
is killed by $[L:K]$.
 We prove the remainder of the theorem by considering the
 following
 cases.

 \medskip
\textit{Case 1: $G$ is a split torus.} This case was discussed in
Example \ref{sss-torcomp} in Chapter \ref{chap-preliminaries}.
Recall that $\Comp G$ and $\Comp{G'}$ are free $\Z$-modules of
rank $\rho_{\spl}(G)=\mathrm{dim}\,G$.

  \medskip
\textit{Case 2: $\rho_{\spl}(G)=0$.} In this case, we need to show
that $\alpha$ is an isomorphism.
 We have a commutative diagram
$$\begin{CD}
\Comp {G_\tor}@>>> \Comp G@>>>  \Comp {G_\ab}@>>> 0
\\ @V\beta VV @V\alpha VV @VV\gamma V @.
\\ \Comp {G'_\tor}@>>> \Comp {G'}@>>>  \Comp {G'_\ab}@>>> 0
\end{CD}$$
whose rows are exact by Proposition \ref{prop-bx}, and such that
all vertical morphisms are injective. Moreover, $\beta$ and
$\gamma$ are isomorphisms, by \cite[5.5]{HaNi-comp} and condition
 \eqref{sss-condition} for $(G_{\ab},e)$. A straightforward diagram chase shows that
$\alpha$ is an isomorphism.

\medskip
\textit{Case 3: General case.} Set $H=G/G_{\spl}$ and consider the
commutative diagram
$$\begin{CD} 0@>>> \Comp {G_{\spl}}@>>> \Comp {G}@>>> \Comp {H}@>>> 0
\\ @. @V\beta VV @V\alpha VV @VV\gamma V @.
\\ 0@>>> \Comp {G'_{\spl}}@>>\delta> \Comp {G'}@>>> \Comp {H'}@>>> 0
\end{CD}$$
The rows are exact by Proposition \ref{prop-bx2} and the fact that
$\Comp {G_{\spl}}$ and $\Comp {G'_{\spl}}$ are free $\Z$-modules.
By Lemma \ref{lemm-semiab}, we know that $H\times_K L$ has
semi-abelian reduction. Thus Case 2 implies that $\gamma$ is an
isomorphism. By the Snake Lemma, $\delta$ induces an isomorphism
between the cokernels of $\beta$ and $\alpha$. It follows from
Case 1 that the cokernel of $\beta$, and thus the cokernel of
$\alpha$, is isomorphic to $(\Z/d\Z)^{\rho_{\spl}(G)}$.

By Proposition \ref{prop-split}, the $\Z$-module $\Comp
{G}_{\free}$ has rank $\rho_{\spl}(G)$. Since $d$ is prime to $e$
and $\Comp {G}_{\tors}$ and $\Comp {G'}_{\tors}$ are killed by
$[L:K]$, the morphism $\alpha_{\tors}$ must be an isomorphism (it
is injective because $\alpha$ is injective, and its cokernel is
killed by both $d$ and $e$). It follows that the natural morphism
$\mathrm{coker}(\alpha)\to \mathrm{coker}(\alpha_{\free})$ is an
isomorphism, so that the image of $\alpha_{\free}$ equals $d\cdot
\Comp {G'}_{\free}$.
\end{proof}

\begin{prop}\label{prop-condjac}
Assume that $k$ is algebraically closed. Let $C$ be a smooth
projective $K$-curve of index one with Jacobian $G$. We denote by
$L$ the minimal extension of $K$ in $K^s$ such that $G\times_K L$
has semi-abelian reduction. If $e$ is a multiple of both $[L:K]$
and the stabilization index $e(C)$, then $(G,e)$ satisfies
condition \eqref{sss-condition}.
\end{prop}
\begin{proof} 

Let $K'$ be a finite separable extension of $K$ of degree $d$ prime to
both $[L:K]$ and to $e(C)$.
Let $E^{\an}\to G^{\an}$ be the non-archimedean uniformization of
$G$, and denote by $(\cdot)'$ the base change functor from $K$ to
$K'$. The proof of \cite[5.7]{HaNi-comp} shows that the base
change morphisms $\Comp{G}\to \Comp{G'}$ and $\Comp{E}\to
\Comp{E'}$ have isomorphic cokernels (in that proof, the
assumption that $A$ is tamely ramified or has potential
multiplicative reduction is only used at the very end, to apply
\cite[5.6]{HaNi-comp}). By Proposition \ref{prop-compfu}, this
implies that the cokernel of $\Comp{E}\to \Comp{E'}$ has
cardinality $d^{\trank{G}}$.

 We denote by $T$ and $B$ the toric, resp.~abelian part of $E$.
 Consider the commutative diagram
 $$\begin{CD}
\Comp{T}@>\delta >> \Comp{E}@>>> \Comp{B}@>>> 0
\\ @V\alpha VV @V\beta VV @V\gamma VV @.
\\ \Comp{T'}@>\varepsilon>> \Comp{E'}@>>> \Comp{B'}@>>> 0.
 \end{CD}$$
 The rows are exact by Proposition \ref{prop-bx}, the vertical
 morphisms are injective by \cite[5.7]{HaNi-comp}, and the
 cokernel of $\alpha$ also has cardinality
 $d^{\trank{G}}=d^{\trank{T}}$ by \cite[5.6]{HaNi-comp}.
 The kernels of $\delta$ and $\varepsilon$ are torsion by Proposition \ref{prop-bx2}, and
 thus killed by $[L:K]$, by Corollary \ref{cor-trace}. But $d$ is
 prime to
 $[L:K]$, and, moreover, $\alpha_{\tors}$ is an isomorphism by Proposition \ref{prop-goodred}. Thus $\alpha$ maps $\ker(\delta)$ surjectively onto
 $\ker(\varepsilon)$, and the Snake Lemma implies that $\gamma$ is
 an isomorphism.
\end{proof}

\begin{theorem}\label{thm-main}
 Assume that $k$ is algebraically closed. Let $G$ be a semi-abelian $K$-variety, with toric part $G_{\tor}$
and abelian part $G_{\ab}$. Let $L$ be the minimal extension of
$K$ in $K^s$ such that $G\times_K L$ has semi-abelian reduction.
Let $K'$ be a finite separable extension of $K$ of degree $d$.

\begin{enumerate}
\item If $d$ is prime to $[L:K]$, then the base change morphism
$$\alpha:\Comp G\rightarrow \Comp {G\times_K K'}$$ is injective.

\item Let $e$ be a multiple of $[L:K]$ in $\Z_{>0}$ such that
$(G_{\ab},e)$ satisfies condition \eqref{sss-condition}. If $d$ is
prime to $e$, then the cokernels of the morphisms
\begin{eqnarray*}
\alpha&:&\Comp G  \rightarrow  \Comp {G\times_K K'}
\\ \alpha_{\tors}&:&\Comp G_{\tors}  \rightarrow  \Comp {G\times_K K'}_{\tors}
\\ \alpha_{\free}&:&\Comp G_{\free}  \rightarrow  \Comp {G\times_K K'}_{\free}
\end{eqnarray*} are isomorphic to  $(\Z/d\Z)^{t(G)}$, resp.
$(\Z/d\Z)^{t(G_{\ab})}$, resp. $(\Z/d\Z)^{\rho_{\spl}(G)}$.
\end{enumerate}
\end{theorem}
\begin{proof}
To simplify notation, we'll denote by $(\cdot)'$ the base change
functor from $K$ to $K'$.
 First, we prove (1). It follows from Proposition \ref{prop-bx} that the
commutative diagram
$$\begin{CD}
\Comp {G_{\tor}}@>\delta>>\Comp {G}@>\epsilon>> \Comp
{G_{\ab}}@>>> 0
\\ @V\beta VV @V\alpha VV @VV \gamma V @.
\\ \Comp {G'_{\tor}}@>>\delta'>\Comp {G'}@>>\epsilon'> \Comp {G'_{\ab}}@>>> 0
\end{CD}$$
has exact rows. Let $x$ be an element of $\Comp G$ such that
$\alpha(x)=0$. Then $dx=0$ by Proposition \ref{prop-trace}.
Moreover,
 $\epsilon(x)=0$ because $\gamma$ is injective \cite[5.7]{HaNi-comp}, so
 that $x$ lifts to an element $\tilde{x}$ in $\Comp {G_{\tor}}$. This
 element $\tilde{x}$ must be torsion, because the kernel of $\delta$ is torsion by Proposition \ref{prop-bx2}. 
 But the torsion part of
$\Comp {G_{\tor}}$ is killed by $[L:K]$ (Corollary
\ref{cor-trace}), so that $x=0$ since $d$ is prime to $[L:K]$.
Hence, $\alpha$ is injective.


 Now, we prove (2).
 The restriction of $\beta$ to $\ker(\delta)$ is a surjection onto
 $\ker(\delta')$, since $\alpha$ is injective, the kernel of $\delta'$ is torsion 
   and
$\beta$ is an isomorphism on torsion parts (Proposition
\ref{prop-goodred}). Thus it follows from the Snake Lemma that
 the sequence
\begin{equation}\label{eq-exact}0\rightarrow \mathrm{coker}(\beta)\rightarrow
\mathrm{coker}(\alpha)\rightarrow
\mathrm{coker}(\gamma)\rightarrow 0\end{equation} is exact.
Applying
 Proposition \ref{prop-goodred} to $G_{\tor}$, we see that the
cokernel of $\beta$ is isomorphic to $(\Z/d\Z)^{\rho_{\spl}(G)}$
(note that $\rho_{\spl}(G_{\tor})=\rho_{\spl}(G)$).

 Assume for a moment that we can prove that the cokernel of $\alpha_{\tors}$ is
 isomorphic to $(\Z/d\Z)^{t(G_{\ab})}$. Replacing $G$ by
 $G_{\ab}$, this also implies that
 the cokernel of $\gamma$
 is isomorphic to $(\Z/d\Z)^{t(G_{\ab})}$, because
  the component group of an abelian variety is finite.
  Moreover, the cokernel of $\alpha$ is killed by $d$, by the
 existence of the trace map \eqref{sss-trace}, and this implies
 that the short exact sequence \eqref{eq-exact} is split.
  Thus $\mathrm{coker}(\alpha)$ is isomorphic to
 $$\coker{\beta}\oplus \coker{\gamma}\cong(\Z/d\Z)^{\rho_{\spl}(G)}\oplus (\Z/d\Z)^{t(G_{\ab})}\cong (\Z/d\Z)^{t(G)}$$
 where the equality $\rho_{\spl}(G)+t(G_{\ab})=t(G)$ follows from
Corollary \ref{cor-torrank}. Another application of the Snake
Lemma shows that the natural morphism $\coker(\alpha_{\tors})\to
\coker(\alpha)$ is injective, and that its cokernel is
 isomorphic to $\coker(\alpha_{\free})$. This implies that
 $\coker(\alpha_{\free})$is isomorphic to
$(\Z/d\Z)^{\rho_{\spl}(G)}$.

Thus it suffices to determine the cokernel of the morphism
$$\alpha_{\tors}:(\Comp G)_{\tors}  \rightarrow  (\Comp {G'})_{\tors}.$$  We
consider the non-archimedean uniformization
$$0\to M\to H\to G^{\an}\to 0$$ of $G$ as in
\eqref{sss-sabuniform}. We set $I=\Gal(K^s/K)$ and
$I'=\Gal(K^s/K')$. Like in the proof of \cite[5.7]{HaNi-comp}, we
get a commutative diagram with exact rows
$$\begin{CD}
0@>>> M^I @>>> \Comp H@>>> \Comp G@>\widetilde{\gamma}_1>>
H^1(I,M)
\\ @. @V\alpha_1 V\wr V @V\widetilde{\alpha}_2 VV @V\widetilde{\alpha}_3 VV @V\wr V\alpha_4 V
\\ 0@>>> M^{I'}@>>> \Comp {H'}@>>> \Comp {G'}@>\widetilde{\gamma}_2>>
H^1(I',M)
\\ @. @. @. @V\widetilde{\beta}_3VV @V\wr V\beta_4 V
\\ @. @. @. \Comp G@>\widetilde{\gamma}_1>> H^1(I,M)
\end{CD}$$
where $\widetilde{\beta_3} \circ \widetilde{\alpha_3}$ and
$\beta_4\circ \alpha_4$ are multiplication by $d$, and where
$\alpha_1,\,\alpha_4$ and $\beta_4$ are isomorphisms. Here we view
$M$ as a discrete $I$-module. Since $H^1(I,M)$ is killed by
$[L:K]$, and $d$ is prime to $[L:K]$, the isomorphism $\alpha_4$
identifies $\widetilde{\gamma}_1(\Comp G_{\tors})$ and
$\widetilde{\gamma}_2(\Comp {G'}_{\tors})$ (see the argument for
Claim 3 in the proof of \cite[5.7]{HaNi-comp}). Hence, looking at
the commutative diagram with exact rows
$$\begin{CD}
0@>>> (\Comp H/M^I)_{\tors}@>>> \Comp G_{\tors}@>>>
\widetilde{\gamma}_1(\Comp G_{\tors})@>>>0
\\ @. @VVV @VVV @VV\wr V @.
\\ 0@>>> (\Comp {H'}/M^{I'})_{\tors}@>>> \Comp {G'}_{\tors}@>>>
\widetilde{\gamma}_2(\Comp {G'}_{\tors})@>>>0
\end{CD}$$
we see that it is enough to show that the cokernel of the
injective morphism
$$\phi:(\Comp H/M^I)_{\tors}\rightarrow (\Comp {H'}/M^{I'})_{\tors}$$ is isomorphic to $(\Z/d\Z)^{t(G_{\ab})}$.
  Since $M^I$ is a free $\Z$-module of
rank $t(G_{\ab})$ \cite[3.13]{HaNi-comp}, it suffices to prove the
following claims.

\vspace{5pt} \textit{Claim 1: The morphism $\Comp H\rightarrow
\Comp {H'}$ is injective.}

\vspace{5pt} \textit{Claim 2: The morphism $\Comp
H_{\tors}\rightarrow \Comp {H'}_{\tors}$ is an isomorphism.}

\vspace{5pt} \textit{Claim 3: The image of the morphism $\Comp
H_{\free}\rightarrow \Comp {H'}_{\free}$ equals $d\cdot (\Comp
{H'})_{\free}$.}

\vspace{5pt}
 From the exact sequence \eqref{eq-seq2}, we deduce the commutative diagram
$$\begin{CD}
\Comp {G_{\tor}}@>\phi_1>> \Comp H @>\phi_2>> \Comp E @>>> 0
\\ @V\delta_1VV @V\delta_2VV @VV\delta_3 V @.
\\ \Comp {G'_{\tor}}@>>\phi'_1> \Comp {H'} @>>\phi'_2> \Comp {E'} @>>> 0
\\ @. @V\epsilon_2VV @VV\epsilon_3V @.
\\ @. \Comp H @>>\phi_2> \Comp E @>>> 0
\end{CD}$$
where $\epsilon_2$ and $\epsilon_3$ are the trace maps; thus
$\epsilon_i\circ \delta_i$ is multiplication by $d$, for $i=2,3$.
 The rows of this diagram are exact and the kernels of $\phi_1$ and $\phi'_1$ are
torsion, by Proposition \ref{prop-exact2}. Moreover, $\delta_1$
and $\delta_3$ are injective, by point (1). An easy diagram chase
shows that the kernel of $\delta_2$ must be contained in the
torsion part of $\Comp {G_{\tor}}$. But  $\Comp
{G_{\tor}}_{\tors}$ is killed by $[L:K]$ (Corollary
\ref{cor-trace}) and $\ker(\delta_2)$ is killed by $d$, so that
$\delta_2$ must be injective. This settles Claim 1.

 We know that
$\Comp {G_{\tor}}_{\tors}$ and $\Comp E_{\tors}$ are killed by
$[L:K]$, and that the kernel of $\phi_1$ is torsion. This easily
implies that $\Comp{H}_{\tors}$ is killed by $[L:K]^2$.
 Moreover,
 $(\delta_1)_{\tors}$ and $(\delta_3)_{\tors}$ are isomorphisms, by Proposition
\ref{prop-goodred}. Since $d$ is prime to $[L:K]$, it is
invertible in $\Comp{E}_{\tors}$, and $(1/d)(\epsilon_3)_{\tors}$
is inverse to $(\delta_3)_{\tors}$. Likewise,
 $(1/d)(\epsilon_2)_{\tors}$ is a right inverse of
$(\delta_2)_{\tors}$. Thus to prove that $(\delta_2)_{\tors}$ is
an isomorphism, it is enough to show that $(\epsilon_2)_{\tors}$
is injective. Let $x$ be an element of $\Comp{H'}_{\tors}$, and
assume that $\epsilon_2(x)=0$. Then $\phi_2'(x)=0$, so that we can
lift $x$ to an element $y$ of $\Comp{G'_{\tor}}$. This element
must be torsion, because the kernel of $\phi_1'$ is torsion. If we
set $$z=(\phi_1\circ (\delta_1)^{-1}_{\tors})(y)$$ then
$\delta_2(z)=x$ and $$z=(1/d)(\epsilon_2\circ \delta_2)(z)=
(1/d)\epsilon_2(x)=0.$$ Thus $x=0$ and $\delta_2$ is injective.
 This proves Claim 2.

\if false
 $\iota=(1/d)(\epsilon_2)_{\tors}$ is a right inverse of
$(\delta_3)_{\tors}$. To prove that  Let $x$ be an element of
$\Comp{H'}_{\tors}$, and set $y=((\delta_2)_{\tors}\circ
\iota)(x)$. Then $\phi_2'(x-y)=0$, so that we can lift $x-y$ to an
element $z$ in $\Comp{G'_{\tor}}_{\tors}$. , so that we can write
it as $\delta_1(w)$ for some $w$ in $\Comp{G_{\tor}}_{\tors}$.

The morphisms $(\delta_1$ and $\delta_3$ induce isomorphisms on
the torsion parts, by Proposition \ref{prop-goodred}.

Since $\epsilon_2\circ \delta_2$ is multiplication by $d$ and $d$
is prime to $[L:K]$, we see that $(\epsilon_2\circ
\delta_2)_{\tors}$ is an isomorphism.

and $d$ is prime to $[L:K]$, the isomorphism $(\delta_3)_{\tors}$
identifies the images of $(\phi_2)_{\tors}$ and
$(\phi'_2)_{\tors}$. Since the kernels of $\phi_1$ and $\phi'_1$
are torsion, a diagram chase shows that $(\delta_2)_{\tors}$ is
surjective, and hence, an isomorphism. This proves Claim 2. \fi

Now we prove Claim 3. We have a commutative diagram with exact
rows
$$\begin{CD}
0@>>>\Comp {G_{\tor}}_{\free}@>(\phi_1)_{\free}>> \Comp H_{\free}
@>\widetilde{\phi}_2>> \Comp E/\Ima((\phi_2)_{\tors}) @>>> 0
\\ @. @V(\delta_1)_{\free}VV @V(\delta_2)_{\free}VV @VV\widetilde{\delta}_3 V @.
\\ 0@>>>\Comp {G'_{\tor}}_{\free}@>>(\phi'_1)_{\free}> \Comp {H'}_{\free}
@>>\widetilde{\phi}'_2> \Comp {E'}/\Ima((\phi'_2)_{\tors}) @>>> 0
\end{CD}$$
where all the vertical maps are injective. Since $\epsilon_2\circ
\delta_2$ is multiplication by $d$, and $\Comp H$ and $\Comp {H'}$
have the same rank, we know that
\begin{equation}\label{eq-inc}d\cdot \Comp {H'}_{\free}\subset
\Ima((\delta_2)_{\free})\end{equation} so that it remains to prove
the converse implication.

 Since $E'$ acquires semi-abelian
reduction over the degree $[L:K]$ extension $L\otimes_K K'$ of
$K'$, and $d$ is prime to $[L:K]$, Corollary \ref{cor-trace}
implies that $d$ is invertible in $\Comp{E'}_{\tors}$. Likewise,
$d$ is invertible in $\Comp{G'_{\tor}}_{\tors}$. Applying
Proposition \ref{prop-goodred} to $E$ and $G_{\tor}$, we see that
every element in the image of $\delta_1$ or $\delta_3$
 is divisible by $d$. Thus every element in the image of
 $(\delta_1)_{\free}$ or $\widetilde{\delta}_3$ is divisible by
 $d$. We will deduce that every element in the image of
 $(\delta_2)_{\free}$ is divisible by $d$.

Let $x$ be an element of $\Comp H_{\free}$ and set
$y=(\delta_2)_{\free}(x)$.
 Then there exists
 an element $z$ in $\Comp {E'}/\Ima((\phi'_2)_{\tors})$ such
 that $$dz=(\widetilde{\delta}_3\circ \widetilde{\phi}_2)(x)=\widetilde{\phi}'_2(y).$$ We can lift $z$ to an element
 $y'$ in $\Comp {H'}_{\free}$ such that
 $\widetilde{\phi}'_2(y-dy')=0$. By \eqref{eq-inc}, we know that $dy'$ lies in the image of $(\delta_2)_{\free}$, so that the element
 $y-dy'$ can be written as $(\delta_2)_{\free}(x')$, with $x'$ in
 $\Comp H_{\free}$. By injectivity of $\widetilde{\delta}_3$, we
 know that $x'$ belongs to $\Comp {G_{\tor}}_{\free}$. Thus $y-dy'$ lies
 in the image of
 $(\delta_1)_{\free}$, and  we can find an element $w$ in $\Comp {G'_{\tor}}_{\free}$
 such that $dw=y-dy'$. It follows that $y$ is divisible by $d$ in
 $\Comp {H'}_{\free}$.  \end{proof}

\begin{cor}\label{cor-index}
Assume that $k$ is algebraically closed. Let $G$ be a semi-abelian
$K$-variety, and let $L$ be the minimal extension of $K$ in $K^s$
such that $G\times_K L$ has semi-abelian reduction.  Let $e$ be a
multiple of $[L:K]$ in $\Z_{>0}$ such that $(G_{\ab},e)$ satisfies
condition \eqref{sss-condition}. Then for every finite separable
extension $K'$ of $K$ of degree prime to $e$,
 we have
$$i(G)=i(G\times_K K').$$
\end{cor}
\begin{proof} We set $G'=G\times_K K'$ and $d=[K':K]$.
 The index of $G$ is given by
$$i(G)=|\mathrm{coker}(\Comp {G_{\spl}}\to \Comp G_{\free})|.$$
 We have a commutative square of free $\Z$-modules
 $$\begin{CD}
\Comp {G_{\spl}}@>>> \Comp G_{\free}
\\ @V\alpha_{\spl}VV @VV\alpha_{\free}V
\\ \Comp {(G')_\spl}@>>> \Comp {G'}_{\free}
 \end{CD}$$
 with injective morphisms. By the proof of \cite[4.2]{HaNi-comp},
 we know that the natural morphism $G_{\spl}\times_K K'\to
 (G')_\spl$ is an isomorphism. Thus all the $\Z$-modules in the
 diagram have rank $t(G)=t(G')$. The cokernels of $\alpha_{\spl}$
 and $\alpha_{\free}$ have the same cardinality, namely,
 $d^{t(G_{\tor})}=d^{\rho_{\spl}(G)}$. Thus the cokernels of
 the horizontal morphisms also have the same cardinality, which
 means that $i(G)=i(G')$.
\end{proof}

\sss If $G$ is an abelian variety with potential multiplicative
reduction or a tamely ramified abelian variety, then we proved
Theorem \ref{thm-main} in \cite[5.7]{HaNi-comp} without the
assumption that $k$ is algebraically closed. To be precise,
 \cite[5.7]{HaNi-comp} only gives a formula for the cardinality of the cokernel of $\alpha$, but
 using the trace map as in the proof of Theorem \ref{thm-main} one can determine the full structure of the cokernel. Unfortunately, there is a slight error in
the statement of \cite[5.7]{HaNi-comp}: the cokernel of the base
change morphism $\alpha$ is isomorphic to $(\Z/r\Z)^{t(G)}$, with
$r$ the ramification index of $K'/K$ instead of the degree. Of
course, when $k$ is algebraically closed, these values coincide.

\subsection{The component series of a semi-abelian variety}
\sss
Theorem \ref{thm-main} makes it possible to extend our previous
results on the component series for abelian varieties
(\cite[5.7]{HaNi-comp} and Theorem \ref{theorem-compser} in
Chapter \ref{chap-jacobians}) to semi-abelian varieties.

\begin{theorem}\label{thm-compsersab} Assume that $k$ is algebraically closed. Let $G$ be a semi-abelian
$K$-variety with toric part $G_{\tor}$ and abelian part $G_\ab$.
Let $L$ be the minimal extension of $K$ in $K^s$ such that
$G\times_K L$ has semi-abelian reduction, and let $e$ be a
multiple of $[L:K]$ in $\Z_{>0}$ such that $(G(d),e/gcd(d,e))$
satisfies condition \eqref{sss-condition} for all $d$ in $\N'$.

Then the component series
$$S^{\Phi}_G(T) = \sum_{d \in \N'} |\Comp{G(d)}_\tors| T^d $$ is rational.
More precisely, it belongs to the subring
$$ \mathscr Z = \Z \left[T, \frac{1}{T^j - 1} \right]_{j \in \Z_{>0}} $$
of $ \Z[[T]]$. It has degree zero if $p=1$ and $G$ has potential
good reduction, and degree $<0$ in all other cases. Moreover,
$S^{\Phi}_G(T)$ has a pole at $T=1$ of order $\ttame{G_{\ab}} +
1$.
\end{theorem}
\begin{proof}
One can simply copy the arguments in \cite[6.5]{HaNi-comp},
invoking Theorem \ref{thm-main} instead of \cite[5.7]{HaNi}.
\end{proof}

\sss  The conditions on $G$ and $e$ in the statement of Theorem
\ref{thm-compsersab} are satisfied, for instance, if $G_{\ab}$ is
tamely ramified or has potential multiplicative reduction  and
$e=[L:K]$, and also if $G_{\ab}$ is isomorphic to the Jacobian of
a smooth projective curve $C$ of index one and $e$ is the
least common multiple of $[L:K]$ and the stabilization index
 $e(C)$ of $C$ (Corollary \ref{cor-e(C)} in Chapter \ref{chap-jacobians} and Proposition \ref{prop-condjac}).
 We expect that
Theorem \ref{theorem-compser} is valid for all semi-abelian
$K$-varieties $G$.

\part{Chai and Yu's base change conductor and Edixhoven's
filtration}\label{part-conductor} \chapter{The base change conductor and Edixhoven's
filtration}\label{chap-conductor}

In this chapter, we recall the definition of the base change
conductor of a semi-abelian $K$-variety $G$ and of Edixhoven's
filtration on the special fiber of the N\'eron model of $G$. We
use Edixhoven's construction to define a new invariant, the tame
base change conductor, which is important for the applications to
motivic zeta functions in Part \ref{part-motzeta}.
 We compare the base change conductor and its tame counterpart on
 some explicit examples. The main result of this section states
 that the jumps of the Jacobian variety of a $K$-curve $C$ only
 depend on the combinatorial reduction data of $C$ (Theorem \ref{Theorem-jumps}). This
 generalizes a previous result of the first author, where an additional condition
 on the reduction data was imposed.

\section{Basic definitions}
\subsection{The conductor of a morphism of modules}
\begin{definition}
Let $f:M\to N$ be an injective morphism of free $R$-modules of
finite rank $r$. The tuple of elementary divisors of $f$ is the
unique monotonically increasing tuple of non-negative integers
$(c_1(f),\ldots,c_r(f))$ such that
$$\coker(f)\cong \oplus_{i=1}^r R/\frak{m}^{c_i(f)}.$$
The conductor of $f$ is the non-negative integer
$$c(f)=\sum_{i=1}^r c_i(f)=\length_R \coker(f).$$
\end{definition}
Note that the elementary divisors of $f$ are the valuations of the
diagonal elements in a Smith normal form of $f$, so that
 $c(f)$ is equal to the
exponent of the determinant ideal of $f$ in $R$.

\subsection{The base change conductor of a semi-abelian variety}
 \sss \label{sssec-bcnot} Let $G$ be a semi-abelian
$K$-variety of dimension $g$ with N\'eron $lft$-model
$\mathscr{G}/R$. Let $ K'/K $ be a  finite separable field
extension of ramification index $e(K'/K)$, and denote by $R'$ the
integral closure of $R$ in $K'$.
 We
denote by $\mathscr{G}'/R'$ the N\'eron model of $G'=G \times_K
K'$.

The canonical base change morphism
$$
 h : \mathscr{G} \times_R
R' \to \mathscr{G}' $$ induces an injective homomorphism
\begin{equation}\label{eq-Lieh}
 \Lie(h) : \Lie(\mathscr{G}) \otimes_R R' \to
\Lie(\mathscr{G}') \end{equation} of free $R'$-modules of rank
$g$.
\begin{definition}\label{def-bc}
 We call $$\left(\frac{c_1(\Lie(h))}{e(K'/K)},\ldots,\frac{c_g(\Lie(h))}{e(K'/K)}\right)$$ the tuple of
 $K'$-elementary divisors of $G$, and we denote it by
 $$(c_1(G,K'),\ldots,c_g(G,K')).$$
 If $G'$
 has semi-abelian reduction, we set $c_i(G) = c_i(G,K')$ for every $i$ in $\{1,\ldots,g\}$ and we call
  $(c_1(G),\ldots,c_g(G))$ the tuple of elementary divisors
  of $G$.

We call the rational number
$$ c(G,K') := \sum_{i=1}^gc_i(G,K')=\frac{1}{e(K'/K)} \cdot \mathrm{length}_{R'} (\coker(\Lie(h))) $$
 the $K'$-base change conductor of $G$. If $G'$
 has semi-abelian reduction, we set $c(G) = c(G,K')$ and we call this invariant the
base change conductor of $G$.
\end{definition}
\sss The values $c_i(G)$ and $c(G)$ are independent of the choice
of a finite separable extension $K'$ of $K$ such that $G\times_K
K'$ has semi-abelian reduction. This follows from the fact that
$\Lie(h)$ is an isomorphism for every $K'$ if $G$ has semi-abelian
reduction.

\sss \label{sss-bc} The base change conductor and the elementary
divisors were introduced for tori by Chai and Yu in \cite{chai-yu}
and  for semi-abelian varieties by Chai in \cite{chai}. The base
change conductor $c(G)$ vanishes if and only if $G$ has
semi-abelian reduction \cite[4.16]{HaNi}, and one can consider
$c(G)$ as a measure for the defect of semi-abelian reduction of
$G$. More generally, $c(G,K')$ measures the difference between the
identity components $(\mathscr{G}\times_R R')^o$ and
$(\mathscr{G}')^o$, and vanishes if and only if the morphism
$$
 h : \mathscr{G} \times_R
R' \to \mathscr{G}' $$ is an open immersion (same proof as
\cite[4.16]{HaNi}).

\sss \label{sss-addbc} It is straightforward to check that
$c(G,K')$ behaves additively in towers, in the following sense: if
$K\subset K'\subset K''$ is a tower of finite separable
extensions, then
$$c(G,K'')=c(G,K')+\frac{c(G',K'')}{e(K'/K)}.$$ Moreover, if $c(G,K')$
is zero,
 then $$c_i(G,K'')=\frac{c_i(G',K'')}{e(K'/K)}$$ for all $i$ in
$\{1,\ldots,g\}$, and if $c(G',K'')$ is zero, then
$c_i(G,K')=c_i(G,K'')$ for all $i$ in $\{1,\ldots,g\}$. Choosing
$K''$ in such a way that $G\times_K K''$ has semi-abelian
reduction, we see in particular that
$$c(G)=c(G,K')+\frac{c(G')}{e(K'/K)}$$ for all finite separable extensions $K'$ of $K$,  and that $$c_i(G)=\frac{c_i(G')}{e(K'/K)}$$ for all $i$ in $\{1,\ldots,g\}$ if
$c(G,K')=0$.

\subsection{Jumps and Edixhoven's filtration}
\sss  One can define a filtration on $\mathscr{G}_k$ by closed
subgroups that measures the behaviour of $\mathscr{G}$ under tame
base change.
  This construction is due to  Edixhoven if $G$ is an abelian variety
  \cite{edix} and extends to semi-abelian varieties in a straightforward way; see
  \cite[\S4.1]{HaNi}. We'll briefly recall its construction.

\sss Let $d$ be an element of $\N'$. Recall that we denote by
 $K(d)$ the unique degree $d$ extension of $K$  in $K^s$, and by $R(d)$ its valuation ring. We
denote by $\mu$ the Galois group
 $\Gal(K(d)/K)$, and by $\cG(d)$ the N\'eron $lft$-model of $G(d)=G \times_K
 K(d)$. The Weil restriction $$\mathscr{W}=\prod_{R(d)/R}\cG(d)$$ carries a
 natural $\mu$-action, and
 its fixed locus
 $\mathscr{W}^{\mu}$ is canonically isomorphic to $\cG$  \cite[4.1]{HaNi}.
 Denote by $\frak{m}(d)$ the maximal ideal in $R(d)$.
 For every $i$ in $\{0,\ldots,d\}$, the reduction modulo $\frak{m}(d)^i$ defines a morphism of
 $k$-group schemes $$\mathscr{W}^{\mu}_k\to
 \prod_{(R(d)/\frak{m}(d)^i)/k}(\cG(d)\times_{R(d)}(R(d)/\frak{m}(d)^i))$$ whose
 kernel we denote by $F^i_d\cG_k$. In this way, we obtain a
 descending filtration
$$ \mathscr{G}_k = F^0_d \mathscr{G}_k \supset F^1_d \mathscr{G}_k \supset \ldots \supset F^d_d \mathscr{G}_k =
0$$ on $\mathscr{G}_k$ by closed subgroups, and $F^i_d
\mathscr{G}_k$ is a smooth
 connected unipotent $k$-group for all $ i > 0 $ \cite[4.8]{HaNi}.
 The
graded quotients of this filtration are denoted by
$$ \Gr_d^i \mathscr{G}_k = F^i_d \mathscr{G}_k/F^{i+1}_d \mathscr{G}_k.$$

\begin{definition}
We say that $ j \in \{0, \ldots, d-1\} $ is a $K(d)$-jump of $G$
if $\Gr_d^j \mathscr{G}_k \neq 0$.  We denote the set of
$K(d)$-jumps of $G$ by $\mathcal{J}_{G,K(d)}$. The dimension of
$\Gr_d^j \mathscr{G}_k$ is called the multiplicity of the
$K(d)$-jump $j$ and is denoted by $m_{G,K(d)}(j)$. We extend
$m_{G,K(d)}$ to a function
$$m_{G,K(d)}:[0,d[\,\to \N:j\mapsto m_{G,K(d)}(j)$$ by sending $j$ to $0$
if $j$ is not a $K(d)$-jump of $G$, and we call this function the
 $K(d)$-multiplicity function of $G$.
\end{definition}

\sss \label{sss-mudjump} It is explained in \cite[5.3]{edix} and
\cite[4.8]{HaNi} how the function $m_{G,K(d)}$ can be computed
from the action of $\mathrm{Gal}(K(d)/K)\cong \mu_d(k)$ on
$\Lie(\mathcal{G}(d)_k)$: for every $i$ in $\{0,\ldots,d-1\}$, the
value of $m_{G,K(d)}$ at $i/d$ equals the dimension of the maximal
subspace of $\Lie(\mathcal{G}(d)_k)$ where each $\zeta \in
\mu_d(k)$ acts by multiplication with $\zeta^i$.

\sss  In \cite{edix}, Edixhoven also introduced a filtration of
$\mathscr{G}_k$ with rational indices that  captures the
filtrations $F^\bullet_d \mathscr{G}_k$ introduced above
simultaneously for
 all $d$ in $\N'$. It is defined as follows.
 If $p=1$ then we set $\Q'=\Q$, and if $p>1$ we set $\Q'=\Z_{(p)}$.
 For each rational number $\alpha = a/b $ in $\Q'\cap [0,1[$ with $ a \in
\mathbb{N} $ and $ b \in \mathbb{N}' $, we put $$
\mathscr{F}^{\alpha} \mathscr{G}_k = F^a_b \mathscr{G}_k. $$ By
\cite[4.11]{HaNi} this definition does not depend on the choice of
$a$ and $b$, and we obtain in this way a decreasing filtration
$\mathscr{F}^{\bullet} \mathscr{G}_k$ of $\mathscr{G}_k$ by closed
subgroups. Note that only finitely many subgroups occur in the
filtration $\mathscr{F}^{\bullet} \mathscr{G}_k$ since
$\mathscr{G}_k$ is Noetherian. One can define jumps also for this
filtration, in the following way. Let $\rho $ be an element of $
\mathbb{R} \cap [0,1[ $. We put $ \mathscr{F}^{> \rho}
\mathscr{G}_k = \mathscr{F}^{\beta} \mathscr{G}_k $, where $\beta$
is any value in $ \Q' \cap \,]\rho,1[ $ such that $
\mathscr{F}^{\beta'} \mathscr{G}_k = \mathscr{F}^{\beta}
\mathscr{G}_k $ for all $ \beta' \in \Q' \cap\, ]\rho, \beta] $.
 If $\rho \neq 0 $, we put $ \mathscr{F}^{< \rho} \mathscr{G}_k = \mathscr{F}^{\gamma} \mathscr{G}_k $, where $\gamma$ is any value in $ \Q' \cap [0,\rho[ $ such that
   $ \mathscr{F}^{\gamma'} \mathscr{G}_k = \mathscr{F}^{\gamma} \mathscr{G}_k $ for all $ \gamma' \in \Q' \cap [\gamma, \rho[ $.
   We set $ \mathscr{F}^{< 0} \mathscr{G}_k = \mathscr{G}_k $. Then we define
$$ \Gr^{\rho} \mathscr{G}_k = \mathscr{F}^{< \rho} \mathscr{G}_k/\mathscr{F}^{> \rho} \mathscr{G}_k $$
 for every $\rho$ in $\R\cap [0,1[$.
\begin{definition}\label{def-tamebc}
Let $ j$ be an element of $\mathbb{R} \cap [0,1[ $. We say that
$j$ is a jump of $\mathscr{G}_k$ if $ \Gr^{j} \mathscr{G}_k \neq 0
$. We denote the set of jumps of $G$ by $\mathcal{J}_G$. The
multiplicity of $ j $ is the dimension of $ \Gr^{j} \mathscr{G}_k
$ and is denoted by $m_G(j)$. If $j$ is not a jump of $G$, we set
$m_G(j)=0$. In this way, we obtain a function
$$m_G:\mathbb{R} \cap [0,1[\to \N$$ which we call the multiplicity
function of $G$. We define the tame base change conductor of $G$
to be the sum
$$ c_{\tame}(G) = \sum_{j\in \mathcal{J}_G}m_G(j)\cdot j. $$
\end{definition}

\sss \label{sss-jumpslim}  The jumps of $G$ can also be described
as follows. Let
 $K_0\subset K_1\subset \ldots$ be a tower of finite tame extensions of $K$ in $K^s$ that is
 cofinal in the set of all finite tame extensions of $K$ in $K^s$,
 ordered by inclusion. Set $d_n=\dgr{K_n}{K}$ for every $n\in \N$. Then $\mathcal{J}_G$ is the limit of
 the sequence of
 subsets $$\frac{1}{d_n}\cdot \mathcal{J}_{G,K_n}$$ of $[0,1[$ for $n\to
 \infty$. More precisely, if we count the $K_n$-jumps of $G$ with
 their multiplicities and put them in ascending order, we get a
 tuple $$0\leq j_{n,1}\leq \ldots \leq j_{n,g}$$ in
 $\{0,\ldots,d_n-1\}$. For every $i\in \{1,\ldots,g\}$, the sequence
 $(j_{n,i}/d_n)_{n\in \N}$ is a monotonically increasing sequence in
 $[0,1]$, whose limit we denote by $j_i$. Then it is easy to see that $j_i<1$ for every
 $i$ and that
 $$0\leq j_{1}\leq \ldots \leq j_{g}<1$$ are the jumps of $G$,
 counted with multiplicities.

\sss For tamely ramified abelian $K$-varieties, there exists an
interesting relation between the jumps and the Galois action on
the $\ell$-adic Tate module of the abelian variety. We refer to
\cite{HaNi-jumps} for details.

\sss We established in \cite{HaNi} the following relationship
between the base change conductor and the jumps.

\begin{prop} \label{prop-numrel}
For every finite tame extension $K'$ of $K$ in $K^s$ of degree
$d$, the tuple
$$(c_1(G,K')\cdot d,\ldots,c_g(G,K')\cdot d)$$ is equal to the
tuple of $K'$-jumps of $G$, if we count every $K'$-jump with its
multiplicity and put them in ascending order. In particular,
$$ c(G,K') \cdot d = \sum_{j\in \mathcal{J}_{G,K'}} m_{G,K'}(j)\cdot j. $$
If $G$ is tamely ramified, then the tuple
$$(c_1(G),\ldots,c_g(G))$$ is equal to the
tuple of jumps of $G$, if we count every jump with its
multiplicity and put them in ascending order. In particular,
$$ c(G)= c_{\tame}(G). $$
\end{prop}
\begin{proof}
See \cite[4.18 and 4.18]{HaNi}.
\end{proof}
\begin{cor}\label{cor-sup}\item
\begin{enumerate}
\item If $d$ and $d'$ are elements of $\N'$ such that $d$ divides
$d'$, then $$c_i(G,K(d))\leq c_i(G,K(d'))$$ for every $i$ in
$\{1,\ldots,g\}$. In particular, $$c(G,K(d))\leq c(G,K(d')).$$
 \item We have
$$c_{\tame}(G)=\sup_{d\in \N'} c(G,K(d))=\lim_{\stackrel{\longrightarrow}{d\in
\N'}}c(G,K(d))$$ where we order $\N'$ by the divisibility
relation.
\end{enumerate}
\end{cor}
\begin{proof}
This follows from Proposition \ref{prop-numrel} and the
description of the jumps in \eqref{sss-jumpslim}.
\end{proof}
 If $G$ is wildly ramified, it can happen that $ c_{\tame}(G) \neq
c(G) $, as we will illustrate in Examples \ref{ex-compartori},
\ref{ex-elliptic-good} and \ref{ex-elliptic-mult}.

 \sss  As noted by Edixhoven in \cite[5.4(5)]{edix}, it is not
at all clear whether the jumps of a semi-abelian $K$-variety $G$
are always rational numbers. We will come back to this issue in
Part \ref{sec-ques}. Let $L$ be the minimal extension of $K$ in
$K^s$ such that $G\times_K L$ has semi-abelian reduction, and set
$e=[L:K]$. If $L$ is a tame extension of $K$, then it follows from
Proposition \ref{prop-numrel} that $m_G(j/e)=m_{G,L}(j)$ for every
$j$ in $[1,e[$. In
 particular, the jumps belong to the set $(1/e)\Z$. Moreover, in
 that case, one can easily compute the $K'$-jumps of $G$ and their
 multiplicities from the jumps of $G$, for all finite tame
 extensions $K'$ of $K$ \cite[4.13]{HaNi}. The crucial point is
 that the identity component of the N\'eron model of $A\times_K L$
 is stable under finite separable extensions of $L$, by
 \cite[IX.3.3]{sga7.1}. Combining the formula in \cite[4.13]{HaNi} with \cite[4.20]{HaNi}, we see
 that $e$ equals the smallest positive integer $m$ such that $m\cdot j$ is
 integer for every jump $j$ of $G$.
 If $G$ is the (possibly wildly ramified) Jacobian of a smooth proper geometrically connected $K$-curve $C$ of index
 one, then we'll show in Corollary \ref{cor-minindex} that the
 jumps of $G$ are rational, and that
 the stabilization index $e(C)$ defined in Chapter \ref{chap-jacobians} is the smallest positive integer
 $m$ such that $mj$ is integer for every jump $j$ of $G$.

\begin{example}\label{ex-compartori}
Assume that $K$ has equal characteristic two. Let $\alpha$ be an
element of $K$, and assume that the valuation  $v_K(\alpha)$ of
$\alpha$ is odd and strictly negative. We denote by $L$ the
Artin-Schreier extension
 $$L=K[t]/(t^2+t+\alpha)$$ of $K$, and by $R_L$ its valuation ring. We
  denote by $T$ the unique non-split one-dimensional $K$-torus with splitting field
  $L$. It has character module $X(T)\cong \Z$ where the action of
   the generator $\sigma$ of $\Gal(L/K)\cong \Z_2$ on $\Z$ is given by $m\mapsto
   -m$. The torus $T$ is the norm torus of the quadratic extension
   $L/K$.

First, we compute the tame base change conductor $c_{\tame}(T)$.
For notational convenience, we set $s=(1-v_K(\alpha))/2$. This is
an element of $\Z_{>0}$.
 The torus $T$ is isomorphic to
$$\Spec K[u,v]/ (u^2+\alpha v^2+uv+1)$$ with multiplication given
by
$$(u_1,v_1)\cdot (u_2,v_2)=(u_1u_2+\alpha v_1 v_2,u_1v_2+u_2v_1+v_1v_2)$$ and inversion
given by $$(u,v)^{-1}=(u+v,v).$$ Let $\pi$ be a uniformizer in
$K$. For every $d\in \N'$, we choose a uniformizer $\pi(d)$ in
$K(d)$ such that $\pi(d)^d=\pi$. We set
$\alpha'=\pi^{2s-1}\alpha$, $x=\pi(d)^{(d-1)/2-ds}(u+1)$ and
$y=\pi^{1-2s}v$. Then $\alpha'$ is a unit in $R$ and we can
rewrite the equation of $T$ as
\begin{equation*}\label{eq-neron}
\pi(d)x^2+\alpha'y^2+\pi(d)^{ds-(d-1)/2}xy+y=0.
\end{equation*}
 The coefficients of this equation belong to $R(d)$.
 We set
$$\cT(d)=\Spec
R(d)[x,y]/ (\pi(d)x^2+\alpha'y^2+\pi(d)^{ds-(d-1)/2}xy+y).$$ It is
 easily checked that the group structure on $T(d)$ extends to $\cT(d)$.
Moreover, for every point $(x_0,y_0)$ in $T(K(d))$, we have the
following properties.
\begin{enumerate}
\item  If $v_{K(d)}(x_0)< v_{K(d)}(y_0)$, then the terms
$\pi(d)x^2_0$ and $y_0$ must have the same valuation;
 \item If
$v_{K(d)}(x_0)\geq v_{K(d)}(y_0)$, then $y^2_0$ and $y_0$ must
have the same valuation.
\end{enumerate}
Thus the coordinates $(x_0,y_0)$ lie in $R(d)$. Now
 \cite[7.1.1]{neron} implies that
$\mathscr{T}(d)$ is a N\'eron $lft$-model of $T(d)$. We'll write
$\cT$ instead of $\cT(1)$.

 The base
change morphism
$$\mathscr{T}\times_R R(d)\to \mathscr{T}(d)$$ is defined by the
morphism of $R(d)$-algebras $$R(d)[x,y]/
(\pi(d)x^2+\alpha'y^2+\pi(d)^{ds-(d-1)/2}xy+y)\to
 R(d)[x,y]/ (\pi x^2+\alpha'y^2+\pi^{s}xy+y)$$ that sends $x$ to
$\pi(d)^{(d-1)/2}x$ and $y$ to $y$. The $R(d)$-module
$\Lie(\mathscr{T}(d))$ is generated by $\partial/\partial x$, so
that $c(T,K(d))=(d-1)/2$ for every $d$ in $\N'$. Using Corollary
\ref{cor-sup}, we see that
$$c_{\tame}(T)=\sup_{d\in \N'} c(T,K(d))=1/2.$$

Now, we  compute the base change conductor $c(T)$. If we set
$u'=u+tv+1$ and $v'=u+(t+1)v+1$ then
$$\cT^o_L=\Spec R_L[u',v']/(u'v'+u'+v')$$ is the identity component of the
N\'eron model of the split torus $T\times_K L$. The $R_L$-module
$\Lie(\cT^o_L)$ is generated by $\partial/\partial u'$, and the
base change morphism
$$\cT^o\times_R R_L\to \cT^o_L$$ is given by $$u'\mapsto
\pi^sx+t\pi^{2s-1}y\mbox{ and }v'\mapsto \pi^s
x+(t+1)\pi^{2s-1}y.$$ Thus $c(T)=s$. In particular,
$c_{\tame}(T)\neq c(T)$.
\end{example}

\section{Computing the base change conductor}
\subsection{Invariant differential forms}
\sss Let $G$ be a semi-abelian $K$-variety of dimension $g$ with
N\'eron $lft$-model $\mathscr{G}$ and let $K'$ be a finite
separable extension of $K$ with valuation ring $R'$. We will
explain a possible strategy to compute the $K'$-base change
conductor $c(G,K')$ in concrete examples.

\sss For every smooth group scheme $\mathscr{H}$ of relative
dimension $g$ over
 a local scheme $S=\Spec A$, we denote by $e_{\mathscr{H}}$ the unit section
 $S\to \mathscr{H}$ and we set
 $\omega_{\mathscr{H}/S}=e_{\mathscr{H}}^*\Omega^g_{\mathscr{H}/S}$.
 This is a free $A$-module of rank $1$, which can be identified
 with the $A$-module of translation-invariant elements of
 $ \Omega^g_{\mathscr{H}/S} (\mathscr{H})$ by \cite[4.2.1]{neron}. The
 module $\omega_{\mathscr{H}/S}$ is the dual of the determinant of
 the free $A$-module $\Lie(\mathscr{H})$ of rank $g$.

\sss  We set $G'=G\times_K K'$ and we denote by $\mathscr{G}'$ the
N\'eron $lft$-model of $G'$. Let $$ h : \mathscr{G} \times_R R'
\to \mathscr{G}' $$ be the canonical base change morphism. Pulling
back
 the natural morphism $$h^* \Omega^1_{\mathscr{G}'/ R'}  \to
\Omega^1_{\mathscr{G} \times_R R' / R'}$$
 through the
unit section $ e_{\mathscr{G} \times_R R'} $, we obtain a morphism
of $R'$-modules
 $$\phi: e^*_{\mathscr{G}'}\Omega^1_{\mathscr{G}'/ R'}  \to
e^*_{\mathscr{G}\times_R R'}\Omega^1_{\mathscr{G} \times_R R' /
R'}$$ which is the dual of the morphism $\Lie(h)$. Putting
$\Lie(h)$ in Smith normal form, it is easy to see that the
cokernels of $\phi$ and $\Lie(h)$ have the same length, and that
this length is equal to the valuation of the determinant of $\phi$
in $R'$.
 Thus we can also compute the
$K'$-base change conductor of $G$ as \begin{eqnarray*}
 c(G,K')
&=& v_{K'}(\det(\phi))/e(K'/K) \\
&=&\frac{1}{e(K'/K)}\cdot \mathrm{length}_{R'}(\coker(\det(\phi):
\omega_{\mathscr{G}' / R'}\to \omega_{\mathscr{G}/R} \otimes_R
R')) \end{eqnarray*} where $v_{K'}$ denotes the normalized
discrete valuation on $K'$.

\sss In practice, one computes $v_{K'}(\det(\phi))$ as follows:
choose a translation-invariant volume form $\omega$ on $G$ that
extends to a relative volume form on the $R$-scheme $\mathscr{G}$,
and denote by $\omega'$ its pullback to $G'$. Then $\omega'$ will,
in general, not extend to a relative volume form on
$\mathscr{G}'$, but it will have poles along the components of the
special fiber $\mathscr{G}'_k$. The order of the pole does not
depend on the choice of the component, by translation invariance
of $\omega'$, and  it is equal to $-v_{K'}(\det(\phi))$.

 \subsection{Elliptic curves}

\begin{prop}[see also Proposition 3.7.2 in \cite{lu}] \label{prop-c-elliptic}
We denote by $v_K$ the normalized discrete valuation on $K$. If
$E$ is an elliptic curve over $K$ with $j$-invariant $j(E)$ and
minimal discriminant $\Delta$, then
$$
c(E)=\left\{\begin{array}{ll}v_K(\Delta)/12&\mbox{ if }E \mbox{
has potential good reduction},
\\[2ex] (v_K(\Delta)+v_K(j(E)))/12 &\mbox{ if }E
\mbox{ has potential multiplicative reduction}.
\end{array}\right.
$$
 Moreover, for every finite separable extension $K'$ of $K$, we
 have $$c(E,K')=\frac{1}{12}(v_{K}(\Delta)-v_{K'}(\Delta')/e(K'/K))$$
 where $\Delta'$ denotes a minimal discriminant of $E\times_K K'$ and $v_{K'}$ denotes the normalized discrete
valuation on $K'$.
\end{prop}

\begin{proof} Let $K'$ be a finite separable extension of $K$, and set $E'=E\times_K K'$. Let $\mathscr{E}$ be the N\'eron model of $E$.
We consider a Weierstrass equation
\begin{equation*}\tag{$\mathcal E$}\label{eq-Weierstrass}
y^2 + (a_1x + a_3)y = x^3 + a_2x^2 + a_4x + a_6
\end{equation*}
 for $E$, with $ a_i \in R $ for all $i$. To Equation \eqref{eq-Weierstrass}
one associates a translation invariant differential form
$$ \omega = \frac{dx}{2y + a_1x + a_3} \in \Omega_{E/K}^1(E). $$
If we also assume that \eqref{eq-Weierstrass} is a minimal
equation, then $\omega$ extends to a generator of the free rank
one $\mathscr{O}_{\mathscr{E}}$-module
 $ \Omega^1_{\mathscr{E}/R}
$ (see e.g.~\cite[9.4.35]{Liubook}). Moreover, pulling back
$\omega$ via the unit section $e_{\mathscr E} $ yields a generator
of the free rank one $R$-module $\omega_{\mathscr{E}/ R}$.


Let $(\mathcal E')$ be a minimal Weierstrass equation for the
elliptic curve $E'$. It yields an invariant differential form $
\omega' \in \Omega_{E'/K'}^1(E') $ in the same way as above. If we
denote by $\pi$ the projection $E'\to E$, then  $ \omega' = r
\cdot \pi^*\omega $ for some element $ r \in R' =
\mathscr{O}_{\mathscr{E}'}(\mathscr{E}') $, and
$$ c(E,K') = \frac{v_{K'}(r)}{e(K'/K)}. $$  Moreover, if we denote by $\Delta$ and
$\Delta'$ the discriminants of the Weierstrass equations
\eqref{eq-Weierstrass} and ($\mathcal{E}')$, then
$$ v_{K'}(\Delta')=e(K'/K)\cdot v_{K}(\Delta)-12v_{K'}(r) $$
by \cite[III.1.1.2]{silverman}. Thus we find that
$$c(E,K')=\frac{1}{12}(v_{K}(\Delta)-\frac{v_{K'}(\Delta')}{e(K'/K)}).$$

Now assume that $E'$ has semi-abelian reduction. Then
$v_{K'}(\Delta')=0$ if $E'$ has good reduction, and
$v_{K'}(\Delta')=-v_{K}(j(E))\cdot e(K'/K)$ if $E'$ has
multiplicative reduction. This yields the required formula for
$c(E)=c(E,K')$.
\end{proof}

\sss Using Proposition \ref{prop-c-elliptic}, we can easily
compute the base change conductor of a tamely ramified elliptic
curve from the data in the Kodaira-N\'eron reduction table (Table
4.1 in \cite[IV.9]{silverman-advanced}). We will see in Theorem
\ref{Theorem-jumps} that the {\em tame} base change conductor of
an elliptic curve only depends on the reduction type, even for
wildly ramified curves. If $p=1$, then $c_{\tame}(E)=c(E)$ for
every
 elliptic $K$-curve $E$ (Proposition \ref{prop-numrel}), so that  we
obtain the following table of values for the tame base change
conductor (see also \cite[5.4.5]{edix} and
\cite[Table~8.1]{halle-neron}).

\begin{table}[h!]
\begin{tabular*}{0.8\textwidth}
     {@{\extracolsep{\fill}}
     |c||c|c|c|c|c|c|c|c|} \hline     & & & & & & & & \\  Type & $I_{\geq 0}$ &
$II$ & $III$ & $IV$ & $I^*_{\geq 0}$ & $IV^*$ & $III^*$ & $II^*$
\\[1ex] \hline & & & & & & & &
\\ $c_{\tame}(E)$ & 0  &1/6 & 1/4 & 1/3 & 1/2 & 2/3 & 3/4 & 5/6
\\[1ex] \hline
\end{tabular*}
\\[2ex]
\refstepcounter{subsubsection} \caption{The tame base change
conductor for elliptic curves} \label{table-elliptic}
\end{table}

\begin{example}\label{ex-elliptic-good}
 Assume that $k$ is an algebraically closed field of characteristic $2$ and that $R=W(k)$. Let $E$
be the elliptic curve given by the minimal Weierstrass equation
\begin{equation*}
y^2 = x^3 + 2. \end{equation*} As we've already noticed in Example
\ref{ex-diff}, $E$ has reduction type $II$ and potential good
reduction. Therefore, $c_{\tame}(E) = 1/6 $ by Table
\ref{table-elliptic}, and
$$c(E)=v_K(\Delta)/12=6/12=1/2$$ by Proposition
\ref{prop-c-elliptic}.
 In particular, $c(E) \neq c_{\tame}(E) $.
\end{example}

\begin{example}\label{ex-elliptic-mult}
In this example we assume that $K$ has equal characteristic $2$.
We'll construct an elliptic curve $E/K$ with  potential
multiplicative reduction such that $c(E)\neq c_{\tame}(E)$.
Lorenzini has proven in \cite[2.8]{lorenzini-wild} that
   every elliptic curve $E/K$ with additive reduction
 and potential multiplicative reduction has reduction of type $I^*_{\nu + 4s}$
for certain integers $\nu > 0$ and $s > 0$ and acquires
multiplicative
 reduction over a degree two extension of $K$.
  In particular, such a curve $E$ is
wildly ramified and $c_{\tame}(E) = 1/2$ by Table
\ref{table-elliptic}. We will now compute $c(E)$.

By the proof of \cite[Thm.~2.8]{lorenzini-wild}, we can find a
Weierstrass equation
$$ y^2 +xy = x^3 + D x^2 + a_6 $$
over $K$, where $ D = u \pi^{-2s + 1} $ for a suitable unit $ u
\in R $ and $ a_6 \in R $ with $ v_K(a_6) > 0 $. We can use Tate's
algorithm to verify that
\begin{equation*}\tag{$\mathfrak E$}\label{eq-Weierstrass-intK}
y^2 + \pi^s xy = x^3 + \pi^{2s} D x^2 + \pi^{6s} a_6,
\end{equation*}
is a minimal Weierstrass equation for $E$ over $R$. Using
Proposition \ref{prop-c-elliptic}, we compute that $c(E)=s$. Thus
$c(E)\neq c_{\tame}(E)$. Also, note that this example shows that
$c(E)$ can be arbitrarily large, whereas the tame base change
conductor of a semi-abelian variety is always strictly bounded by
the dimension.
\end{example}

\subsection{Behaviour under non-archimedean uniformization}
\sss Another technique that is quite useful to compute the jumps
and elementary divisors of an abelian variety is the use of
non-archimedean uniformization. Let $A$ be an abelian $K$-variety,
and let
$$0\to M\to E^{\an}\to A^{\an}\to 0$$ be its non-archimedean uniformization  as in \eqref{sss-uniformAV} of Chapter \ref{chap-preliminaries}. Thus $E$ is the
 extension of an abelian $K$-variety with potential good reduction
 by a $K$-torus, and $M$ is an \'etale $K$-lattice in $E^{\an}$.

 \begin{prop}\label{prop-bcuniform}
For every finite separable extension $K'$ of $K$, the
$K'$-elementary divisors of $E$ and $A$ coincide, so that
$c(E,K')=c(A,K')$. In particular, the elementary divisors of $E$
and $A$ coincide and $c(E)=c(A)$. Moreover, if $K'$ is tame over
$K$, then the multiplicity functions $m_{E,K'}$ and $m_{A,K'}$ are
equal. Thus $m_E=m_A$ and $c_{\tame}(E)=c_{\tame}(A)$.
 \end{prop}
\begin{proof}
By Proposition \ref{prop-numrel}, we only need to prove the claim
about the $K'$-elementary divisors. This can be done by using the
same arguments as in \cite[5.1]{chai}: if we denote by
$\mathscr{E}$ and $\mathscr{A}$ the N\'eron $lft$-models of $E$
and $A$, and by $\widehat{\mathscr{E}}$ and
$\widehat{\mathscr{A}}$ their formal $\frak{m}$-adic completions,
then the morphism $E^{\an}\to A^{\an}$ extends uniquely to a
morphism of formal $R$-group schemes $\widehat{\mathscr{E}}\to
\widehat{\mathscr{A}}$. This morphism is a local isomorphism by
\cite[2.3]{B-X}, and thus induces an isomorphism of $R$-modules
$$\Lie(\mathscr{E})\to \Lie(\mathscr{A}).$$ Now the result follows
from the fact that $E^{\an}\times_K K'\to A^{\an}\times_K K'$ is
the non-archimedean uniformization of $A\times_K K'$.
\end{proof}

\begin{example}
 Let $E$, $K$ and $D$ be as in Example \ref{ex-elliptic-mult}, and consider the Artin-Schreier extension
 $$ L = K[w]/(w^2 + w + D)$$ of $K$. We claim that $E\times_K L$ has multiplicative reduction.
  To see this, it suffices to make a change of
variables $ y = y' + w x $, which yields an integral
 Weierstrass equation
\begin{equation*}\tag{$\mathfrak{E}'$}\label{eq-Weierstrass-intL}
y'^2 + xy' = x^3 + a_6
\end{equation*}
over the valuation ring of $L$. Tate's algorithm shows that
\eqref{eq-Weierstrass-intL} is minimal and that $E\times_K L$ has
multiplicative reduction. Thus the uniformization morphism of $E$
is of the form $u:T^{\an}\to E^{\an}$, where $T$ is the norm torus
of $L/K$, because this is the unique one-dimensional $K$-torus
with minimal splitting field $L$. The norm torus of $L/K$ is
precisely the torus from Example \ref{ex-compartori}, with
$\alpha=D$. We've seen that, both for $E$ and for $T$, the tame
base change conductor equals $1/2$ and the base change conductor
equals $(1-v_K(D))/2$.
\end{example}

\sss In view of Proposition \ref{prop-bcuniform}, it would be
quite useful to have a formula that expresses the base change
conductor of a semi-abelian $K$-variety in terms of the conductors
of its toric and abelian part, even if one is only interested in
the base change conductors of abelian varieties. The natural guess
is the following.

\begin{conjecture}[Chai]
Assume that $k$ is algebraically closed. If $G$ is a semi-abelian
$K$-variety with toric part $G_\tor$ and abelian part $G_\ab$,
then
$$c(G)=c(G_\tor)+c(G_\ab).$$
\end{conjecture}
If $G$ is tamely ramified, then one can prove this statement in an
elementary way; see \cite[4.23]{HaNi}. In \cite{chai}, Chai proves
 his conjecture when $K$ has
 characteristic zero and also when $k$ is finite. In \cite{CLN},
 the characteristic zero case is reproven by reducing it to a
 Fubini property for motivic integrals.


\section{Jumps of Jacobians}
In this section, we assume that $k$ is algebraically closed.
\subsection{Dependence on reduction data}
\sss  Let $C$ be a smooth, projective, geometrically connected
$K$-curve of genus $g>0$ and index one. We denote by $A$ the
Jacobian of $C$. It is possible to obtain detailed information
concerning the
  jumps of $A$ and their multiplicities in terms of the geometry of $sncd$-models of $C$.
   Let $\mathscr{C}/R$ be an $sncd$-model of $C$, and consider the following condition on $\mathscr{C}$:\\

$ (*) $ \emph{If $E$ and $E'$ are irreducible components of
$\mathscr C_k$ such that $E\cap E'$ is non-empty, then the
multiplicity of $E$ or $E'$ in $\mathscr{C}_k$ is prime to $p$.}\\

It is easy to see that if condition $(*)$ holds for some
$sncd$-model of $C$, then it also holds for the minimal
$sncd$-model. In particular, this is the case if $C$ is tamely
ramified, by Saito's criterion \cite[3.11]{Saito}. In
\cite{halle-neron}, the first author proved the following result.

\begin{theorem} \label{Theorem-jumps*}
 Let $\mathscr{C}/R$ be an $sncd$-model of $C$ and assume
that $\mathscr{C}$ satisfies condition $(*)$. Then the jumps of
$A$ and their multiplicities only depend on the combinatorial data
of $ \mathscr{C}_k $.
\end{theorem}
\begin{proof}
See \cite[8.2]{halle-neron}.
\end{proof}

We will now show that the condition $(*)$ can be omitted, which
yields the following stronger theorem.

\begin{theorem} \label{Theorem-jumps} Let $C$ be a smooth, projective, geometrically
connected $K$-curve of genus $g>0$ and index one. We denote by $A$
the Jacobian of $C$.
 Let $\mathscr{C}/R$ be an $sncd$-model of $C$. Then the jumps of
$A$ and their multiplicities only depend on the combinatorial data
of $ \mathscr{C}_k $. In particular, they do not depend on the
characteristic exponent $p$ of $k$.
\end{theorem}
\begin{proof} We will indicate how the proof of
\cite[8.2]{halle-neron} can be generalized. We write
$\mathscr{C}_k = \sum_{i \in I} N_i E_i $.
 For every $d$ in $\N'$, we denote by $\mathscr{C}(d)$ the minimal
 desingularization of the normalization $\mathscr{C}_d$ of
 $\mathscr{C}\times_R R(d)$ as in Chapter \ref{chap-jacobians},
 \eqref{sss-desingsetup}. Then $\mathscr{C}(d)$ is an $sncd$-model
 of $C\times_K K(d)$, by Proposition \ref{prop-normdesing} in Chapter \ref{chap-jacobians}.
  The jumps in
$\mathscr{F}^{\bullet} \mathscr{A}_k$ can be computed if one has a
sufficiently good description of the Galois action on the
$sncd$-models $\mathscr{C}(d)$ as $d$ varies in $\N'$. Indeed,
there is a natural $\Gal(K(d)/K)$-action on $ C \times_K K(d) $
which extends uniquely to $ \mathscr{C}(d) $, in such a way that
the natural isomorphism
$$ \Pic^0_{\mathscr{C}(d)/R(d)} \cong \mathscr{A}(d)^o $$
is equivariant. This yields an equivariant isomorphism
$$ H^1(\mathscr{C}(d)_k, \mathscr{O}_{\mathscr{C}(d)_k}) \cong \Lie(\mathscr{A}(d)_k) $$
(cf.~\cite[2.4]{halle-neron}).

Moreover, it is not necessary to treat \emph{all} $d \in \N'$. Let
us put $$ \lambda(\mathscr{C}) = \lcm \{ N_i \mid i \in I \}. $$
In order to prove Theorem \ref{Theorem-jumps}, it suffices to
show, for every $ d \in \mathbb{N}' $ prime to
$\lambda(\mathscr{C})$, that the $K(d)$-jumps and their
multiplicities only depend on the combinatorial structure of $
\mathscr{C}_k $. This follows from the fact that the set $
\mathbb{Z}_{(\lambda(\mathscr{C}))}\cap \Q' \cap [0,1] $ is dense
in $ \Q' \cap [0,1] $. The reason for restricting to these values
$d$ is that both the geometry of $\mathscr{C}(d)$ and the
$\Gal(K(d)/K)$-action can be described sufficiently well. This is
not clear for general $ d \in \N' $.

 Thus let $d$ be an element of $\N'$ that is prime to
$\lambda(\mathscr{C})$.
 Using Lemma \ref{lemma-fiberdescription} in Chapter \ref{chap-jacobians} and
 the results in Section  \ref{sec-loctor} of that chapter, it is easy to check that both the statement and proof of \cite[3.4]{halle-neron} extend directly to our situation.
  This means that the following properties hold.

\begin{enumerate}
\item Every irreducible component $F$ of $\mathscr{C}(d)_k$ is
stable under the $\Gal(K(d)/K)$-action.

\item $\Gal(K(d)/K)$ acts trivially on $F$ unless $F$ belongs to
the exceptional locus of the minimal desingularization $ \rho :
\mathscr{C}(d) \to \mathscr{C}_d $.

\item Every intersection point $ x \in F \cap F' $ of distinct
irreducible components of $\mathscr{C}(d)_k$ is a fixed point for
$\Gal(K(d)/K)$.
\end{enumerate}

Then in order to compute the irreducible components of the
 action of $\Gal(K(d)/K) \cong \mu_d(k)$ on
$$H^1(\mathscr{C}(d)_k, \mathscr{O}_{\mathscr{C}(d)_k}),$$ it
suffices to determine the irreducible components of the
 action of $\Gal(K(d)/K)$ on the cotangent space $\Omega_x$ of $\mathscr{C}(d)$ at $x$ for
 each intersection point $x \in F \cap F'$ in the special fiber
$\mathscr{C}(d)_k$. In fact, the proof in \cite{halle-neron}
extends directly to our situation.


 The Galois action on the cotangent space $\Omega_x$ can be
 computed on the completed local ring
 $\widehat{\mathscr{O}}_{\mathscr{C}(d),x}$.
The point $x$ maps to an intersection point $y$ in the special
fiber of $\mathscr{C}_d$. We gave an explicit description of
$\widehat{\mathscr O}_{\mathscr{C}_d, y}$ in Lemma
\ref{lemma-invariantring}, and with this presentation it is
straightforward to describe the $\Gal(K(d)/K)$-action, similarly
as in \cite[\S4.1]{halle-neron}. Moreover, using the explicit
description of the minimal desingularization
$$ \rho_y : \mathscr{Z}_y \to \Spec \widehat{\mathscr O}_{\mathscr{C}_d, y} $$
in Section \ref{sec-loctor} of Chapter \ref{chap-jacobians}, it is
easy to describe the unique lifting of the $\Gal(K(d)/K)$-action
to $\mathscr{Z}_y$, and we can use the same arguments as in
\cite[8.2]{halle-neron}.
\end{proof}

\begin{cor}\label{cor-jumps} We keep the notations of Theorem
\ref{Theorem-jumps}. For every $ d \in \mathbb{N}' $, the
$K(d)$-jumps of $A$ and their multiplicities only depend on the
combinatorial data of the special fiber $\mathscr{C}_k$.
\end{cor}
\begin{proof}
An integer $ i \in \{0, \ldots, d-1 \} $ is a $K(d)$-jump of
multiplicity $m>0$ of $A$ if and only if
$$ \dim (F^i_d \mathscr{A}_k) - \dim (F^{i+1}_d \mathscr{A}_k) = m. $$
But by construction, $ F^k_d \mathscr{A}_k = \mathscr{F}^{k/d}
\mathscr{A}_k $ for every $ k \in \{0, \ldots, d-1 \} $, so the
result follows from Theorem \ref{Theorem-jumps}.
\end{proof}

\begin{cor} \label{cor-minindex} We keep the notations of Theorem
\ref{Theorem-jumps}. The jumps of $A$ are rational numbers, and
the stabilization index $e(C)$ is the smallest integer $ n >0 $
such that $ n \cdot j$ is integer for every jump $j$ of $A$. In
particular, $e(C)$ only depends on the
 abelian $K$-variety $A$, and not on $C$.
\end{cor}
\begin{proof}
Let $\mathscr C$ be the minimal $sncd$-model of $C$. By
 Theorem \ref{theo-winters} we can find a smooth, projective and
geometrically connected $\mathbb{C}((t))$-curve $D$ of genus $g$
and an $sncd$-model $\mathscr{D}/\mathbb{C}[[t]]$ of $D$ such that
the special fibers $\mathscr{C}_k$ and $\mathscr{D}_{\C}$ have the
same combinatorial data. Then $\mathscr{D}$ is the minimal
$sncd$-model of $D$ and $ e(C) = e(D) $. Set $B = \Jac(D)$. By
Theorem \ref{Theorem-jumps}, the jumps of $B$ and their
multiplicities are the same as those of $A$. Thus we may assume
that $K=\C((t))$. Then $A$ is tamely ramified, and $e(A)$ equals
the degree of the minimal extension of $\mathbb{C}((t))$ where $A$
acquires stable reduction, by Proposition \ref{prop-min}. The
result now follows from \cite[2.4]{halle-neron}.
\end{proof}

 The following corollary shows how the jumps of $A=\Jac(C)$ are
 related to the characteristic polynomial $P_C(t)$ that we introduced in
 Chapter \ref{chap-jacobians}, Definition \ref{def-charpol}.
\begin{cor} \label{cor-roots} We keep the notations of Theorem
\ref{Theorem-jumps}. For every jump $j$ of $A$, the value
$\exp(2\pi i j)$ is a root of the characteristic polynomial
$P_C(t)$ of multiplicity at least $m_A(j)$.
\end{cor}
\begin{proof}
Let $\mathscr C$ be the minimal $sncd$-model of $C$. It is clear
that $P_C(t)$ can be computed from the combinatorial data of
$\mathscr{C}_k$, and by Theorem \ref{Theorem-jumps}, the same is
true for the jumps of $A$ and their multiplicities. Thus we can
reduce to the case $K=\C((t))$ as in the proof of Corollary
\ref{cor-minindex}.

Let $\sigma$ be a topological generator of the absolute Galois
group $\Gal(K^s/K)$. We know by Proposition \ref{prop-charpol}
that $P_C(t)$ is the characteristic polynomial of the action of
$\sigma$ on
$$H^1(C\times_K K^s,\Q_\ell)\cong (T_\ell
A)^{\vee}\otimes_{\Z_\ell}\Q_\ell$$ so that the result follows
from \cite[5.14]{HaNi}.
\end{proof}

\chapter{The base change conductor and the Artin
conductor}\label{chap-artin} In this chapter, we assume that $k$
is algebraically closed. We will compare the base change conductor
of the Jacobian variety of a $K$-curve $C$ to Saito's Artin
conductor of $C$ and other invariants of the curve, assuming that
the genus of $C$ is $1$ or $2$.

\section{Some comparison results}
\subsection{Algebraic tori}
\sss Let $T$ be an algebraic torus over $K$. We denote its
cocharacter module by $X_{\!*}(T)$. The Artin conductor $\Art(T)$
of $T$ is defined as the Artin conductor of the
$\Gal(K^s/K)$-module $X_{\!*}(T)\otimes_{\Z}\Q$. Likewise, we
define the Swan conductor $\mathrm{Sw}(T)$ as the wild part of the
Artin conductor of $X_{\!*}(T)\otimes_{\Z}\Q$.

\sss A deep result of Chai and Yu  \cite[\S11 and \S12]{chai-yu}
states that the base change conductor $c(T)$ of $T$ is related to
the Artin conductor by the formula
\begin{equation}\label{eq-artin}
c(T)=\frac{\Art(T)}{2}.\end{equation} In particular, $c(T)$ is
invariant under isogeny. This is not the case for the elementary
divisors of $T$, which depend on the
 integral structure of the Galois module $X_{\!*}(T)$ \cite[8.5]{chai}.
 We will now give a similar interpretation of the tame base change
 conductor of $T$.

 \begin{prop}\label{prop-ctame}
Let $T$ be an algebraic $K$-torus of dimension $g$ with character
module $X(T)$. Denote by $\mathrm{Sw}(T)$ the wild part of the
Artin conductor $\Art(T)$. The tame part
$$\Art(T)-\mathrm{Sw}(T)=g-\mathrm{rank}_{\Z}X(T)^I$$ of the Artin conductor is equal to the
unipotent rank $u(T)$ of $T$, i.e., the unipotent rank of the
identity component of the special fiber of the N\'eron model of
$T$. Moreover, ordering $\N'$ by the divisibility relation, we
have the following equalities:
\begin{eqnarray}
 \frac{\mathrm{Sw}(T)}{2}&=& \lim_{\stackrel{\longrightarrow}{d\in \N'}}
\frac{c(T(d))}{d}, \label{eq-ctame2}
\\[2ex]
\frac{u(T)}{2}&=&c_{\tame}(T). \label{eq-ctame1}
\end{eqnarray}
 \end{prop}

 \begin{proof}
 It is well-known that $$u(T)=g-\mathrm{rank}_{\Z}X(T)^I.$$ A
 proof can be found,
 for instance, in \cite[3.14]{HaNi-comp}.  The equality
 \eqref{eq-ctame2} is a direct consequence of \eqref{eq-artin} and Proposition
 \ref{prop-swan} in Chapter \ref{chap-preliminaries}. The equality
  \eqref{eq-ctame1}  follows immediately from \eqref{eq-ctame2}, using
  \eqref{sss-addbc} in Chapter \ref{chap-conductor} and Corollary \ref{cor-sup}.
 \end{proof}

\subsection{Saito's discriminant-conductor formula}
\sss For abelian varieties,
 the base change conductor can change under isogenies
 \cite[6.10]{chai}, so that we cannot hope to express the base change conductor in terms of
  the $\Q_{\ell}$-adic Tate module -- except in some special cases; see for
 instance \cite[5.2]{chai} for the case of potential multiplicative reduction.

\sss Nevertheless, it is possible to rewrite the formula for the
base change conductor of an elliptic curve in Proposition
\ref{prop-c-elliptic} in such a way that the Artin conductor of
the curve appears in the expression. The proper notion of Artin
conductor to use in this setting is the one from Saito's
influential paper \cite{saito-cond}. If $\mathscr{X}$ is a regular
proper flat $R$-scheme and $\ell$ is a prime different from $p$,
then the $\ell$-Artin conductor of $\cX$ is defined as
$$\Art_\ell(\cX)=\chi(\cX_K)-\chi(\cX_k)+\sum_{i\geq
1}(-1)^i \mathrm{Sw}_R H^i(\cX_K\times_K K^s,\Q_\ell)$$ where
$\mathrm{Sw}_R$ denotes the Swan conductor. If $\cX$ has relative
dimension at most one over $R$, then it is known that the Artin
conductor is independent of $\ell$; it will be simply denoted by
$\Art(\cX)$. When $C$ is a geometrically connected smooth
projective $K$-curve of genus $g\geq 1$, then we set
$$\Art(C)=\Art(\cC)$$ where $\cC$ is the minimal regular $R$-model
of $\cC$. The invariant $\Art(C)$ vanishes if and only if $C$ has
good reduction, except if $C$ is a genus one curve without
rational point whose Jacobian has good reduction. We define the
tame part of the Artin conductor of $C$ by
$$\Art_{\tame}(C)=\chi(C)-\chi(\cC_k).$$

\sss If $K'$ is a finite separable extension of $K$ with valuation
ring $R'$, and if $S=\Spec R'$, then
$$\Art(S)=\dgr{K'}{K}-1+\mathrm{Sw}_R \Q_\ell[K'/K]$$
which is the usual Artin conductor of the extension $K'/K$. Thus
$\Art_\ell(\cX)$ is a natural generalization of the Artin
conductor for finite separable extensions of $K$. The classical
discriminant-conductor formula states that
$$\Art(S)=v_K(\Delta_{K'/K})$$ where $\Delta_{K'/K}$ denotes the
discriminant of the extension $K'/K$. Saito generalized this
result to curves $C$ over $K$, where the valuation of the
discriminant is defined by measuring the degeneration of a certain
morphism of rank one free $R$-modules that is an isomorphism over
$K$ \cite[Thm.1]{saito-cond}. Here we will only consider elliptic
and hyperelliptic curves, in which case the valuation of the
discriminant can be defined in a more explicit way.

\section{Elliptic curves}
\subsection{The potential degree of degeneration}
\sss For elliptic curves, the valuation of Saito's discriminant
coincides with the
 valuation of the classical minimal discriminant of a Weierstrass equation \cite[\S1,\,Cor.2]{saito-cond}.
 Let $E$ be an elliptic $K$-curve with minimal discriminant
$\Delta$ and $j$-invariant $j(E)$, and let $\mathscr{X}$ be the
minimal regular $R$-model of $E$. Then
$$\Art(E)= -\chi(\mathscr{X}_k)-\delta(E)$$ where $\delta(E)$ denotes the
 wild part of the conductor of $E$ (i.e., the Swan conductor of $T_\ell E\otimes_{\Z_\ell}\Q_\ell$). Since $\Art(E)$
 vanishes if and only if $E$ has good reduction, it is not
 reasonable to expect that we can compute $c(E)$ directly from
 $\Art(E)$ (recall that $c(E)$ vanishes as soon as $E$ has
 semi-abelian reduction). There should be a second term involved that measures the potential degeneration. A natural candidate is
 the invariant $\frak{d}_{\pot}(E)$ that we define as follows.

\begin{definition}
Let $E$ be an elliptic curve over $K$, and let $L$ be a finite
separable extension of $K$ such that $E\times_K L$ has
semi-abelian reduction. We define the potential degree of
degeneration $\frak{d}_{\pot}(E)$ as
$$\frak{d}_{\pot}(E) =\left\{\begin{array}{ll}0 & \mbox{if $E\times_K L$ has good reduction},
\\[2ex] \displaystyle \frac{1}{\dgr{L}{K}}\cdot |\Comp{E\times_K L}| &\mbox{if $E\times_K L$ has multiplicative reduction.} \end{array}\right.$$
\end{definition}

 This definition does not depend on the choice of $L$, by the following lemma.
\begin{lemma}\label{lemm-jcomp}
For every elliptic $K$-curve $E$, we have
$$\frak{d}_{\pot}(E)=\left\{\begin{array}{ll}0&\mbox{ if }E \mbox{
has potential good reduction},
\\[2ex] -v_K(j(E)) &\mbox{ if }E
\mbox{ has potential multiplicative reduction}.
\end{array}\right.$$
\end{lemma}
\begin{proof}
This follows immediately from the equality $|\Comp{E}|=-v_K(j(E))$
for elliptic $K$-curves $E$ with multiplicative reduction
\cite[9.2(d)]{silverman-advanced}.
 \end{proof}

\subsection{A formula for the base
change conductor} We will now explain how to compute
 $c(E)$ from  $\Art(E)$ and $\frak{d}_{\pot}(E)$.
 \sss
  Saito's discriminant-conductor formula
  for curves \cite[Thm.1]{saito-cond} implies that
 $$v_K(\Delta)=-\Art(E)=\chi(\mathscr{X}_k)+\delta(E),$$
 which can be rewritten as Ogg's formula
 $$v_K(\Delta)= f(E/K)+m-1$$ where $f(E/K)$ is the
 conductor of $E$ and $m$ denotes the number of irreducible
 components of $\mathscr{X}_k$. The equality
 $$-\Art(C)=f(C/K)+m-1$$ holds for curves of arbitrary genus
 $g\geq 1$, by \cite[Prop.\,1]{liu-genre2}; for elliptic curves,
 it can also  easily be checked by a case-by-case computation on the
 Kodaira-N\'eron reduction table.

 \sss
Using Saito's discriminant-conductor formula, we can rewrite the
formula for $c(E)$ in
 the following way.
 \begin{prop}\label{prop-c-ell-artin} For every elliptic $K$-curve $E$, we have
 $$c(E)=-\frac{1}{12}(\Art(E)+\frak{d}_{\pot}(E)).$$
\end{prop}
\begin{proof} Immediate from Proposition
\ref{prop-c-elliptic} and Lemma \ref{lemm-jcomp}.
\end{proof}

\if false
 \sss The formula
$$-\Art(E)= f(E/K)+m-1$$
 makes it easy to compute the
base change conductor for tamely ramified elliptic curves, since a
tamely ramified elliptic curve with additive reduction has
potential multiplicative reduction if and only if it has reduction
type $I^*_n$ for some $n>0$. In particular, we see that the base
change conductor of a tamely ramified elliptic curve only depends
on the reduction type, and not on the residue characteristic. \fi

 Using this result, we can also interpret the wild part of the
 base change conductor, i.e., the difference
$c(E)-c_{\tame}(E)$, when $E$ is wildly ramified.

\begin{prop}\label{prop-diffcctame}
Let $E$ be an elliptic curve over $K$, and denote by $\delta(E)$
the wild part of the conductor of $E$.
\begin{enumerate}
\item We have
\begin{eqnarray*}c_{\tame}(E)&=&-\frac{1}{12}(\Art_{\tame}(E)+\frak{d}_{\pot}(E)),
\\[1.5ex] c(E)-c_{\tame}(E)&=&\frac{\delta(E)}{12},\end{eqnarray*}
 unless
$p=2$ and $E$ has reduction type $I_n^*$ for some $n>0$. \item
Assume that $p=2$ and that $E$ has reduction type $I_n^*$ for some
$n>0$. Then $E$ is wildly ramified. If $E$ has potential good
reduction, then
$$c(E)-c_{\tame}(E)=\frac{1}{12}(\delta(E)+n).$$ If $E$ has potential
multiplicative reduction, then
$$c(E)-c_{\tame}(E)=\frac{\delta(E)}{4}.$$ Moreover, in the latter case, $$v_K(j(E))=2\delta(E)-n\mbox{ and }c(E)=\frac{1}{4}(\delta(E)+2).$$
 \item We have $c(E)=c_{\tame}(E)$ if and only if $E$ is
tamely ramified.
\end{enumerate}
\end{prop}
\begin{proof}
Let $E'$ be an elliptic curve over $\C((t))$ with the same
reduction type as $E$. Theorem \ref{Theorem-jumps} tells us that
$c_{\tame}(E)=c_{\tame}(E')$, and we know that
$c_{\tame}(E')=c(E')$ because $E'$ is tamely ramified. Thus
$$c(E)-c_{\tame}(E)=c(E)-c(E').$$

(1)  First, assume that $p\neq 2$. Then $E$ and $E'$ have the same
potential reduction type, since $E$ has potential multiplicative
reduction if and only if it is of type $I_n$ or $I_n^*$ with
$n>0$, in which case $\frak{d}_{\pot}(E)=n$ (this can be deduced
from Table 4.1 in \cite[IV]{silverman}).
 Proposition \ref{prop-c-ell-artin} now yields
\begin{eqnarray*}
c_{\tame}(E)&=&c(E')=-\frac{1}{12}(\Art_{\tame}(E)+\frak{d}_{\pot}(E)),
\\[1.5ex] c(E)-c(E')&=&\frac{1}{12}\cdot \mathrm{Sw}_R H^1(E\times_K
 K^s,\Q_\ell)=\frac{\delta(E)}{12}.\end{eqnarray*}

 The case $p=2$ is more delicate, because there exist elliptic
 curves with reduction type $I^*_{n}$, $n>0$, and potential good
 reduction. The elliptic curves with potential multiplicative
 reduction have been classified by Lorenzini in
 \cite[2.8]{lorenzini-wild}. His result implies that each such curve $E$ has reduction type $I^*_n$ for
 some  $n>0$ and that $\frak{d}_{\pot}(E)=n$. Thus if $E$ does not
 have reduction type $I^*_{>0}$,
 our argument
 from the $p\neq 2$ case again yields
\begin{eqnarray*}
c_{\tame}(E)&=&-\frac{1}{12}(\Art_{\tame}(E)+\frak{d}_{\pot}(E)),
\\[1.5ex] c(E)-c(E')&=&\frac{\delta(E)}{12}.\end{eqnarray*} This concludes the proof of (1).

 (2) Assume that $p=2$ and that $E$
has reduction type $I_n^*$ for some $n>0$. Then $E$ is wildly
ramified, by Saito's criterion \cite[3.11]{Saito}.
 If $E$ has potential
good reduction, Proposition \ref{prop-c-ell-artin} yields
$$c(E)-c(E')=\frac{1}{12}(\delta(E)+\frak{d}_{\pot}(E'))=\frac{1}{12}(\delta(E)+n).$$ If $E$ has potential multiplicative
reduction, then
 we have
$\frak{d}_{\pot}(E)=-v_K(j(E))$ so that
$$c(E)-c(E')=\frac{1}{12}(\delta(E)+n+v_K(j(E))).$$
 Moreover, $c(E')=1/2$ by Table \ref{table-elliptic}, and it follows from  \cite[5.2]{chai} that
$$c(E)=\frac{2+\delta(E)}{4}.$$ Thus
$$c(E)-c(E')=\frac{\delta(E)}{4}\mbox{ and
}v_K(j(E))=2\delta(E)-n.$$ (3) This is obvious, since
$\delta(E)>0$ when $E$ is wildly ramified.
\end{proof}

%

\section{Genus two curves}
\label{sec-genus2}
 Looking at the formula in Proposition \ref{prop-c-ell-artin},
 it is natural to wonder if there exists a similar formula for higher genus curves,
 expressing the base change conductor of the Jacobian in terms of
 the Artin conductor of the curve and a suitable measure of
 potential degeneration.
 To conclude this chapter we will now briefly discuss the case of
 genus $2$ curves, and we will see that already there the
 situation gets more delicate.  Let $C$ be a smooth geometrically connected projective
$K$-curve of genus $2$. We will describe the relationship between
the base change conductor of the Jacobian of $C$ and the minimal
discriminant and  Artin conductor of $C$. Detailed proofs, as well
as generalizations to curves of higher genus, will be given in
future work.

\subsection{Hyperelliptic equations}
\sss We start by recalling a few facts that can be found in
\cite{liu-genre2}. By a \emph{hyperelliptic equation} for $C$, we
mean an affine equation
\begin{equation}\tag{$\mathcal{E}$}\label{eq-hyperell} y^2 + Q(x) y = P(x)
\end{equation} where $ P, Q \in K[x] $ satisfy $ \deg(Q) \leq 3 $
and $ \deg(P) \leq 6 $, and with the property that
$$\Spec~K[x,y]/(y^2 + Q(x) y - P(x)) $$
is isomorphic to an open dense subset of $C$. Such an equation can
always be found.

\sss The \emph{discriminant} of \eqref{eq-hyperell} is defined as
$$ \Delta(\mathcal{E}) = 2^{-12} \mathrm{disc}_6(4P(x) + Q(x)^2) $$
where $\mathrm{disc}_6(4P(x) + Q(x)^2)$ denotes the discriminant
of $4P(x) + Q(x)^2$ considered as a polynomial of degree $6$ in
the variable $x$.

Also associated to $ (\mathcal{E})$ are the
 differential forms
$$  \omega_i(\mathcal{E}) = \frac{x^{i-1}dx}{2y + Q(x)} $$
for $ i \in \{1, 2 \}$, which form a $K$-basis of $ H^0(C,\Omega_C^1) $.

\sss If
\begin{equation}\tag{$\mathcal{E}'$}\label{eq-hyperell2}
 z^2 + Q'(v) z = P'(v)
 \end{equation}
 is another
hyperelliptic equation for $C$, there exists an element $ u \in K $ such that
\begin{equation*}\label{eq-traf1}
\Delta(\mathcal{E}') = u^{10} \cdot \Delta(\mathcal{E})
\end{equation*}
and
\begin{equation*}\label{eq-traf2}
\omega_1(\mathcal{E}') \wedge \omega_2(\mathcal{E}') = u^{-1} \cdot \omega_1(\mathcal{E}) \wedge \omega_2(\mathcal{E}).
\end{equation*}
Thus the element
$$ \Delta(\mathcal{E}) (\omega_1(\mathcal{E}) \wedge \omega_2(\mathcal{E}))^{\otimes 10} \in (\wedge^2 H^0(C, \Omega^1_{C/K}))^{\otimes 10} $$
is independent of the choice of hyperelliptic equation.

\subsection{Minimal equations}\label{seq-mineq}
\sss Let $\mathscr X $ be a regular model of $C$ over $R$, and
denote by $ \omega_{\mathscr X/R} $ its relative canonical sheaf.
Then
 $H^0(\mathscr X, \omega_{\mathscr X/R})$ is a free
$R$-module of rank $2$, and we have a canonical morphism of
$K$-vector spaces
$$ H^0(\mathscr X, \omega_{\mathscr X/R})\otimes_R K \to H^0(C, \Omega^1_{C/K}) $$
so that we can view $H^0(\mathscr X, \omega_{\mathscr X/R})$ as an
$R$-lattice in $H^0(C, \Omega^1_{C/K})$.
\begin{definition}[Definition 1 of \cite{liu-genre2}]
Let $ \mathscr C $ be the minimal regular model of $C$. We say
that \eqref{eq-hyperell} is a minimal equation of $C$ if $
\{\omega_1(\mathcal{E}), \omega_2(\mathcal{E})\}$ is an $R$-basis
of the lattice $H^0(\mathscr C, \omega_{\mathscr C/R})$. In this
case, we call the integer $  v_K(\Delta(\mathcal{E})) $ the
minimal discriminant of $C$, and we denote it by
$v(\Delta(C)_{\mathrm{min}})$.
\end{definition}

\sss By \cite[Prop.~2]{liu-genre2} one can always find a minimal
equation $ (\mathcal{E}) $. The definition of $
v(\Delta(C)_{\mathrm{min}}) $ is independent of choice of such $
(\mathcal{E}) $.

\sss \label{sss-mindiscbc} Let $ K'/K $ be a finite separable
field extension and let $ (\mathcal{E}) $ be a minimal equation
for $C$. It can also be viewed as a hyperelliptic equation for $
C' = C \times_K K' $,
 but over $K'$, the equation $ (\mathcal{E}) $ may no longer be minimal. If $
(\mathcal{E}') $ is a minimal equation for $C'$, we have
\begin{equation}\label{eq-disc-change}
\Delta(\mathcal{E}') = r^{10} \cdot \Delta(\mathcal{E})
\end{equation}
and
\begin{equation}
\omega_1(\mathcal{E}) \wedge \omega_2(\mathcal{E}) = r \cdot \omega_1(\mathcal{E}') \wedge \omega_2(\mathcal{E}')
\end{equation}
for a certain element $ r \in K' $.

\subsection{Comparison of the base change conductor and the minimal discriminant}\label{sec-comp-disc-cond}

\sss Assume that the curve $C$ has index one. We put $ A =
Jac(C)$, and we denote by $\mathscr A$ the N\'eron model of $A$.
Let moreover $ f : \mathscr C \to \Spec R $ be the minimal regular
model of $C$. One can show that there exists a canonical
isomorphism
$$ \omega_{\mathscr A/R} \to f_* (\omega_{\mathscr{C}/R}), $$
where $\omega_{\mathscr A/R}$ is the module of invariant
differential forms associated to $\mathscr A$. In particular, the
differentials $\omega_1(\mathcal{E}) $ and $ \omega_2(\mathcal{E})
$ form an $R$-basis of $\omega_{\mathscr A/R}$, and using this
fact it is straightforward to see that, for every finite separable
extension $K'/K$, we can compute the $K'$-base change conductor as
$$c(A,K') = \frac{1}{[K':K]} \cdot v_{K'}(r)$$ where the element
$r$ of $K'$ is defined as in  \eqref{sss-mindiscbc}.

\sss We assume now in addition that $ A $ acquires semi-abelian
reduction over $K'$, which is equivalent to the property that
$C'=C\times_K K'$ has a semi-stable model over $R'$. Equation
(\ref{eq-disc-change}) yields the following formula:
 \begin{equation}\label{eq-mindisc}
v(\Delta(C)_{\min}) = 10 \cdot c(A) + v(\Delta(C')_{\min})/[K':K].
\end{equation}


We next compute $v(\Delta(C')_{\min})$ purely in terms of a
semi-stable reduction of $C'$. In order to do this, recall that
Saito  introduced in \cite{saito-genus2} three ``discriminants'' $
\Delta(C) $, $\Delta'(C)$ and $\Delta_1(C)$ to which he associated
three values $\mathrm{ord}~\Delta(C)$, $\mathrm{ord}~\Delta'(C)$
and $\mathrm{ord}~\Delta_1(C)$ in $\Z$. Moreover, he showed that
$$ \mathrm{ord}~\Delta(C) = - \mathrm{Art}_{\mathscr C/S} $$
and that
\begin{equation}\label{eq-saito}
\mathrm{ord}~\Delta'(C) = - \mathrm{Art}_{\mathscr C/S} +
\mathrm{ord}~\Delta_1(C).
\end{equation}
 On the other hand, Liu showed in \cite{liu-genre2}  that
\begin{equation}\label{eq-liu}
\mathrm{ord}~\Delta'(C) = v(\Delta(C)_{\min}).
\end{equation}

Let $\mathscr{C}'$ denote the minimal regular model of $ C' $. By
our assumption, the special fiber $\mathscr{C}'_k $ is a
semi-stable curve. We denote by $\sigma$ the cardinality of the
set $ \mathrm{Sing}(\mathscr{C}'_k)$ of singular points on
$\mathscr{C}'_k$. It is easy to verify that
\begin{equation}\label{eq-semi-stable-cond}
- \mathrm{Art}_{\mathscr{C}'/R'} = \sigma
\end{equation}
(see for instance \cite[Prop.~1]{liu-genre2}). Let moreover $\tau$
denote the cardinality of the set of points $ p \in
\mathrm{Sing}(\mathscr{C}'_k) $ such that $ \mathscr{C}'_k - \{p\}
$ is disconnected. Then the proposition on page 234 of
 \cite{saito-genus2} tells us that
\begin{equation}\label{eq-semi-stable-saito}
\mathrm{ord}~\Delta_1(C') = \tau.
\end{equation}

Combining Equations (\ref{eq-mindisc}) -
(\ref{eq-semi-stable-saito}) now yields an interesting
relationship between the minimal discriminant and the base change
conductor. Note the similarity to the formula for elliptic curves
in Proposition \ref{prop-c-ell-artin}: here the value
$\sigma+\tau$ plays the role of the degree of potential
degeneration. In particular, it vanishes if and only if $C$ has
potential good reduction.

\begin{prop}\label{prop-formula-disc-cond}
With the notation introduced above, we have
$$ v(\Delta_{\min}) = 10 \cdot c(A) + (\sigma + \tau)/[K':K]. $$
\end{prop}

\sss The relationship between the base change conductor $c(A)$ and
the Artin conductor $ \mathrm{Art}_{\mathscr{C}/S} $  is more
subtle. By \eqref{eq-saito} and \eqref{eq-liu}, the difference
between $ v(\Delta(C)_{\min}) $ and $ -
\mathrm{Art}_{\mathscr{C}/S} $ is precisely measured by $
\mathrm{ord}~\Delta_1 $. In \cite{liu-genre2}, Liu shows that $
\mathrm{ord}~\Delta_1 $ can be computed entirely from the
\emph{numerical type} of the minimal regular model $\mathscr{C}$.
In particular, $ \mathrm{ord}~\Delta_1 \neq 0 $ for certain
numerical types (see \cite[Prop.~7+8]{liu-genre2}). In view of
Proposition \ref{prop-formula-disc-cond}, this means that $ -
\mathrm{Art}_{\mathscr{C}/S} $ differs ``more'' from $ c(A)$ than
$ v(\Delta_{\min}) $.

\part{Applications to motivic zeta functions}\label{part-motzeta}
\chapter{Motivic zeta functions of semi-abelian varieties}\label{chap-motzeta}
In this chapter, we assume that $k$ is algebraically closed. We
will prove in Theorem \ref{theorem-zetamain} the rationality of
the motivic zeta function of a Jacobian variety, and we show that
it has a unique pole, which coincides with the tame base change
conductor from Chapter \ref{chap-conductor}. We will also
investigate the case of Prym varieties.

\section{The motivic zeta function}
\subsection{Definition}
\sss  We'll recall the definition of the motivic zeta function
$Z_G(T)$ of a semi-abelian $K$-variety $G$, which was introduced
in \cite[8.1]{HaNi}. It measures how the N\'eron model of $G$
varies under tamely ramified extensions of $K$.

\sss First we need to introduce some notation. Let $ K_0(\Var_k) $
be the Grothendieck ring of $k$-varieties. For each $k$-variety
$V$, we denote by $ [V] $ its class in $ K_0(\Var_k) $. We set $
\LL = [\mathbb{A}^1_k] $.
 We refer
to \cite{NS-K0} for the definition of the Grothendieck ring and
its basic properties.

\sss For every $d \in \N'$, we denote by $\mathscr{G}^{\qc}(d)$
the N\'eron model of $G \times_K K(d)$ as defined in Chapter
\ref{chap-preliminaries},\eqref{sss-neronqc}. Recall that, by
definition, $\mathscr{G}^{\qc}(d)$ is the largest quasi-compact
open subgroup scheme of the N\'eron $lft$-model $\cG(d)$ of
$G(d)$. We'll write $\cG$ and $\cG^{\qc}$ instead of $\cG(1)$ and
$\cG^{\qc}(1)$ for the N\'eron $lft$-model, resp.~N\'eron model,
of $G$. We denote by $h_d$ the canonical base change morphism
 $$h_d:\mathscr{G}^{\qc}\times_R R(d)\to \mathscr{G}^{\qc}(d).$$
\begin{definition}
The order function
$$\ord_G:\N'\to \N$$ is defined by
$$ \ord_{G}(d) = c(G,K(d)) \cdot d = \mathrm{length}_{R(d)}\coker(\Lie(h_d)). $$
\end{definition}
\sss The values $\ord_G(d)$ were defined in a different way in
\cite[7.2]{HaNi}, but we proved the formula $ \ord_{G}(d) =
c(G,K(d)) \cdot d $ in \cite[7.5]{HaNi}. For the purpose of this
paper, this formula can serve as a definition. Note that  we would
have obtained the same function $\ord_G$ by working with N\'eron
$lft$-models instead of N\'eron models. The reason for working
with N\'eron models is that in the following definition, we need
the property that $\mathscr{G}^{\qc}(d)_k$ is of finite type over
$k$, in order to take its class in the Grothendieck ring
$K_0(\Var_k)$ (note that the Grothendieck ring of $k$-schemes
locally of finite type is trivial by the existence of infinite
coproducts).

\begin{definition}\label{def-abzeta}
We define the motivic zeta function $Z_G(T)$ of $G$ as
$$Z_G(T)=\sum_{d\in\N'}[\mathscr{G}^{\qc}(d)_k] \LL^{\ord_{G}(d)} T^d \in K_0(\Var_k)[[T]].$$
\end{definition}

\subsection{Decomposing the identity component}
 \sss We see from the definition that the motivic zeta function
$Z_G(T)$ depends on two factors: the behaviour of the class
$[\mathscr{G}^{\qc}(d)_k]$, and the order function $\ord_G$. We
will now analyze the class $[\mathscr{G}^{\qc}(d)_k]$ in more
detail. For each $d$ in $\N'$, we consider the Chevalley
decomposition $$ 0 \to T(d) \times U(d) \to \mathscr{G}(d)_k^o \to
B(d) \to 0 $$ of $\mathscr{G}(d)_k^o$, with $T(d)$ a $k$-torus,
$U(d)$ a unipotent $k$-group and $B(d)$ an abelian $k$-variety.
The dimensions of $T(d)$ and $U(d)$ are called the toric and
unipotent rank of $G(d)$ and denoted by $t(G(d))$ and $u(G(d))$;
see Section
 \ref{ss-torrank} in Chapter \ref{chap-preliminaries}.

\begin{prop}\label{prop-classidcomp}
For each $d$ in $\N'$, we have
$$[\mathscr{G}^{\qc}(d)_k]=|\Comp{G(d)}_{\tors}|\cdot [\cG(d)^o_k] =|\Comp{G(d)}_{\tors}|\cdot (\LL-1)^{t(G(d))}\LL^{u(G(d))}[B(d)]$$
in $K_0(\Var_k)$.
\end{prop}
\begin{proof}
 Since $k$ is algebraically closed, each connected component of $\mathscr{G}^{\qc}(d)_k$ is isomorphic to the identity component
 $\mathscr{G}(d)^o_k$. The statement now follows from the computation of the class in $K_0(\Var_k)$ of a connected smooth commutative algebraic $k$-group in  \cite[3.1]{Nicaise}.
\end{proof}

\sss We've made an extensive study of the groups
$\Comp{G(d)}_{\tors}$ in Part \ref{part-comp}. If $G$ is tamely
ramified, then we have analyzed the behaviour of the order
function $\ord_G$ and the invariants $t(G(d))$, $u(G(d))$ and
$B(d)$ in \cite[\S6-7]{HaNi}. In the next section, we will extend
these results to (possibly wildly ramified) Jacobians, using our
results in Chapters \ref{chap-jacobians} and \ref{chap-conductor}.

\section{Motivic zeta functions of Jacobians}
\subsection{Behaviour of the identity component}
\sss \label{sss-notidcomp} Let $C$ be a smooth, projective,
geometrically connected $K$-curve of genus $ g
> 0 $ and index one.
 We
denote by $\mathscr{C}$ the minimal $sncd$-model of $C$ over $R$.
Recall that for every $d \in \N'$ we construct from $\mathscr C$
an $sncd$-model $\mathscr{C}(d)$ of $C(d)=C \times_K K(d)$ by the
procedure explained in Section \ref{sec-sncd} of Chapter
\ref{chap-jacobians}. We write its special fiber as
$$\mathscr{C}(d)_k = \sum_{i \in I(d)} N(d)_i E(d)_i.$$ We denote
by $e(C)$ the stabilization index of $C$, which was introduced in
Chapter \ref{chap-jacobians}, Definition \ref{def-stabind}.


\sss We set $A=\Jac(C)$. For each $d$ in $\N'$, we denote by
$\cA(d)$ the N\'eron model of $A(d)=A\times_K K(d)$ and by  $B(d)$
the abelian quotient in the Chevalley decomposition of
$\mathscr{A}(d)^o_k$. We'll write $\cA$ and $B$ instead of
$\cA(1)$ and $B(1)$.

\begin{prop}\label{prop-idclass}
Let $d$ be an element of $\N'$ that is prime to the stabilization
index $e(C)$ of $C$. Then $ B \cong B(d)$, $ t(A) = t(A(d)) $ and
$ u(A) = u(A(d)) $. Moreover, we have
$$ [\mathscr{A}(d)^o_k] = [\mathscr{A}^o_k] $$
in $K_0(\Var_k)$.
\end{prop}
\begin{proof}
For each $d \in \N'$, there is a canonical isomorphism
$$ \prod_{i \in I(d)} \Pic^0_{E(d)_i/k} \cong B(d) $$
by \cite[9.2.5 and 9.2.8]{neron}. If $d$ is prime to $e(C)$, one
concludes from Lemma \ref{lemma-fiberdescription} in Chapter
\ref{chap-jacobians} that the disjoint union of the principal
components of $\mathscr{C}_k$ is isomorphic to the disjoint union
of the principal components of $\mathscr{C}(d)_k$. This implies
that $ B \cong B(d) $, since any non-principal component is
rational and $ \Pic^0_{\mathbb{P}^1_k/k} $ is trivial.

The equality $ t(A) = t(A(d)) $ was proven already in Lemma
\ref{lemma-torrank} of Chapter \ref{chap-jacobians}. Since $$ g =
 t(A) + u(A)+\dim B = t(A(d)) + u(A(d))+\dim B(d),
$$ we find that $u(A)=u(A(d))$.
 The remainder of the statement follows from Proposition
 \ref{prop-classidcomp}.
\end{proof}

\subsection{Behaviour of the order function}
\sss We keep the notations and assumptions of
\eqref{sss-notidcomp}. We denote by $c_{\tame}(A)$ the tame base
change conductor of $A$, which was introduced in Chapter
\ref{chap-conductor}, Definition \ref{def-tamebc}.

\begin{prop}\label{prop-ordjav}  For every element $\alpha$ in   $$ \{1, \ldots, e(C)\}
\cap \mathbb{N}' $$ and every integer $q\geq 0$ such that $ \alpha
+ q e(C) \in \mathbb{N}' $, we have
$$ \ord_A(\alpha + q
e(C)) = \ord_A(\alpha) + q  e(C) c_{\tame}(A). $$
\end{prop}
\begin{proof}
Using Theorem \ref{Theorem-jumps} and Corollary \ref{cor-jumps} in
Chapter \ref{chap-jacobians}, we can again reduce to the case
$K=\C((t))$ as in the proof of Corollary \ref{cor-minindex} in
Chapter \ref{chap-jacobians}. Then $A$ is tamely ramified,
 and $e(C)$ equals the degree of the minimal extension of $K$
in $K^s$ where $A$ acquires semi-abelian reduction, by Proposition
\ref{prop-min} in Chapter \ref{chap-jacobians}. The result now
follows from \cite[7.6]{HaNi}.
\end{proof}

\section{Rationality and poles}
\subsection{Rationality of the zeta function}
\sss The following theorem is the main result of this chapter.
\begin{theorem} \label{theorem-zetamain} Let $G$ be either a tamely ramified semi-abelian
$K$-variety or the Jacobian of a smooth, projective, geometrically
connected $K$-curve of index one. We denote by $G_{\ab}$ the
abelian part of $G$; in particular, $G=G_{\ab}$ when $G$ is a
Jacobian.

 The motivic zeta function $Z_G(T)$ is rational, and belongs
to the subring
$$\mathscr{R}_k^{c_{\tame}(G)}=K_0(\Var_k)\left[T,\frac{1}{1-\LL^a T^b}\right]_{(a,b)\in
\Z \times \Z_{>0},\,a/b=c_{\tame}(G)}$$ of $K_0(\Var_k)[[T]]$. The
zeta function $Z_G(\LL^{-s})$ has a unique pole at
$s=c_{\tame}(G)$, whose order is equal to $\ttame{G_{\ab}}+1$. The
degree of $Z_G(T)$ is zero if $p=1$ and $G$ has potential good
reduction, and negative in all other cases.
\end{theorem}
\begin{proof}
 If $G$ is tamely ramified, then one can simply copy the proof of
 \cite[8.6]{HaNi}, invoking Theorem \ref{thm-main} in Chapter
 \ref{chap-uniform} instead of \cite[5.7]{HaNi-comp}.
 Thus we can assume that $G$ is the Jacobian of a smooth projective geometrically connected $K$-curve $C$
 of index one, and that $G$ is wildly ramified.
 Then $p>1$ and $e(C)$ is divisible by
$p$, by Proposition \ref{prop-min} in Chapter
\ref{chap-jacobians}. The proof of this case follows the same
lines as the proof of Theorem \ref{theorem-compser} in Chapter
\ref{chap-jacobians}.

We set $e=e(C)$ to ease notation. For every element $ \alpha $ in
$\{1, \ldots, e\} \cap \mathbb{N}'$ we define the auxiliary series
$$Z^{(\alpha)}_G(T)=\sum_{d \in (\alpha+\N e)}\left(|\Comp{G(d)}|[\mathscr{G}(d)_k^o] \LL^{\ord_G(d)}
   T^d\right)\quad \in K_0(\Var_k)[[T]].$$
Then we have
$$Z_G(T)=\sum_{\alpha \in \{1, \ldots, e\} \cap \mathbb{N}'} Z^{(\alpha)}_{G}(T).$$

 We fix an element $\alpha$ in $\{1, \ldots, e\} \cap \mathbb{N}'$
 and we set
 $\alpha'=\gcd(\alpha,e)$. We know that $e(C(\alpha'))=e/\alpha'$ by Proposition \ref{prop-e(C)} in Chapter \ref{chap-jacobians}.
  Then for every integer $d=\alpha+qe$ with $q\in \N$,
  we have $[\mathscr{A}(d)_k^o]=[\mathscr{A}(\alpha')_k^o]$ in
 $K_0(\Var_k)$ by Proposition \ref{prop-idclass},
$$|\Comp{G(d)}|=(d/\alpha')^{t(G(\alpha'))}|\Comp{G(\alpha')}|$$
 by Proposition \ref{prop-compfu} in Chapter \ref{chap-jacobians}, and
$$\ord_G(d)=\ord_G(\alpha) + q  e c_{\tame}(G)$$ by Proposition \ref{prop-ordjav}. We can
therefore write
$$Z^{(\alpha)}_G(T)=|\Comp{G(\alpha)}|\cdot [\mathscr{G}(\alpha)_k^o]\LL^{\ord_G(\alpha)}
S^{(\alpha)}_G(T)$$ with
$$S^{(\alpha)}_G(T)=\sum_{q\in \N}\left( \frac{qe+\alpha}{\alpha'}\right)^{t(G(\alpha'))}\LL^{q e c_{\tame}(G) } T^{qe+\alpha}.$$
 The rest of the proof is completely analogous to \cite[8.6]{HaNi}.
\end{proof}

\subsection{Poles and monodromy}
\sss If $G$ is a tamely ramified abelian $K$-variety, we can
relate the unique pole $s=c_{\tame}(G)$ of $Z_G(T)$ to the
monodromy action on the cohomology of $G$, as follows.  As we've
explained in \cite[\S2.5]{HaNi}, this result can be viewed as a
global version of the Motivic Monodromy Conjecture for
semi-abelian $K$-varieties.

\begin{theorem}
Let $G$ be a tamely ramified semi-abelian $K$-variety of dimension
$g$. Let $m$ be the order of $c(G)=c_{\tame}(G)$ in $\Q/\Z$, and
let $\Phi_m(T)$ be the cyclotomic polynomial whose roots are the
primitive $m$-th roots of unity. Then, for every topological
generator $\sigma$ of the tame inertia group $\Gal(K^t/K)$ and
every prime $\ell\neq p$, the polynomial $\Phi_m(T)$ divides the
characteristic polynomial of the action of $\sigma$ on
$H^g(G\times_K K^t,\Q_\ell)$. Thus for every embedding of
$\Q_\ell$ in $\C$, the value $\exp(2\pi c(G)i)$ is an eigenvalue
of $\sigma$ on $H^g(G\times_K K^t,\Q_\ell)$.
\end{theorem}
\begin{proof}
See Corollary 5.15 in \cite{HaNi}.
\end{proof}

\sss It is not at all clear how to find a similar cohomological
interpretation of $c_{\tame}(G)$ when $G$ is wildly ramified. The
tame cohomology spaces $H^i(G\times_K K^t,\Q_\ell)$ certainly
contain too little information; for instance, $H^1(G\times_K
K^t,\Q_\ell)$ is trivial for every wildly ramified elliptic
$K$-curve. On the other hand, the inertia group $I$ is not
 procyclic if $p>1$, so that we cannot associate eigenvalues to its
action on $H^g(G\times_K K^s,\Q_\ell)$ in any straightforward way.

\subsection{Prym varieties}
\sss The results that we've obtained for Jacobians can be extended
to a larger class of abelian $K$-varieties, the so-called Prym
varieties. In particular, we will explain how one can prove
 the rationality of Edixhoven's jumps and of the motivic zeta
 function for this class. A suitable reference for the theory of Prym varieties is \cite{AlBiHu}, where it
is not assumed that the ground field is algebraically closed.
 However, whenever we work with Prym varieties, we assume that $
\mathrm{char}(k) \neq 2 $.

\sss Let $C$ be a smooth, projective and geometrically connected
$K$-curve of genus $g>0$. Assume that $C$ carries an involution $$
\iota : C \to C $$  with either $0$ or $2$ fixed points. We'll
refer to the first situation as case (0) and to the latter as case
(2). We
 denote by
$$ \pi : C \to C_1 $$
 the quotient map. Put $ A = \mathrm{Jac}(C) $ and $ A_1 =
\mathrm{Jac}(C_1) $. There is induced a \emph{norm map}
$$ Nm : A \to A_1 $$
which can be described (on points) as $$ \sum m_Q [Q] \to \sum m_Q
[\pi(Q)]. $$The following result is well-known
(cf.~\cite[Ch.~1]{AlBiHu}).

\begin{prop}\label{prop-Prym}\item
\begin{enumerate}
\item The norm map is surjective with smooth kernel. The identity
component
$$ P := ker(Nm)^0 $$
is an abelian $K$-variety.

\item In case (0), the kernel $ ker(Nm) $ has $ 2 $ connected
components, and in case (2) it is connected.

\item Let $ \mathcal{L} = \mathcal{L}(\Theta) $ be the theta
divisor on $A$. The restriction $ \mathcal{M} = \mathcal{L}|_P $
is twice a principal polarization.
\end{enumerate}
\end{prop}

\begin{definition}\label{def-Prym}
The abelian $K$-variety $P$ is called the Prym variety associated
to $(C,\iota)$.
\end{definition}

\sss We denote by $u_K:P\to A$ the inclusion morphism. Let $
\phi_{\mathcal L} : A \to A^{\vee} $ be the principal polarization
associated to $ \mathcal L $. We denote by $ \overline{P} $ the
complement of $P$ in $A$; it is given by
$$ \overline{P} := (\ker(u_K^{\vee} \circ \phi_{\mathcal L})_{\mathrm{red}})^0. $$
We consider the difference map
$$ v_K : P \times_K \overline{P} \to A:(x,y)\mapsto x-y $$
 and the
morphism $ w_K : \overline{P} \to A_1  $ induced by the norm map.
\begin{lemma}\label{lemma-Prym}
 The morphisms $v_K$ and $w_K$  are
 isogenies whose degree is a power of $2$.
\end{lemma}
\begin{proof}
The  kernel of $v_K$ can be identified with $ P \times_{A}
\overline{P} $. By construction, it is a subgroup of the kernel of
$$u_K^{\vee}\circ \phi_{\mathcal{L}}\circ u_K:P\to P^{\vee}$$ but this
morphism is precisely the polarization $\phi_{\mathcal{M}}$
associated to $\mathcal{M}=u_K^*\mathcal{L}$.
 Since $\mathcal M $ is twice a principal
polarization by Proposition \ref{prop-Prym}, the degree of
 $\phi_{\mathcal M} $  is a power of $2$.
 Thus the degree of $v_K$ is also a power of $2$.

Now we turn our attention to $ w_K : \overline{P} \to A_1 $.  We
can factor $w_K$ as $ \alpha : \overline{P} \to A/P $ followed by
$ \beta : A/P \to A_1 $. We've already seen that the degree of
$$\ker(\alpha)= P \times_A \overline{P} $$ is a power of $2$.
 The morphism $\pi:C\to C_1$ induces a morphism of abelian
 $K$-varieties $\gamma:A_1\to A/P$ such that the composition
 $\beta\circ \gamma$ is multiplication by $\mathrm{deg}(\pi)=2$.
 Thus the degree of $\beta$ is a power of 2, as well, and the same conclusion holds for the composition $ w_K =
\alpha \circ \beta $.
\end{proof}

\begin{prop}\label{prop-Prymjumps}
For every $ d \in \N' $, we have
$$ m_{A,K(d)} = m_{P,K(d)} + m_{A_1,K(d)}. $$
In particular,
$$ m_{A} = m_{P} + m_{A_1}. $$
\end{prop}
\begin{proof}
For each $ d \in \N' $ we have isogenies
$$ v_{K(d)} : P(d) \times_{K(d)} \overline{P}(d) \to A(d) $$
and
$$ w_{K(d)} : \overline{P}(d) \to A_1(d) $$ whose degree is prime to $p$, by Lemma
\ref{lemma-Prym} and our assumption that $p\neq 2$.
 By \cite[7.3.6]{neron}, these morphisms extend uniquely to Galois
equivariant  isogenies
$$ v(d) : \mathscr{P}(d) \times_{R(d)} \overline{\mathscr{P}}(d) \to \mathscr{A}(d) $$
and
$$ w(d) : \overline{\mathscr{P}}(d) \to \mathscr{A}_1(d) $$
on the level of N\'eron models, and the degrees of these isogenies
 are still prime to $p$.

 Passing to the special fibers and considering the tangent spaces
 at the origin, we get isomorphisms of $k[\mu_d(k)]$-modules
$$ \mathrm{Lie}(v(d)_k) = \mathrm{Lie}( \mathscr{P}(d)_k ) \oplus \mathrm{Lie}(\overline{\mathscr{P}}(d)_k)  \to \mathrm{Lie}(\mathscr{A}(d)_k) $$
and
$$ \mathrm{Lie}(w(d)_k) : \mathrm{Lie}(\overline{\mathscr{P}}(d)_k) \to \mathrm{Lie}(\mathscr{A}_1(d)_k). $$
 By \eqref{sss-mudjump} in Chapter \ref{chap-conductor}, this implies that
 $m_{P,K(d)}+m_{\overline{P},K(d)}=m_A$ and
 $m_{\overline{P},K(d)}=m_{A_1,K(d)}$, so that
 $m_{A,K(d)}=m_{P,K(d)}+m_{A_1,K(d)}$.
\end{proof}
\begin{cor}
The jumps of $P$ are rational numbers.
\end{cor}
\begin{proof}
This is an immediate consequence of Proposition
\ref{prop-Prymjumps} and Corollary \ref{cor-minindex} in Chapter
\ref{chap-conductor}.
\end{proof}

\begin{theorem}
Assume that $ P $ has potential multiplicative reduction. Then the
motivic zeta function $Z_P(T)$ is rational, and belongs to the
subring
$$\mathscr{R}_k^{c_{\tame}(P)}=\mathcal{M}_k\left[T,\frac{1}{1-\LL^a T^b}\right]_{(a,b)\in
\Z \times \Z_{>0},\,a/b=c_{\tame}(P)}$$ of $\mathcal{M}_k[[T]]$.
The zeta function $Z_P(\LL^{-s})$ has a unique pole at
$s=c_{\tame}(P)$, whose order is equal to $\ttame{P}+1$. The
degree of $Z_P(T)$ is zero if $p=1$ and $P$ has potential good
reduction, and strictly negative in all other cases.
\end{theorem}
\begin{proof}
We can assume that $P$ is wildly ramified, since the tame case was
settled in \cite[8.6]{HaNi}. Let $L$ be the minimal extension of
$K$ in $K^s$ such that $P\times_K L$ has semi-abelian reduction,
and set $ e_P = [L:K]$. Then we define
$$ e(P) = \mathrm{lcm} \{e_P,e(C),e(C_1)\}. $$
For any $ d \in \N' $ we have that $ e(P(d)) =
e(P)/\mathrm{gcd}(d, e(P)) $, since the same property holds for
$e_P$, $e(C)$ and $e(C_1)$ individually.

We first observe that by Proposition \ref{prop-Prym}, the
equalities
$$ \mathrm{ord}_P(d) = c(P,K(d)) \cdot d = c(A,K(d)) \cdot d - c(A_1,K(d)) \cdot d = \mathrm{ord}_A(d) - \mathrm{ord}_{A_1}(d) $$
hold for all $ d \in \N' $. Hence, for all integers $ \alpha \in
\N' $ and $q$ such that $ \alpha + q \cdot e(P) \in \N' $, it
follows from Proposition \ref{prop-ordjav} that

\begin{eqnarray*}
\mathrm{ord}_P(\alpha + q \cdot e(P)) &=& \mathrm{ord}_A(\alpha) +
q \cdot e(P) \cdot c_{\tame}(A) \\ & & -
(\mathrm{ord}_{A_1}(\alpha) + q \cdot e(P) \cdot c_{\tame}(A_1))
\\ &=& (\mathrm{ord}_A(\alpha) - \mathrm{ord}_{A_1}(\alpha)) +  q
\cdot e(P) \cdot (c_{\tame}(A) - c_{\tame}(A_1)) \\ &=&
\mathrm{ord}_P(\alpha) + q \cdot e(P) \cdot c_{\tame}(P).
\end{eqnarray*}

Secondly, we claim that $ [\mathscr{P}_k^0] = [\mathscr{P}(d)_k^0]
$ for each $ d \in \N' $ that is prime to $ e(P) $. To see this,
note that $$ [\mathscr{P}(n)_k^0] = \LL^{u(P(n))} \cdot (\LL -
1)^{t(P(n))} $$ for all $ n \in \N' $ because we assume that $P$
has potential multiplicative reduction. Thus it suffices to show
that $u(P(d)) = u(P) $ and $ t(P(d)) = t(P) $ for all $d$ prime to
$e(P)$.

If $G$ denotes either $A$ or $A_1$, we know from Proposition
\ref{prop-idclass} that $ a(G(d)) = a(G) $, $ t(G(d)) = t(G) $ and
$ u(G(d)) = u(G) $ for each $d$ prime to $e(P)$. It follows from
 \cite[7.3.6]{neron} that the invariants $a$, $t$ and $u$ are preserved under isogenies of degree prime to $p$, and
it is obvious that they behave additively with respect to products
of abelian varieties. We therefore conclude that
$$ t(P(d)) = t(A(d)) - t(A_1(d)) = t(A) - t(A_1) = t(P) $$
and likewise that
$$ u(P(d)) = u(A(d)) - u(A_1(d)) = u(A) - u(A_1) = u(P). $$
This proves the claim.

Since we assume that $P$ has potential multiplicative reduction,
\cite[5.7]{HaNi-comp} asserts that
$$ \phi_P(d) = d^{t(P)} \cdot \phi_P $$
for every $ d \in \N' $ prime to $e(P)$. This gives us all the
necessary ingredients to repeat the proof of Theorem
\ref{theorem-zetamain}.
\end{proof}


\chapter{Cohomological interpretation of the motivic zeta function}\label{chap-cohomological}
In this chapter, we assume that $k$ is algebraically closed. We
will show how the motivic zeta function of a tamely ramified
semi-abelian $K$-variety admits a cohomological interpretation by
means of a trace formula, which is quite similar to the
Grothendieck-Lefschetz trace formula for varieties over finite
fields. The material in this chapter supersedes the unpublished
preprint \cite{Ni-traceab} of the second author on the trace
formula for abelian varieties. The case of algebraic tori was
discussed in \cite{Nicaise}.

\section{The trace formula for semi-abelian varieties}
\subsection{The rational volume}
\sss  We say that a $K$-variety $X$ is bounded if $X(K)$ is
bounded in $X$, in the sense of \cite[1.1.2]{neron}. This
boundedness property is equivalent to the existence of a
quasi-compact open subvariety of the rigid analytification
$X^{\an}$ of $X$ that contains all $K$-rational points of $X$
\cite[4.3]{Ni-tracevar}.

\sss Let $X$ be a smooth $K$-variety. A weak N\'eron model for $X$
is a smooth $R$-variety $\mathcal{X}$, endowed with an isomorphism
of $K$-varieties
$$\mathcal{X}\times_R K\to X,$$
such that the natural map
$$\mathcal{X}(R)\to X(K)$$ is a bijection \cite[3.5.1]{neron}. It follows from \cite[3.1.3 and 3.5.7]{neron} that $X$ has a weak N\'eron model if and only
if $X$ is bounded.

\sss A weak N\'eron model $\mathcal{X}$ is not unique, but, using
the change of variables formula for motivic integrals, one can
show that the Euler characteristic $\chi(\mathcal{X}_k)$ of the
special fiber of $\mathcal{X}$ does not depend on the choice of
the weak N\'eron model; see \cite[4.5.3]{motrigid},
\cite[5.2]{Ni-tracevar} and \cite[\S2.4]{nise-err}. Therefore, the
following definition only depends on $X$, and not on
$\mathcal{X}$.

\begin{definition} Let $X$ be a smooth and bounded $K$-variety, and let $\mathcal{X}$ be a weak N\'eron
model of $X$. Then the rational volume of $X$ is the integer
$$ s(X) := \chi(\mathcal{X}_k)\in \Z.$$
\end{definition}

\sss One can consider the rational volume $s(X)$ as a measure for
the set of $K$-rational points on $X$; loosely speaking, via the
reduction map
$$X(K)=\mathcal{X}(R)\to \mathcal{X}_k(k),$$ we can view $X(K)$ as
a family of balls in $K^d$ parameterized by the $k$-variety
$\mathcal{X}_k$. In particular, $s(X)$ vanishes if $X(K)$ is
empty, but the converse implication does not hold.

\sss As explained in \cite[\S1]{Nicaise}, it is possible to extend
the definition of the rational volume to varieties that are not
smooth and bounded. For a semi-abelian $K$-variety $G$, the
rational volume can be computed directly on the N\'eron model
$\cG^{\qc}$ of $G$: it is simply given by
$$s(G)=\chi(\cG_k^{\qc}).$$
 Note that $\cG_k^{\qc}$ is a weak N\'eron model of $G$ if and
 only if $G$ is bounded, in which case $\cG^{\qc}$ coincides with the $lft$-N\'eron model $\cG$ of
 $G$. We can further refine this expression for $s(G)$ in the
 following way. We say that $G$ has additive reduction if
 $\cG^o_k$ is unipotent.

 \begin{prop}\label{prop-ratvolsab}
Let $G$ be a semi-abelian $K$-variety.  Then $s(G)=0$ unless $G$
has additive reduction. In that case, $G$ is bounded and
$$s(G)=|\Comp{G}|.$$
 \end{prop}
 \begin{proof}
It follows from \cite[3.2]{Nicaise} that the toric rank of $G$ is
non-zero if $G$ contains a non-trivial split subtorus. Thus we
deduce from \eqref{sss-qcneron} in Chapter
\ref{chap-preliminaries} that $G$ is bounded when $G$ has additive
reduction. The remainder of the statement follows from Proposition
\ref{prop-classidcomp} in Chapter \ref{chap-motzeta} by applying
the Euler characteristic $\chi(\cdot)$, since the Euler
characteristics of a non-trivial $k$-torus and a non-trivial
abelian $k$-variety are zero.
 \end{proof}

\subsection{The trace formula and the number of N\'eron components}
\sss Let $X$ be a smooth, proper, geometrically connected
$K$-variety. Assume that $X(K^t)$ is non-empty, and that the wild
inertia $P$ of $K$ acts trivially on the $\ell$-adic cohomology
spaces
$$H^i(X\times_K K^s,\Q_\ell)$$ for all $i\geq 0$. Let $\sigma$ be
a topological generator of the tame inertia $\Gal(K^t/K)$.
 In
\cite[4.1.4]{ni-saito}, the second author asked whether the
following cohomological interpretation of the rational volume
$s(X)$ always holds:
\begin{equation}\label{eq-traceform}
s(X)=\sum_{i\geq 0}(-1)^i\mathrm{Trace}(\sigma\,|\,H^i(X\times_K
K^t,\Q_\ell)).\end{equation} He proved this result when $k$ has
characteristic zero \cite[6.5]{Ni-tracevar}, and when $X$ is a
curve \cite[\S7]{Ni-tracevar}. Moreover, in
\cite[4.2.1]{ni-saito}, he gave an explicit formula for the error
term in this formula in terms of an $sncd$-model of $X$. We will
now prove that the trace formula \eqref{eq-traceform} is valid if
we replace $X$ by a tamely ramified semi-abelian $K$-variety.

\sss Let $G$ be a semi-abelian $K$-variety. We denote by
$P_{G}(T)$ the characteristic polynomial
$$P_{G}(T)=\det(T\cdot \mathrm{Id}-\sigma\,|\,H^1(G\times_K K^t,\Q_\ell))$$
of $\sigma$ on $H^1(G\times_K K^t,\Q_\ell)$.

\begin{prop}\label{prop-charpolsab}\item
\begin{enumerate}
\item If we denote by $G_{\tor}$ and $G_{\ab}$ the toric and
abelian part of $G$, then
$$P_G(T)=P_{G_{\tor}}(T)\cdot P_{G_{\ab}}(T).$$

 \item The polynomial $P_{G}(T)$ is independent of the prime
$\ell\neq p$. It is a product of cyclotomic polynomials in
$\Z[T]$, and its roots are roots of unity of order prime to $p$.

\item We have $P_G(1)\neq 0$ if and only if $G$ has additive
reduction.

\item If $G$ is tamely ramified, then
$$P_{G}(1)=\sum_{i\geq 0}(-1)^i \mathrm{Trace}(\sigma\,|\,H^i(G\times_K
K^t,\Q_\ell)).$$
\end{enumerate}
\end{prop}
\begin{proof}
(1) The sequence of $\ell$-adic Galois representations $$0\to
H^1(G_{\ab}\times_K K^s,\Q_\ell)\to H^1(G\times_K K^s,\Q_\ell)\to
H^1(G_{\tor}\times_K K^s,\Q_\ell)\to 0$$ is exact, as it can be
obtained by dualizing the exact sequence of $\ell$-adic Tate
modules
$$0\to T_\ell G_{\tor}\to T_\ell G\to T_\ell G_{\ab}\to 0$$ and inverting
$\ell$ \cite[5.7]{HaNi}. Since the wild inertia $P$ is a
pro-$p$-group and $p$ is different from $\ell$, taking
$P$-invariants yields an exact sequence of $\ell$-adic
$\Gal(K^t/K)$-representations
$$0\to
H^1(G_{\ab}\times_K K^t,\Q_\ell)\to H^1(G\times_K K^t,\Q_\ell)\to
H^1(G_{\tor}\times_K K^t,\Q_\ell)\to 0.$$ The result now follows
from multiplicativity of the characteristic polynomial in short
exact sequences.

(2) By (1), we may assume that $G$ is a $K$-torus or an abelian
$K$-variety. If $G$ is a torus with character module $X(G)$, then
the statement follows from the canonical isomorphism of
$\ell$-adic $\Gal(K^t/K)$-representations
$$H^1(G\times_K K^t,\Q_\ell)\cong (X(G)\otimes_{\Z}\Q_\ell(1))^P$$
and the fact that the order of $\sigma$ on $X(G)^P$ is finite and
prime to $p$. If $G$ is an abelian $K$-variety, then it was proven
in \cite[2.10]{lorenzini} that $P_G(t)$ belongs to $\Z[T]$ and is
independent of $\ell$. Then $P_G(t)$ must be a product of
cyclotomic polynomials by quasi-unipotency of the
 Galois action on $T_\ell G$ \cite[IX.4.3]{sga7.1}, and the orders of its roots
 are prime to $p$ by triviality of the
 pro-$p$ part of $\Gal(K^t/K)$.

(3) The inequality
 $P_{G}(1)\neq 0$ is equivalent to the property that $(T_\ell
 G)^I=0$. Denoting by $\cG$ the N\'eron $lft$-model of $G$, we can
 identify $(T_\ell G)^I$ with $T_\ell \cG^o_k$, by
 \cite[IX.2.2.3.3 and IX.2.2.5]{sga7.1} (the results
 in \cite[IX]{sga7.1} are formulated for abelian $K$-varieties,
 but the proofs remain valid for semi-abelian varieties). Thus it
 follows from \cite[IX.2.1.11]{sga7.1} that the rank of the free $\Z_\ell$-module $(T_\ell
 G)^I$ equals the toric rank of $G$ plus twice the abelian rank of
 $G$. Therefore, $G$ has additive reduction if and only if
 $(T_\ell
 G)^I=0$.

(4) By \cite[5.7]{HaNi} and the fact that $P$ acts trivially on
the cohomology spaces $$H^i(G\times_K K^s,\Q_\ell),$$ there exists
for every $i\geq 0$ a Galois-equivariant isomorphism
$$H^i(G\times_K K^t,\Q_\ell)\cong\bigwedge
H^1(G\times_K K^t,\Q_\ell).$$ A straightforward computation now
shows that
$$P_{G}(1)=\sum_{i\geq 0}(-1)^i \mathrm{Trace}(\sigma\,|\,H^i(G\times_K
K^t,\Q_\ell)).$$
\end{proof}

\begin{cor}\label{cor-notadd}
If $G$ is a tamely ramified abelian $K$-variety that does not have
additive reduction, then
$$s(G)= \sum_{i\geq 0}(-1)^i \Trace(\sigma\,|\,H^i(G\times_K
K^t,\Q_\ell))=0.$$
\end{cor}
\begin{proof}
 The equality $s(G)=0$ follows from Proposition
 \ref{prop-ratvolsab}. By Proposition
 \ref{prop-charpolsab}, we have $$\sum_{i\geq 0}(-1)^i \Trace(\sigma\,|\,H^i(G\times_K
K^t,\Q_\ell))=P_{G}(1)=0$$ because
 $G$ does not have additive reduction.
\end{proof}

In order to investigate the case where $G$ has additive reduction,
we'll need some elementary lemmas.
\begin{lemma}\label{cyclo}
Fix an integer $d>1$ and let $\Phi_d(T)\in \Z[T]$ be the
cyclotomic polynomial whose roots are the primitive $d$-th roots
of unity. Then $\Phi_d(1)$ is a positive divisor of $d$.
\end{lemma}
\begin{proof}
We proceed by induction on $d$. For $d=2$ the result is clear, so
assume that it holds for each value $d'$ with $1<d'<d$. We can
write
$$\frac{T^d-1}{T-1}=\prod_{e|d,\,e>1}\Phi_e(T)$$ and evaluating at
$T=1$ we get
$$d=\prod_{e|d,\,e>1}\Phi_e(1)$$ so $\Phi_d(1)|d$.
By the induction hypothesis, $\Phi_e(1)>0$ for $1<e<d$, so
$\Phi_d(1)>0$ as well.
\end{proof}
\begin{lemma}\label{linalg}
Let $q$ be a prime, $M$ a free $\Z_q$-module of finite type, and
$\alpha$ an endomorphism of $M$. Then $M/\alpha M$ is torsion if
and only if $\alpha$ induces an automorphism on
$M\otimes_{\Z_q}\Q_q$. In this case, the order $|M/\alpha M|$ of
$M/\alpha M$ satisfies
$$|M/\alpha M|=|\det(\alpha\,|\,M\otimes_{\Z_q}\Q_q)|_q^{-1}$$
where $|\cdot|_q$ denotes the $q$-adic absolute value.
\end{lemma}
\begin{proof}
The module $M/\alpha M$ is torsion if and only if  $(M/\alpha
M)\otimes_{\Z_q}\Q_q=0$, i.e., if and only if $\alpha$ induces a
surjective and, hence, bijective endomorphism on
$M\otimes_{\Z_q}\Q_q$. In this case, we have
$$M/\alpha M\cong \Z_q/q^{c_1}\Z_q\oplus\ldots \oplus \Z_q/q^{c_r}\Z_q$$
where $q^{c_1},\ldots,q^{c_r}$ are the invariant factors of
$\alpha$ on $M$. Since $\det(\alpha\,|\,M\otimes_{\Z_q}\Q_q)$
equals $q^{c_1+\ldots+c_r}$ times a unit in $\Z_q$, we find
$$|M/\alpha M|=q^{c_1+\ldots+c_r}=|\det(\alpha\,|\,M\otimes_{\Z_q}\Q_q)|_q^{-1}.$$
\end{proof}
\begin{theorem}\label{theo-trace}
Let $G$ be a semi-abelian $K$-variety with
 additive reduction. If we denote by $\Comp{G}'$ the prime-to-$p$ part of
 the component group $\Comp{G}$,
  then
 $$P_G(1)=|\Comp{G}'|.$$
 If, moreover, $G$ is tamely ramified, then
$$\sum_{i \geq 0} (-1)^i \Trace(\sigma\,|\,H^i(G\times_K
K^t,\Q_\ell))=|\Comp{G}'|=|\Comp{G}|.$$
\end{theorem}
\begin{proof}
 We first show that the $p$-primary part of $\Comp{G}$ is trivial
 when $G$ is tamely ramified. By Proposition \ref{prop-bx}, the sequence
$$\Comp{G_{\tor}}\to \Comp{G}\to \Comp{G_{\ab}}\to 0$$ is exact. Let $L$ be the
minimal extension of $K$ in $K^s$ such that $G\times_K L$ has
semi-abelian reduction. It follows from \cite[4.1]{HaNi-comp} that
the torus $G_{\tor}\times_K L$ is split, and that the abelian
variety $G_{\ab}\times_K L$ has semi-abelian reduction. Thus we
can deduce from \cite[1.8]{liu-lorenzini} that  $\Comp{G_{\ab}}$
is killed by $[L:K]^2$, and from \cite[3.4]{Nicaise} that
$\Comp{G_{\tor}}_{\tors}$ is killed by $[L:K]$. (In fact,
$G_{\tor}$ is anisotropic because $G$ has additive reduction, so
that $\Comp{G_{\tor}}$ is torsion, but we don't need this fact
here.) Hence, $\Comp{G}$ is killed by $[L:K]^3$. In particular,
its $p$-primary part is trivial, since $L$ is a tame extension of
$K$.

Thus, in view of Proposition \ref{prop-charpolsab}(4), it suffices
to prove the first assertion in the statement. Let $q$ be a prime
different from $p$, and denote by
$$T^t_qG=(T_q G)^P$$ the tame $q$-adic Tate module of $G$. By
\cite[4.4]{HaNi-comp}, the order of the $q$-primary part
$\Comp{G}_{q}$ of $\Comp{G}$ equals the cardinality of
$$H^1(G(K^s/K),T_qG)_{\tors}\cong H^1(G(K^t/K),T^t_{q}G)\cong T^t_q G/(\mathrm{Id}-\sigma)T^t_q G$$
where the first isomorphism follows from the fact that the functor
$(\cdot)^P$ is exact on pro-$q$-groups, and the second from the
fact that $\sigma$ is a topological generator of $\Gal(K^t/K)$.

 Since $G$ has additive reduction, Proposition \ref{prop-charpolsab} tells us that $1$ is not an eigenvalue
 of $\sigma$ on $$T^t_qG\otimes_{\Z_q}\Q_q\cong H^1(G\times_K K^t,\Q_q)^{\vee}.$$
 Thus by Lemma \ref{linalg}, the $\Z_q$-module
$T^t_q G/(\mathrm{Id}-\sigma)T^t_q G$ is torsion, and
 its order is given by
$$|\det(\mathrm{Id}-\sigma\,|\,H^1(G\times_K K^t,\Q_q))|_q^{-1}$$ where $|\cdot|_q$ denotes the
$q$-adic absolute value. Since $P_G(T)$ is independent of
$\ell$, we find
\begin{equation}\label{eq-qprim}|\Comp{G}_{q}|=|P_G(1)|_{q}^{-1}\end{equation}
for every prime $q\neq p$.

 We know by Proposition \ref{prop-charpolsab} that
 each root of $P_G(T)$ is a root of unity of order prime
 to $p$. Thus Lemma \ref{cyclo} implies that
 $|P_G(1)|_p=1$ if $p>1$.
Hence, taking the product of \eqref{eq-qprim} over all primes
$q\neq p$, we get
$$|\Comp{G}'|=\prod_{q\neq p}|P_G(1)|_{q}^{-1}=\prod_{r\,\mathrm{prime}}|P_G(1)|_{r}^{-1}=|P_G(1)|=P_G(1)$$
where the last equality follows from Lemma \ref{cyclo}. This
concludes the proof.
\end{proof}
\begin{cor}[Trace formula for semi-abelian
varieties]\label{cor-traceform} If $G$ is a tamely ramified
semi-abelian $K$-variety, then
$$s(G)=\sum_{i\geq 0}(-1)^i \Trace (\sigma\,|\,H^i(G\times_K
K^t,\Q_\ell)).$$
\end{cor}
\begin{proof}
Combine Proposition \ref{prop-ratvolsab}, Corollary
\ref{cor-notadd} and Theorem \ref{theo-trace}.
\end{proof}
\begin{cor}
If $G$ is a semi-abelian $K$-variety with additive reduction, then
$|\Comp{G}'|$ is invariant under isogeny.
\end{cor}
\begin{proof}
This is an immediate consequence of Theorem \ref{theo-trace},
since an isogeny induces an isomorphism on $\Q_\ell$-adic
cohomology spaces. Note that the property of having additive
reduction is preserved under isogeny, by the same arguments as in
\cite[IX.2.2.7]{sga7.1}.
\end{proof}
\begin{cor}
 If
 $G$ is a tamely ramified semi-abelian $K$-variety with additive reduction, then
$$0\to \Comp{G_{\tor}}\to \Comp{G}\to \Comp{G_{\ab}}\to 0$$ is
exact.
\end{cor}
\begin{proof}
We already know from Proposition \ref{prop-bx} that this sequence
is right exact. Thus it is enough to show that $\Comp{G_{\tor}}$
is finite and
\begin{equation}\label{eq-compexact}|\Comp{G}|=|\Comp{G_{\tor}}|\cdot
|\Comp{G_{\ab}}|.\end{equation}

The torus $G_{\tor}$ and the abelian variety $G_{\ab}$ must have
additive reduction, since $P_G(1)\neq 0$ so that
$P_{G_{\tor}}(1)\neq 0$ and $P_{G_{\ab}}(1)\neq 0$ (see
Proposition \ref{prop-charpolsab}). Thus $G_{\tor}$ is anisotropic
and $\Comp{G_{\tor}}$ is finite. Proposition \ref{prop-charpolsab}
also tells us that
$$P_G(1)=P_{G_{\tor}}(1)\cdot P_{G_{\ab}}(1)$$ so that the
equality \eqref{eq-compexact} follows from Theorem
\ref{theo-trace}.
\end{proof}

\begin{remark}
In the proof of Theorem \ref{theo-trace}, we invoked
\cite[1.8]{liu-lorenzini}. The proof of \cite[1.8]{liu-lorenzini}
is based on \cite[5.6 and 5.9]{B-X}; we refer to
\eqref{sss-BXcorrect} in Chapter \ref{chap-uniform} for a comment
on these results.
\end{remark}

\subsection{Cohomological interpretation of the motivic zeta
function} \sss Let $G$ be a semi-abelian $K$-variety. The trace
formula in Corollary \ref{cor-traceform} yields a cohomological
interpretation of the motivic zeta function $Z_G(T)$ from Chapter
\ref{chap-motzeta}, in the following way. We denote by
$\chi(Z_G(T))$ the element of $\Z[[T]]$ that we obtain by taking
the images of the coefficients of the series $Z_G(T)\in
K_0(\Var_k)[[T]]$ under the ring morphism
$$\chi:K_0(\Var_k)\to \Z$$ that sends the class of a $k$-variety $X$ to the Euler
characteristic $\chi(X)$. Explicitly, we have
$$\chi(Z_G(T))=\sum_{d>0}\chi(\cG(d)^{\qc}_k)T^d\quad \in \Z[[T]]$$
where $\cG(d)^{\qc}$ denotes the N\'eron model of $G(d)=G\times_K
K(d)$.

\begin{theorem}[Cohomological interpretation of the motivic zeta
function]\label{thm-cohint} Let $G$ be a semi-abelian $K$-variety,
and denote by $\mathrm{Add}_G$ the set of elements $d$ in $\N'$
such that $G(d)$ has additive reduction. Then
\begin{eqnarray*}
\chi(Z_G(T))&=& \sum_{d\in \N'}s(G\times_K K(d))T^d
\\ &=&\sum_{d\in \mathrm{Add}_G}|\Comp{G(d)}|T^d
\\ &=& \sum_{d\in \N'}\sum_{i\geq 0}(-1)^i
\Trace(\sigma^d\,|\,H^i(G\times_K K^t,\Q_\ell))T^d
\end{eqnarray*}
in $\Z[[T]]$.
\end{theorem}
\begin{proof}
The first equality follows from the definition of the rational
volume, and the second from Proposition \ref{prop-ratvolsab}
(these two equalities do not require $G$ to be tamely ramified).
The last equality is a consequence of the trace formula in
Corollary \ref{cor-traceform}.
\end{proof}

\sss Theorem \ref{thm-cohint} was stated for abelian varieties in
\cite[8.4]{HaNi}, with a reference to the second author's
unpublished preprint \cite{Ni-traceab} for the proof. Theorem
\ref{thm-cohint} supersedes that statement, and extends it to
semi-abelian varieties.

\section{The trace formula for Jacobians}
To conclude this chapter, we give an alternative proof of Theorem
\ref{theo-trace} if $G$ is the Jacobian $\Jac(C)$ of a $K$-curve
$C$. The proof is based on an explicit expression for the
characteristic polynomial $P_{G}(T)$ in terms of an $sncd$-model
of the curve $C$.
\subsection{The monodromy zeta function}\label{ss-monzeta}
\sss Let $C$ be a smooth, projective, geometrically connected
$K$-curve of genus $g(C)>0$. Let $\mathscr{C}$ be
 a minimal $sncd$-model of $C$, with special fiber
$$\mathscr{C}_k=\sum_{i\in I}N_iE_i.$$ We set
$\delta(C)=\gcd\{N_i\,|\,i\in I\}$. For each
$i\in I$, we denote by $N_i'$ the prime-to-$p$ part of $N_i$.
Moreover, we set $E_i^o=E_i\setminus \cup_{j\neq i}E_j$ and we
denote by $d_i$ the cardinality of $E_i\setminus E_i^o$.

\sss We denote by $A$ the Jacobian of $C$. Then there exists an
isomorphism of $G(K^t/K)$-representations
$$H^1(C\times_K K^t,\Q_\ell)\cong H^1(A\times_K K^t,\Q_\ell).$$
 Thus $C$ is cohomologically tame if and only if $A$ is tamely
 ramified, and the polynomial $P_A(T)$ is equal to the
 characteristic polynomial of $\sigma$ on $$H^1(C\times_K
 K^t,\Q_\ell).$$
 We denote by $\zeta_C(T)$ the reciprocal of the monodromy
zeta function of $C$, i.e.,
$$\zeta_C(T)=\prod_{i=0}^{2}\mathrm{det}(T\cdot \mathrm{Id}- \sigma\,|\,H^i(X\times_K
K^t,\Q_\ell))^{(-1)^{i+1}}\ \in \Q_\ell(T).$$

\begin{theorem}\label{acampo}
 We have
\begin{eqnarray}\label{eq-zeta1.1}
\zeta_C(T)&=&\prod_{i\in I}(T^{N'_i}-1)^{-\chi(E_i^o)},
\\ P_{A}(T)&=&(T-1)^2\prod_{i\in
I}(T^{N'_i}-1)^{-\chi(E_i^o)}. \label{eq-zeta1.2}
\end{eqnarray}
 If $C$ is cohomologically tame, and either $\delta(C)$ is prime to $p$, or $g(C)\neq 1$, then
\begin{eqnarray}\label{eq-zeta2.1}
\zeta_C(T)&=&\prod_{i\in I}(T^{N_i}-1)^{-\chi(E_i^o)},
\\ P_{A}(T)&=&(T-1)^2\prod_{i\in
I}(T^{N_i}-1)^{-\chi(E_i^o)}. \label{eq-zeta2.2}
\end{eqnarray}
\end{theorem}
\begin{proof}
The expressions for $P_{A}(T)$ follow immediately from the
expressions for $\zeta_C(T)$, since $\sigma$ acts trivially on the
degree $0$ and degree $2$ cohomology of $C$. Formula
\eqref{eq-zeta1.1} is a special case of the arithmetic A'Campo
formula in \cite[2.6.2]{ni-saito}.

Now assume that $C$ is cohomologically tame and that $g(C)\neq 1$
or that $\delta(C)$ is prime to $p$. Note that the property that
$\delta(C)$ is prime to $p$ implies that $C(K^t)$ is non-empty, by
\cite[3.1.4]{ni-saito}.
 Then Saito's geometric criterion for cohomological tameness
 \cite[3.11]{Saito}, phrased in the form of
 \cite[3.3.2]{ni-saito},
implies that $E_i^o\cong \mathbb{G}_{m,k}$ if  $N_i$ is not prime
to $p$. Thus $\chi(E_i^o)=0$ for all $i\in I$ such that $N_i\neq
N'_i$. Therefore,
 (\ref{eq-zeta2.1}) follows from (\ref{eq-zeta1.1}).
\end{proof}

\sss As an immediate corollary, we obtain an alternative proof of
Theorem 2.1(i) in \cite{lorenzini}. Denote by $g(E_i)$ the genus
of $E_i$, for each $i\in I$. We put $a=\sum_{i\in I}g(E_i)$. We
denote by $\Gamma$ the dual graph of $\mathscr{C}_k$, and by $t$
its first Betti number.
\begin{cor}[Lorenzini]\label{cor-lorenzini}  We have
$$P_{A}(T)=(T-1)^{2a+2t}\prod_{i\in
I}\left(\frac{T^{N'_i}-1}{T-1}\right)^{2g(E_i)+d_i-2}.$$
\end{cor}
\begin{proof}
This follows immediately from Theorem \ref{acampo}, the formula
$$\chi(E_i^o)=2-2g(E_i)-d_i,$$ and the fact that
$$2-2t=\sum_{i\in I}(2-d_i)$$ (both sides equal twice the Euler
characteristic of $\Gamma$).
\end{proof}
\sss In \cite{lorenzini}, it was assumed that $\delta(C)=1$, but
our arguments show that this is not necessary. If $\delta(C)=1$,
then $a$ equals the abelian rank of $A=\Jac(C)$ and $t$ its toric
rank (see \cite[p.\,148]{Lorenzini-jac}).

\subsection{The trace formula for Jacobians}
We can use the above results to give an alternative proof  of
 Theorem \ref{theo-trace} and Corollary \ref{cor-traceform} for the
 Jacobian $A=\Jac(C)$ of the curve $C$.
\begin{prop} We keep the notations of Section \ref{ss-monzeta},
and we assume that $\delta(C)=1$.
\begin{enumerate}
\item If $A$ does not have additive reduction, then $P_{A}(1)=0$.

\item If $A$ has additive reduction, then $P_{A}(1)=|\Comp{A}'|$,
where $\Comp{A}'$ denotes the prime-to-$p$ part of $\Comp{A}$.

\item If $A$ has additive reduction and is tamely ramified, then
$\Comp{A}'=\Comp{A}$ and
 $$s(A)=\sum_{i\geq 0}(-1)^i \Trace (\sigma\,|\,H^i(A\times_K
K^t,\Q_\ell)).$$
\end{enumerate}
\end{prop}
\begin{proof}
(1) It follows from Corollary \ref{cor-lorenzini} that the order
of $1$ as a root of $P_{A}(T)$ equals $2a+2t$. Hence, if $A$ does
not have additive reduction, then $P_{A}(1)=0$.

 (2)  By
\cite[1.5]{Lorenzini-jac}, we have $$|\Comp{A}|=\prod_{i\in
I}N_i^{d_i-2}$$ because the toric rank of $A$ is zero. By
Corollary \ref{cor-lorenzini} we know that
$$P_{A}(T)=\prod_{i\in
I}\left(\frac{T^{N'_i}-1}{T-1}\right)^{d_i-2}$$ because $a=t=0$.
This yields
$$P_{A}(1)=\prod_{i\in I}(N'_i)^{d_i-2}= \vert \Comp{A}' \vert.$$

(3) If $A$ is tamely ramified, then Saito's criterion for
cohomological tameness \cite[3.11]{Saito} implies that $d_i=2$ if
$N_i\neq N_i'$, so that
$$|\Comp{A}'|=\prod_{i\in I}(N'_i)^{d_i-2}=\prod_{i\in
I}N_i^{d_i-2}=|\Comp{A}|.$$ Combining this with (1), (2) and
Propositions \ref{prop-ratvolsab} and \ref{prop-charpolsab}(4), we
find $$s(A)=\sum_{i\geq 0}(-1)^i \Trace (\sigma\,|\,H^i(A\times_K
K^t,\Q_\ell)).$$
\end{proof}

\part{Some open problems}\label{sec-ques}
To conclude, we will formulate some open problems and directions
for future research stemming from the results in the preceding
chapters. We assume that $k$ is algebraically closed.

\setcounter{section}{0}
\section{The stabilization index}\label{sec-ques-e}
  Let $A$ be an abelian $K$-variety, and let $ L/K $ be the minimal
extension of $K$ in $K^s$ such that $A\times_K L$ has semi-abelian
reduction. Let $K'$ be a finite tame extension of $K$ and denote
by $R'$ the integral closure of $R$ in $K'$. We set $A'=A\times_K
K'$. We denote by $\mathscr{A}'$ the N\'eron model of $A'$, and by
$$h:\mathscr{A}\times_R R'\to \mathscr{A}'$$ the canonical
 base change morphism.

 A central theme in this monograph was the study of the properties
 of this base change morphism $h$. As we've explained in Section
 \ref{subsec-guide} of the introduction, the basic idea is that
 the N\'eron models
 $\mathscr{A}$ and $\mathscr{A}'$ should differ
 as little as possible if the  extension $K'/K$
 is sufficiently orthogonal to the extension $L/K$. This
principle was a crucial ingredient in establishing rationality and
 determining the poles of the component series $ S^{\Phi}_A(T) $
and the motivic zeta function $ Z_A(T)$. The qualification ``as
little as possible'' includes, in particular, the following
properties.
  \begin{itemize}
  \item The number of components grows as if $\mathscr{A}$ had
  semi-abelian reduction, i.e., the equality in Proposition
  \ref{prop-compfu} of the introduction holds:
  $$|\Comp{A(d)}|=d^{t(A)}\cdot |\Comp{A}|.$$
  \item The $k$-varieties $\mathscr{A}^o_k$ and $(\mathscr{A}')^o_k$ define the same class in
  $K_0(\Var_k)$.
  \end{itemize}

 As we've seen, the meaning of ``sufficiently orthogonal'' is less
 clear. We'd like to express this property by saying that the
 degree of the extension $K'/K$ is coprime to a certain invariant
 $e(A)$ of the abelian $K$-variety $A$, that we will call the
 {\em stabilization index}. We can define the stabilization index
 in the following cases.

 \begin{enumerate}
\item If $A$ is tamely ramified or $A$ has potential
multiplicative reduction, then we can set $e(A)=[L:K]$. As
 we've explained in Section
 \ref{subsec-guide} of the introduction, the two
 above properties are satisfied for every finite tame extension
 $K'/K$ of degree prime to $e(A)$.

\item If $A$ is the Jacobian of a smooth proper $K$-curve $C$ of
index one, then we can set $e(A)=e(C)$, where $e(C)$ is the
stabilization index of $C$ that we introduced in Chapter
\ref{chap-jacobians}, Definition \ref{def-stabind}. We've shown in
Corollary  \ref{cor-minindex} of Chapter \ref{chap-conductor} that
$e(A)$ only depends on $A$, and not on $C$. The two above
properties are satisfied for every finite tame extension $K'/K$ of
degree prime to $e(A)$, by Chapter \ref{chap-jacobians},
Proposition \ref{prop-compfu} and Chapter \ref{chap-motzeta},
Proposition \ref{prop-idclass}.
 \end{enumerate}

This immediately raises several questions. If $A$ is a tamely
ramified Jacobian, then the two definitions of $e(A)$ are
equivalent, by Proposition \ref{prop-min} in Chapter
\ref{chap-jacobians}. We expect that the two definitions are also
equivalent for wildly ramified Jacobians with potential
multiplicative reduction; this is true if and only if the
following question has a positive answer. We say that $C$ has
multiplicative reduction if $\mathrm{Jac}(C)$ has multiplicative
reduction; this is equivalent to the property that the special
fiber of the minimal $sncd$-model of $C$ consists of rational
curves.

\begin{openquestion}\label{ques-compat}
Let $C$ be a smooth proper $K$-curve of index one. Let $L$ be the
minimal extension of $K$ in $K^s$ such that $C\times_K L$ has
semistable reduction, and assume that $C\times_K L$ has
multiplicative reduction. Is it true that $[L:K]=e(C)$?
\end{openquestion}
 If $L/K$ is tame this follows from Proposition \ref{prop-min} in
 Chapter \ref{chap-jacobians}, but if $L/K$ is wild, it is a very
 difficult problem to determine the degree of $L/K$ in terms of
 the geometry of an $sncd$-model of $C$. The equality $e(C)=[L:K]$
 can fail if we don't assume that $C\times_K L$ has
 multiplicative reduction: see Examples \ref{ex-diff} and \ref{ex-diff2} in Chapter
 \ref{chap-jacobians}.

  A second question is whether we can find a suitable definition
  of the stabilization index  $e(A)$ for arbitrary abelian
  $K$-varieties $A$.
 A possible candidate is the following. Assume that all the jumps
of $A$ are rational numbers. This is currently still an open
problem, except in the tame case and for Jacobians. Then we define
the stabilization index $e(A)$ of $A$ as the smallest integer
$e>0$ such that $e\cdot j$ belongs to $\Z$ for every jump $j$ of
$A$.
 Corollary \ref{cor-minindex} in Chapter \ref{chap-conductor} guarantees that $ e(A) = e(C) $ if
 $A $ is the Jacobian of a $K$-curve $C$. Moreover, it follows
 easily from \cite[4.20 and 5.1]{HaNi} that $e(A)$ is equal to $[L:K]$ when $A$ is tamely
 ramified. Thus our new definition of the stabilization index $e(A)$
 is equivalent to Definition (1) above for tamely
 ramified abelian varieties and to Definition (2) for Jacobians. We believe that it is
 also equivalent to Definition (1) if $A$ is a wildly
 ramified abelian $K$-variety with potential multiplicative
 reduction. In this case, the uniformization space of
 $A$ is an algebraic $K$-torus. Thus, in view of Proposition \ref{prop-bcuniform}, it would be enough to show that the following question has a positive answer.

 \begin{openquestion}\label{ques-tori}
Let $T$ be an algebraic $K$-torus, and let $L$ be the minimal
splitting field of $K$ in $K^s$. Is it true that the jumps of $T$
are rational, and that $[L:K]$ is the smallest integer $e>0$ such
that $e\cdot j$ belongs to $\Z$ for every jump $j$ of $T$?
 \end{openquestion}

Note that a positive answer to Question \ref{ques-tori} implies
that Definitions (1) and (2) are equivalent for Jacobians with
potential multiplicative reduction, so that it would also provide
an affirmative answer to Question \ref{ques-compat}. This seems to
be a promising approach to solve Question \ref{ques-compat}. A key
step would be to find a suitable interpretation for the jumps of a
torus in terms of its character module.

Finally, for arbitrary abelian $K$-varieties $A$, the question
remains whether the jumps of $A$ are indeed rational numbers and
whether our candidate for $e(A)$ has the required properties.  All
 of these problems will be investigated in future research.

\section{The characteristic
polynomial}\label{sec-ques-charpol}
 We defined the characteristic polynomial $P_C(t)$ of a smooth, projective,
 geometrically connected $K$-curve $C$ in Chapter \ref{chap-jacobians}, Definition
 \ref{def-charpol}. When $C$ is tamely ramified, the
 characteristic polynomial has a natural cohomological
 interpretation: by Proposition \ref{prop-charpol} in Chapter \ref{chap-jacobians}, it is the characteristic polynomial $P'_C(t)$ of the action
 of any topological generator of $\Gal(K^t/K)$ on $H^1(C\times_K
 K^t,\Q_\ell)$. When $C$ is wildly ramified, then $P'_C(t)$ still
 divides $P_C(t)$, but we do not know how to interpret the
 remaining factor in terms of the cohomology space $H^1(C\times_K
 K^s,\Q_\ell)$. In view of Proposition \ref{prop-charpol-e} in Chapter \ref{chap-jacobians}, such
 an interpretation would also yield a cohomological expression for
 the stabilization index $e(C)$, and this might lead to a
 definition of the stabilization index for all wildly ramified abelian
 $K$-varieties and a proof of the rationality of the jumps (see Section \ref{sec-ques-e}).

\section{The motivic zeta function and the monodromy
conjecture} Let $A$ be a tamely ramified  abelian $K$-variety of
dimension $g$. We proved in \cite[8.6]{HaNi} that the base change
conductor $c(A)$ is the only pole of the motivic zeta function
$Z_A(\LL^{-s})$. Moreover, it follows from our result in
\cite[5.13]{HaNi} that for every embedding of $\Q_\ell$ in $\C$
and for every topological generator $\sigma$ of $\Gal(K^t/K)$, the
value $\exp(2\pi i c(A))$ is an eigenvalue of the action of
$\sigma$ on the tame $\ell$-adic cohomology group $$ H^{g}(A
\times_K K^t, \mathbb{Q}_{\ell}). $$ As we've explained in
\cite[\S2.5]{HaNi}, this result can be viewed as a version of
Denef and Loeser's motivic monodromy conjecture for abelian
varieties.

 We've seen in Chapter \ref{chap-motzeta}, Theorem \ref{theorem-zetamain}, that the tame base
 change conductor
 $c_{\tame}(A)$ is still the unique pole of the motivic zeta function
$Z_A(\LL^{-s})$ when $A=\Jac(C)$ is a (possibly wildly ramified)
Jacobian. But it is not clear at all how a suitable version of the
monodromy conjecture could be formulated in this case, because the
tame cohomology spaces do not contain enough information.  For
instance, for every wildly ramified elliptic $K$-curve $E$, the
tame cohomology space
$$ H^1(E \times_K K^t, \mathbb{Q}_{\ell}) $$ is trivial.

This question is closely related to the problem of finding a
cohomological interpretation of the characteristic polynomial
$P_C(t)$ (Section \ref{sec-ques-charpol}). We know by Proposition
\ref{cor-roots} in Chapter \ref{chap-conductor} that every jump
$j$ of $A$ is a root of $P_C(t)$ of multiplicity at least
$m_A(j)$. If $A$ is a tamely ramified Jacobian, then by
Proposition \ref{prop-charpol} in Chapter \ref{chap-jacobians},
this implies that $\exp(2\pi i j)$ is an eigenvalue of order at
least $m_A(j)$ of the action of $\sigma$ on
$$H^1(A\times_K K^t,\Q_\ell)$$ so that $$\exp(2\pi i c_{\tame}(A))=\prod_{j\in
\mathcal{J}_A} \exp(2\pi i m_A(j)j)$$ is an eigenvalue of the
action of $\sigma$ on
 $$H^g(A\times_K K^t,\Q_\ell)\cong \bigwedge^g H^1(A\times_K K^t,\Q_\ell).$$
 Thus it seems plausible that a suitable cohomological
 interpretation of $P_C(t)$ in the wildly ramified case would also
 give rise to a cohomological interpretation of $c_{\tame}(A)$.

 We expect that, for arbitrary wildly ramified
 abelian $K$-varieties $A$, the motivic zeta function $Z_A(\LL^{-s})$ is
 still rational, with a unique pole at $c_{\tame}(A)$. The key
 step in the proof of such a result would be a suitable characterization of
 the stabilization index (Section \ref{sec-ques-e}).
  One may
 also ask for a cohomological interpretation of $c_{\tame}(A)$ in
 this case.

\section{Base change conductor for Jacobians}
Let $A$ be the Jacobian of a $K$-curve $C$ of index one. In
Chapter \ref{chap-artin}  we compared the base change conductor
$c(A)$ to the Artin conductor $\mathrm{Art}(C)$ when $C$ had genus
$1$ or $2$. Based on these results, one is led to speculate
whether it is still possible to compute $c(A)$ in terms of other
arithmetic invariants of the curve $C$ if the genus of $C$is
bigger than $2$. More precisely, Proposition
\ref{prop-c-ell-artin} and Proposition
\ref{prop-formula-disc-cond} both state that a certain multiple of
$c(A)$ equals the valuation of a ``minimal'' discriminant, up to a
correction term that can be computed from the special fiber of a
semi-stable minimal regular model of $C$. In future work, we plan
to investigate in how far these results generalize to curves of
higher genus.

\section{Component groups of Jacobians}
A key element in the proof of our results on the component series
and the motivic zeta function of a Jacobian $ A = \mathrm{Jac}(C)$
is Proposition \ref{prop-compfu} in Chapter \ref{chap-jacobians},
which asserts that
$$ \vert \Phi(A(d)) \vert = d^t \cdot \vert \Phi(A) \vert $$
 for every $ d \in \NN' $ prime to $e(C)$. We deduced this
formula in a somewhat indirect fashion, using Winters' theorem to
transfer it from the equal characteristic zero case. It would be
quite interesting to have a more direct proof of Proposition
\ref{prop-compfu} that avoids using Winters' theorem.

In the literature one can find numerous results concerning the
computation of the order of $\Phi(A)$ in terms of a regular
$R$-model $\mathscr C $ of $C$ (see \cite[9.6]{neron} for an
overview). All of these results essentially involve computing
certain minors of the intersection matrix associated to the
special fiber $\mathscr{C}_k$. Even though, starting with the
minimal $sncd$-model $ \mathscr C $ of $C$, our results in Chapter
\ref{chap-jacobians} provide us with good control over the special
fiber of $\mathscr{C}(d) $ for every $ d \in \NN' $ prime to
$e(C)$, one is faced with the problem that the rank of the
intersection matrix of $\mathscr{C}(d)_k $ grows together with
$d$. If $t(A)=0$, an easy application of \cite[9.6/6]{neron}
enables one to derive the formula mentioned above, but whenever
$t(A)>0$ it remains a combinatorial challenge to use this approach
for computing the order of the component group of $A(d)$ as a
function of $d$.

\backmatter


\begin{thebibliography}{10}
\bibitem[SGA1]{sga1}
{\em Rev\^etements \'etales et groupe fondamental.} \newblock
\newblock S\'eminaire de G\'eom\'etrie Alg\'ebrique du Bois Marie
1960/61 (SGA 1).
 Dirig\'e par A.~Grothendieck. Volume~224 of {\em Lecture Notes in Mathematics.}
Springer, Berlin, 1971.


\bibitem[SGA3-I]{sga3.1}
{\em Sch\'emas en groupes. {I}: {P}ropri\'et\'es g\'en\'erales des
sch\'emas en
  groupes}.
\newblock S\'eminaire de G\'eom\'etrie Alg\'ebrique du Bois Marie 1962/64 (SGA
  3). Dirig\'e par M.~Demazure et A.~Grothendieck. Volume~151 of {\em Lecture Notes in
  Mathematics.} Springer, Berlin, 1970.

  \bibitem[SGA3-II]{sga3.2}
{\em Sch\'emas en groupes. {II}: Groupes de type multiplicatif, et
structure
  des sch\'emas en groupes g\'en\'eraux}.
\newblock S\'eminaire de G\'eom\'etrie Alg\'ebrique du Bois Marie 1962/64 (SGA
  3). Dirig\'e par M. Demazure et A. Grothendieck. Volume~152 of {\em Lecture Notes in
  Mathematics.} Springer, Berlin, 1970.



\bibitem[SGA7-I]{sga7.1}
{\em Groupes de monodromie en g\'eom\'etrie alg\'ebrique. {I}}.
\newblock S\'eminaire de G\'eom\'etrie Alg\'ebrique du Bois-Marie 1967--1969
  (SGA 7 {I}). Dirig\'e par A.~Grothendieck. Avec la collaboration de
  M.~Raynaud et D.S.~Rim. Volume~288 of {\em Lecture Notes in Mathematics.} Springer-Verlag, Berlin, 1972.


\bibitem[ABH02]{AlBiHu}
V.~Alexeev, Ch.~Birkenhake, K.~Hulek.
\newblock {Degenerations of {P}rym varieties}.
\newblock {\em J. Reine Angew. Math.}, 553:73-116, 2002.

\bibitem[BL93]{formrigII}
 S.~Bosch and W.~L\"utkebohmert.
 \newblock Formal and rigid geometry II. Flattening techniques.
 \newblock {\em Math. Ann.} 296(3):403-–429, 1993.

\bibitem[BLR90]{neron}
S.~Bosch, W.~{L\"u}tkebohmert, and M.~Raynaud.
\newblock {\em {N\'eron models.}} Volume~21 of
 {\em Ergebnisse der Mathematik und ihrer Grenzgebiete}. \newblock Springer-Verlag, 1990.

\bibitem[BS95]{bosch-neron}
S.~Bosch and K.~Schl{\"o}ter.
\newblock N\'eron models in the setting of formal and rigid geometry.
\newblock {\em Math. Ann.}, 301(2):339--362, 1995.

\bibitem[BX96]{B-X}
S.~Bosch and X.~Xarles.
\newblock {Component groups of N{\'e}ron models via rigid uniformization}.
\newblock {\em Math. Ann.}, 306:459--486, 1996.

\bibitem[Ch00]{chai}
C.-L. Chai.
\newblock {N\'eron models for semiabelian varieties: congruence and change of
  base field}.
\newblock {\em Asian J. Math.}, 4(4):715--736, 2000.

\bibitem[CY01]{chai-yu}
C.-L. Chai and J.-K. Yu.
\newblock {Congruences of N\'eron models for tori and the Artin conductor
(with an appendix by E.~de Shalit).}
\newblock {\em Ann. Math. (2)}, 154:347--382, 2001.

\bibitem[Co02]{conrad-chevalley}
B.~Conrad.
\newblock {A modern proof of Chevalley's theorem on algebraic groups.}
\newblock {\em J. Ramanujan Math. Soc.}, 17(1):1--18, 2002.

\bibitem[CES03]{CES}
B.~Conrad, B.~Edixhoven and W.~Stein. \newblock $J_1(p)$ has
connected fibers. \newblock {\em Doc. Math.} 8:331-–408, 2003.

\bibitem[CLN11]{CLN}
R.~Cluckers, F.~Loeser and J.~Nicaise.
\newblock{Chai's conjecture and Fubini properties of dimensional motivic
integration.} {\em Submitted}, arXiv:1102.5653.



\bibitem[Ed92]{edix}
B.~Edixhoven.
\newblock N\'eron models and tame ramification.
\newblock {\em Compos. Math.}, 81:291--306, 1992.

\bibitem[ELL96]{ELL}
B.~Edixhoven, Q.~Liu and D. Lorenzini.
\newblock The $p$-part of
the group of components of a N\'eron model.
\newblock{\em J. Algebr. Geom.}, 5(4):801--813, 1996.

\bibitem[Fu93]{fulton}
W.~Fulton.
\newblock { Introduction to toric varieties.}
\newblock Volume 131 of {\em Annals of Mathematics Studies}.  Princeton University Press, Princeton, NJ,
1993.

\bibitem[EGA~I]{ega1}
A.~Grothendieck and J.~Dieudonn\'e.
\newblock {El\'ements de {G}\'eom\'etrie {A}lg\'ebrique, I.}
\newblock {\em Publ. Math., Inst. Hautes \'Etud. Sci.}, 4:5--228, 1960.



\bibitem[Ha10a]{Halle-stable}
L.H.~Halle.
\newblock{Stable reduction of curves and tame ramification.}
\newblock{\em Math.~Zeit.}, 265(3): 529--550, 2010.


\bibitem[Ha10b]{halle-neron}
L.H.~Halle.
\newblock{Galois actions on N\'eron models of Jacobians.}
\newblock{\em Ann.~Inst.~Fourier}, 60(3):853--903, 2010.



\bibitem[HN10]{HaNi-comp}
L.H.~Halle and J.~Nicaise.
\newblock{The N\'eron component series of an abelian variety.}
\newblock{\em Math.~Ann.}, 348(3):749--778, 2010.

\bibitem[HN11a]{HaNi}
L.H.~Halle and J.~Nicaise.
\newblock{Motivic zeta functions of abelian varieties, and the monodromy conjecture.}
\newblock{\em Adv.~Math.}, 227:610--653, 2011.

\bibitem[HN11b]{HaNi-jumps}
L.H.~Halle and J.~Nicaise.
\newblock{Jumps and monodromy of abelian varieties.}
\newblock{\em Doc.~Math.}, 16:937--968, 2011.

\bibitem[HN12]{HaNi-survey}
L.H.~Halle and J.~Nicaise.
\newblock{Motivic zeta functions for degenerations of abelian varieties and
Calabi-Yau varieties}. In: A. Campillo et al. (eds.), {\em Recent
Trends on Zeta Functions in Algebra and Geometry}. Vol.~566 of
Contemporary Mathematics, Amer. Math. Soc., Providence, RI, pages
233--259, 2012.

\bibitem[Ha77]{hartshorne}
R.~Hartshorne.
\newblock{Algebraic Geometry}. Volume 52 of {\em Graduate Texts in
Mathematics}, Springer Verlag, New York, 1977.

\bibitem[Il05]{FGA-illusie}
L.~Illusie. \newblock Grothendieck's existence theorem in formal
geometry. In: B.~Fantechi et al.~(editors), {\em Fundamental
algebraic geometry. Grothendieck's FGA explained}. Vol.~123 of
Math. Surveys Monogr., Amer. Math. Soc., Providence, RI, pages
179-–233, 2005.

\bibitem[Ka94]{kato}
K.~Kato.
\newblock{Toric singularities.}
\newblock {\em Am. J. Math.}, 116(5):1073--1099, 1994.

\bibitem[Ki10]{Kir}
F.~Kiraly.
\newblock {Wild quotient singularities of arithmetic surfaces and their regular models.}
\newblock{\em PhD Thesis}, Ulm University, 2010, available at
http://d-nb.info/1008594172/34.

\bibitem[Li69]{Lip}
J.~Lipman.
\newblock {Rational singularities, with applications to algebraic surfaces and unique factorization.}
\newblock {\em Publ. Math., Inst. Hautes \'Etudes Sci.} 36:195--279, 1969.

\bibitem[Li94]{liu-genre2}
Q.~Liu. \newblock Conducteur et discriminant minimal de courbes de
genre 2.
\newblock
{\em Compositio Math.}, 94(1):51-–79, 1994.


\bibitem[Li02]{Liubook}
Q.~Liu.
\newblock {\em Algebraic geometry and arithmetic curves}. Volume~6 of {\em
  Oxford Graduate Texts in Mathematics}.
\newblock Oxford University Press, Oxford, 2002.

\bibitem[LL01]{liu-lorenzini}
Q.~Liu and D.~Lorenzini.
\newblock {Special fibers of N\'eron models and wild ramification.}
\newblock {\em J. Reine Angew. Math.}, 532:179--222, 2001.

\bibitem[LLR04]{liu-lorenzini-raynaud}
Q.~Liu, D.~Lorenzini, and M.~Raynaud.
\newblock {N\'eron models, Lie algebras, and reduction of curves of genus one.}
\newblock {\em Invent. Math.}, 157(3):455--518, 2004.

\bibitem[LS03]{motrigid}
F.~Loeser and J.~Sebag.
\newblock {Motivic integration on smooth rigid varieties and invariants of
  degenerations}.
\newblock {\em Duke Math. J.}, 119:315--344, 2003.


\bibitem[Lo90]{Lorenzini-jac}
D.~Lorenzini.
\newblock {Groups of components of N\'eron models of Jacobians.}
\newblock {\em Compos. Math.}, 73(2):145--160, 1990.


\bibitem[Lo93]{lorenzini}
D.~Lorenzini.
\newblock{ The characteristic polynomial of a monodromy
transformation attached to a family of curves.}
\newblock{\em Comment. Math. Helvetici}, 68:111-137, 1993.



\bibitem[Lo10]{lorenzini-wild}
D.~Lorenzini.
\newblock{Models of curves and wild ramification.}
\newblock{\em Appl. Math. Q.} 6(1), Special Issue in honor of John Tate. Part 2, pages 41--82, 2010.

\bibitem[Lu10]{lu}
H.~Lu. \newblock  R\'eduction de courbes elliptiques.
\newblock{\em PhD thesis}, Bordeaux, 2010.

\bibitem[Mi80]{milne}
J.~S. Milne.
\newblock {\em \'Etale Cohomology.} Volume~33 of {\em Princeton Mathematical
  Series}.
\newblock Princeton University Press, 1980.

\bibitem[N\'e64]{AN}
A.~N\'eron.
\newblock {Mod\`eles minimaux des vari\'et\'es ab\'eliennes sur les corps locaux et globaux.}
\newblock {\em Publ. Math., Inst. Hautes \'Etudes Sci.} No. 21, 1964.

\bibitem[Ni09]{Ni-traceab}
J.~Nicaise.
\newblock {Trace formula for component groups of N\'eron models}.
\newblock {\em Preprint}, arXiv:0901.1809v2.

\bibitem[Ni11a]{Ni-tracevar}
J.~Nicaise.
\newblock {A trace formula for varieties over a discretely valued field}.
\newblock   {\em J. Reine Angew. Math.}, 650:193-238, 2011.

\bibitem[Ni11b]{Nicaise}
J.~Nicaise.
\newblock Motivic invariants of algebraic tori.
\newblock {\em Proc. Amer. Math. Soc.} 139:1163--1174, 2011.

\bibitem[Ni12]{ni-saito}
J.~Nicaise.
\newblock {Geometric criteria for tame ramification.}
\newblock To appear in {\em Math. Z.}, arXiv:0910.3812.

\bibitem[NS11]{NS-K0}
J.~Nicaise and J.~Sebag.
\newblock{The Grothendieck ring of varieties.} In: R.~Cluckers,
J.~Nicaise and J.~Sebag (editors). {\em Motivic integration and
its interactions with model theory and non-archimedean geometry.}
 Volume 383 of {\em London Mathematical Society Lecture Notes
 Series}.
Cambridge University Press, pages 145--188, 2011.

\bibitem[NS11]{nise-err}
J.~Nicaise and J.~Sebag.
\newblock{Motivic invariants of rigid varieties, and applications to complex
singularities.} In: R. Cluckers, J.~Nicaise and J.~Sebag
(editors). {\em Motivic integration and its interactions with
model theory and non-archimedean geometry.} Volume 383 of {\em
London Mathematical Society Lecture Notes Series}. Cambridge
University Press, pages 244--304, 2011.

\bibitem[Ra70]{raynaud}
M.~Raynaud.
\newblock {Sp{\'e}cialisation du foncteur de Picard.}
\newblock {\em Publ. Math., Inst. Hautes \'Etud. Sci.}, 38:27--76, 1970.

\bibitem[Sa87]{Saito}
T.~Saito.
\newblock Vanishing cycles and geometry of curves over a discrete valuation
  ring.
\newblock {\em Amer. J. Math.}, 109(6):1043--1085, 1987.

\bibitem[Sa88]{saito-cond}
T.~Saito. \newblock  Conductor, discriminant, and the Noether
formula of arithmetic surfaces. \newblock {\em Duke Math. J.},
57(1):151-–173, 1988.

\bibitem[Sa89]{saito-genus2}
T.~Saito.
\newblock  {The discriminants of curves of genus 2.}
\newblock {\em Compositio Math.}, 69, no.~2, 229-240, 1989.


\bibitem[Se68]{serre-corpslocaux}
J.-P. Serre,
\newblock {\em Corps locaux (Troisi\`eme \'edition)},
\newblock {Hermann, Paris, 1968.}

\bibitem[Si86]{silverman}
J.H.~Silverman. \newblock {\em The arithmetic of elliptic curves.}
Volume 106 of {\em Graduate Texts in Mathematics.}
Springer-Verlag, New York, 1986.

\bibitem[Si94]{silverman-advanced}
J.H.~Silverman. \newblock{\em Advanced topics in the arithmetic of
elliptic curves.} Volume 151 of {\em Graduate Texts in
Mathematics}. Springer-Verlag, New York, 1994.

\bibitem[Wi74]{winters}
G.~B. Winters.
\newblock {On the existence of certain families of curves.}
\newblock {\em Am. J. Math.}, 96:215--228, 1974.

\end{thebibliography}
\end{document}